\documentclass[12 pt]{amsart}

\usepackage{hyperref}
\usepackage{adjustbox}
\usepackage{etex}
\usepackage[shortlabels]{enumitem}
\usepackage{amsmath}
\usepackage{amsxtra}
\usepackage{amscd}
\usepackage{amsthm}
\usepackage{amsfonts}
\usepackage{amssymb}
\usepackage{eucal}
\usepackage[all]{xy}
\usepackage{graphicx}
\usepackage{tikz-cd}
\usepackage{mathrsfs}
\usepackage{subfiles}
\usepackage{mathpazo}
\usepackage[colorinlistoftodos, textsize=tiny]{todonotes}
\usepackage{morefloats}
\usepackage{pdfpages}
\usepackage{thm-restate}
\usepackage[utf8]{inputenc}
\usepackage{epigraph}
\usepackage{csquotes}
\usepackage{stmaryrd}
\usepackage{subfiles}
\usepackage[margin=1in]{geometry}
\graphicspath{ {images/} }

\RequirePackage{color}
\definecolor{myred}{rgb}{0.75,0,0}
\definecolor{mygreen}{rgb}{0,0.5,0}
\definecolor{myblue}{rgb}{0,0,0.65}

\hypersetup{citecolor=blue}
\usepackage{tikz}
\usetikzlibrary{matrix,arrows,decorations.pathmorphing}

\theoremstyle{plain}
\newtheorem{theorem}[subsection]{Theorem}
\newtheorem{proposition}[subsection]{Proposition}
\newtheorem{lemma}[subsection]{Lemma}
\newtheorem{corollary}[subsection]{Corollary}
\theoremstyle{definition}
\newtheorem{definition}[subsection]{Definition}
\newtheorem{remark}[subsection]{Remark}
\newtheorem{example}[subsection]{Example}

\theoremstyle{remark}
\newtheorem{notation}[subsection]{Notation}
\newtheorem{construction}[subsection]{Construction}
\numberwithin{equation}{section}
\newcommand\nc{\newcommand}
\nc\on{\operatorname}
\nc\renc{\renewcommand}

\newcommand\ba{\mathbb A}

\newcommand\bg{\mathbb G}

\newcommand\bp{\mathbb P}
\newcommand\bq{\mathbb Q}

\newcommand\bz{\mathbb Z}

\newcommand\scc{\mathscr C}

\newcommand\sce{\mathscr E}
\newcommand\scf{\mathscr F}
\newcommand\scg{\mathscr G}

\newcommand\sci{\mathscr I}

\newcommand\scl{\mathscr L}
\newcommand\scm{\mathscr M}
\newcommand\scn{\mathscr N}
\newcommand\sco{\mathscr O}

\DeclareMathOperator\sing{sing}
\DeclareMathOperator\cusp{cusp}
\DeclareMathOperator\node{node}

\newcommand \xra{\xrightarrow}

\DeclareMathOperator\spec{\on{Spec}}
\DeclareMathOperator\proj{\on{Proj}}

\newcommand*{\shom}{\mathscr{H}\kern -.5pt om}
\newcommand*{\stor}{\mathscr{T}\kern -.5pt or}
\newcommand*{\sext}{\mathscr{E}\kern -.5pt xt}
\newcommand*{\Ext}{Ext}

\makeatletter
\newcommand{\customlabel}[2]{\protected@write \@auxout {}{\string \newlabel {#1}{{#2}{\thepage}{#2}{#1}{}} }\hypertarget{#1}{#2}}

\newcommand\nstack[1]{\mathscr{S}^{(#1)}} 
\newcommand\cuspstack[1]{\mathscr{M}^{(#1)}_{1,\cusp}} 
\newcommand\nodestack[1]{\mathscr{M}^{(#1)}_{1,\node}} 
\newcommand\singsstack[1]{\widetilde{\mathscr{M}}^{\hspace{1mm}(#1)}_{1,\sing}} 
\newcommand\cuspsstack[1]{\widetilde{\mathscr{M}}^{\hspace{1mm}(#1)}_{1,\cusp}} 
\newcommand\nodesstack[1]{\widetilde{\mathscr{M}}^{\hspace{1mm}(#1)}_{1,\node}} 
\newcommand\onestack[1]{\mathscr{M}_1^{(#1)}} 

\newcommand\nhilb[1]{\mathscr{H}^{(#1)}} 
\newcommand\univfn[1]{\mathcal{U}^{\on{smile},(#1)}} 
\renewcommand\smile[1]{\mathcal{V}^{\on{smile}, (#1)}} 
\newcommand\smilestack[1]{\mathscr{V}^{\on{smile}, (#1)}} 
\newcommand\singshilb[1]{\widetilde{\mathscr{H}}^{\hspace{.5mm}(#1)}_{\sing}} 

\newcommand\weierstrass{\mathscr{W}} 
\newcommand\cuspw{\mathscr{W}_{\cusp}} 
\newcommand\nodew{\mathscr{W}_{\node}} 
\newcommand\singw{{\mathscr{W}}_{\sing}} 
\newcommand\singsw{\widetilde{\mathscr{W}}_{\sing}} 
\newcommand\cuspsw{\widetilde{\mathscr{W}}_{\cusp}} 
\newcommand\nodesw{\widetilde{\mathscr{W}}_{\node}} 
\newcommand\codirectrix{\xi} 

\newcommand\flagcubic{\mathcal{H}^{1,3t} } 
\newcommand\hilbw{\mathcal{H}^{\circ,1,3t}} 
\newcommand\twoflagcubic{\mathcal{H}^{1,1,3t}} 
\newcommand\hilbsw{\mathcal{H}^{\circ,1,3t}_{\sing}} 

\newcommand\universalW{\overline{\mathscr{E}}} 
\newcommand\smoothUW{\mathscr{E}} 
\newcommand\affineAut[1]{G_{#1}} 
\newcommand\unipotentRadical[1]{U_{#1}} 
\newcommand\affineAutSub[1]{G'_{#1}} 
\newcommand\projAut[1]{A_{#1}} 
\newcommand\projAutSub[1]{A'_{#1}} 
\newcommand\affineSections[1]{V_{#1}} 
\newcommand\unitResultant[1]{V^{\res \in \bg_m}_{#1}} 

\DeclareMathOperator\id{id}
\DeclareMathOperator\cl{Cl}

\DeclareMathOperator\coker{coker}

\DeclareMathOperator\pic{Pic}

\DeclareMathOperator\im{im}

\DeclareMathOperator\sym{Sym}
\DeclareMathOperator\res{Res}
\DeclareMathOperator\pgl{PGL}

\DeclareMathOperator\blow{Bl}

\newcommand\ul{\underline}
\newcommand\ol{\overline}

\DeclareMathOperator\isom{isom}

\DeclareMathOperator\aut{Aut}
\DeclareMathOperator\gl{GL}
\DeclareMathOperator\norm{Nm}

\DeclareMathOperator\sel{Sel}

\DeclareMathOperator\sm{sm}

\DeclareMathOperator\disc{disc}

\DeclareFontFamily{U}{wncy}{}
\DeclareFontShape{U}{wncy}{m}{n}{<->wncyr10}{}
\DeclareSymbolFont{mcy}{U}{wncy}{m}{n}
\DeclareMathSymbol{\Sha}{\mathord}{mcy}{"58}

\def\listtodoname{List of Todos}
\def\listoftodos{\@starttoc{tdo}\listtodoname}

\title{A geometric approach to the Cohen-Lenstra heuristics}
\author{Aaron Landesman}

\usepackage{microtype}
\begin{document}

\maketitle

\begin{abstract}
	We give a new 
	geometric description of when an element
	of the class group of a quadratic field, thought of as a quadratic
	form $q$, is $n$-torsion.
	We show that $q$ corresponds to an $n$-torsion element 
	if and only if there exists a degree $n$ polynomial
	whose resultant with $q$ is $\pm 1$.
	This is motivated by a more precise 
	geometric parameterization 
	which directly connects torsion in class groups of quadratic fields to Selmer groups
	of singular genus $1$ curves.
\end{abstract}

\section{Introduction}
\label{section:introduction}

The goal of this paper is to construct 
geometric spaces which parameterize
$n$-torsion elements in 
class groups of quadratic fields.
More generally, we prove a structure theorem in algebraic geometry
which describes a simple quotient presentation for a
stack approximately parameterizing
$\mu_n$ torsors on degree $2$ covers.
More precisely, this stack has $B$ points parameterizing $n$-coverings of generically singular
relative genus $1$ curves over $B$, as is discussed further in
\autoref{remark:selmer-relation}.

To begin, we give an easy to state consequence of our main results.
For $K$ a number field, we use $\cl(K)$ to denote the class group of $K$
and
$\cl(K)[n]$ to denote its $n$-torsion.
For the statement of the following theorem, it will be helpful to
recall the classical correspondence between 
elements of 
the class group of a quadratic field of discriminant $d$
and
primitive quadratic forms of discriminant $d$ with integer coefficients,
as is well exposited in
\cite[Theorem 1.5]{wood:compositionBase}.
\begin{theorem}
	\label{theorem:resultant-intro}
	Let $n \geq 1$ be an integer, and let $K$ be a quadratic number field of
	discriminant $d$.
	A primitive quadratic form $q := ax^2 + bxy + cy^2 \in \bz[x,y]$ of discriminant $d$
	corresponds to an element in the subgroup $\cl(K)[n]\subset \cl(K)$ 
	if and only if
	there exists a degree $n$ polynomial $\codirectrix := \sum_{i=0}^n t_i x^i y^{n-i}
	\in \bz[x,y]$
	such that the resultant of $q$ and $\codirectrix$ is either $1$ or $-1$.
\end{theorem}

We prove a stronger form of \autoref{theorem:resultant-intro} in
\autoref{theorem:resultant},
which also applies when $K$
is replaced with an order in a quadratic field.
We give a geometric proof in \autoref{subsection:idea-resultant}
and an algebraic proof in \autoref{subsection:idea-resultant-algebraic}.

However, \autoref{theorem:resultant-intro} is not completely satisfactory for
enumerating
$n$-torsion elements in class groups of quadratic fields.
In order to enumerate such elements, given a quadratic form $q$, we 
will also need a good understanding of the set of possible elements
$\codirectrix$
such that the resultant $\res(q,\codirectrix) = \pm 1$.

We now introduce notation to give the description of
these possible elements $\codirectrix$ in terms of orbits of a certain group
action over $\spec \bz$.
Let 
$\affineSections n$ denote the $3 + (n+1)$ dimensional affine space
defined in 
\autoref{definition:affine-sections}
parameterizing coefficients of
pairs of polynomials $(q,\codirectrix)$.
Let $\unitResultant n \subset \affineSections n$ denote the open subscheme
defined in \autoref{definition:unit-resultant}
parameterizing pairs $(q,\codirectrix)$ whose resultant is a unit.
Let $\affineAut n$ denote the algebraic group defined in
\autoref{definition:automorphism-groups},
which is generated by $\gl_2$ acting on the $x$ and $y$ coordinates,
$\bg_m$ diagonally scaling $q$ and $\codirectrix,$ and $\bg_a^{n-1}$ adding multiples of
$q$ to $\codirectrix$.
Also recall the definition of the $n$-Selmer group of a number field $K$, 
$\sel_n(K)$,
as defined in \autoref{remark:n-selmer-group-of-number-field}.
\begin{theorem}
	\label{theorem:z-orbit-parameterization}
	Let $K$ be a quadratic number field of discriminant $d$.
	There is a bijection from orbits $(q,\codirectrix)$ in $\unitResultant n(\bz)/
	\affineAut n(\bz)$ satisfying $\disc(q) = d$ 
	to the set quotient of $\coker \left( \sel_n(\bq) \to \sel_n(K) \right)$ by
	the action of inversion coming from the nontrivial automorphism of $K$
	over $\bq$.
\end{theorem}

\autoref{theorem:z-orbit-parameterization} follows fairly immediately from
\autoref{theorem:intro-hypercohomology-to-quotient-stack-bijection}
by combining it
with the more or less self-contained \autoref{lemma:comparison-selmer}.
In fact, the significantly more general 
parameterization 
\autoref{theorem:intro-hypercohomology-to-quotient-stack-bijection}
works over an arbitrary normal integral base scheme.
To state this more general result,
let $\Pi_n$ denote the natural map from 
$\left[ \unitResultant n/\affineAut n \right]$ to the stack of
degree $2$ finite locally free covers.
Locally, $\Pi_n$ is given by sending
$(q,\codirectrix)$ to the vanishing locus $V(q)$, viewed as a degree $2$ finite
locally free cover.
Here and throughout, we use cohomology with abelian sheaf coefficients to mean
flat cohomology.

\begin{theorem}
	\label{theorem:intro-hypercohomology-to-quotient-stack-bijection}
	Let $B$ be a
	normal integral scheme and $n \geq 3$ an integer.
	Fix a degree $2$ locally free generically \'etale cover $g: X \to B$.
	There is an injection from orbits $(q,\codirectrix) \in 
	\unitResultant n(B)/\affineAut n(B)$ such that
	$V(q) \simeq X$
	to
	$\Pi_n^{-1}([X]) \subset [\unitResultant n/\affineAut n](B)$.
	In turn, $\Pi_n^{-1}([X])$
	is identified bijectively with
	elements of $H^1(B, g_* \bg_m/\bg_m \xra{\times n} g_*
	\bg_m/\bg_m)/\aut_{X/B}(B)$.
	The above injection is a bijection if $H^1(B, \pgl_2) = H^1(B, \bg_m) =
	H^1(B, \bg_a) =0$.
\end{theorem}

We prove \autoref{theorem:intro-hypercohomology-to-quotient-stack-bijection}
and a further statement identifying corresponding stabilizers
in \autoref{theorem:hypercohomology-to-quotient-stack-bijection}.
See \autoref{section:geometric-bijection} for the geometric construction which
gives the idea behind the proof of
\autoref{theorem:intro-hypercohomology-to-quotient-stack-bijection}.
We now describe a number of salient features of our approach.
\begin{remark}
	\label{remark:selmer-relation}
	There have been many conjectures regarding Selmer groups of elliptic
	curves which are eerily similar to those governing class groups.
	We provide a seemingly new link between these two sets of
	heuristics by observing that $n$-torsion class groups can understood in
	terms of the $n$-covering group of a certain associated singular genus
	$1$ curve, constructed in
	\autoref{notation:singular-genus-1-construction}.
	The $n$-covering group of the smooth locus of this genus $1$ curve over
	$B$ is isomorphic to 
	$H^1(B, g_* \bg_m/\bg_m \xrightarrow{\times n} g_*
	\bg_m/\bg_m)$. In the case $B = \spec \bz$, this hypercohomology group is isomorphic to $\coker(\sel_n(\bq) \to
	\sel_n(K))$.
	On the other hand, using cohomological exact sequences,
	the $n$-covering group is closely related to the
	$n$-Selmer group of the $1$-dimensional algebraic group $g_* \bg_m/\bg_m$.
	This suggests a way to realize conjectures 
	regarding torsion of class groups 
	(such as those in \cite{cohen1984heuristicsShort,cohen1984heuristics})
	and $n$-Selmer groups of abelian
	varieties
	(such as those in
	\cite{poonenR:random-maximal-isotropic-subspaces-and-selmer-groups,bhargavaKLPR:modeling-the-distribution-of-ranks-selmer-groups})
	both as special cases of conjectures regarding $n$-covering groups
	of (not necessarily proper) algebraic groups.
	In particular, we believe conjectures on $n$-torsion in class groups of quadratic
	fields and $n$-Selmer groups of
	elliptic curves should be special cases of
	conjectures on $n$-covering groups of
	$1$-dimensional algebraic groups.
\end{remark}
\begin{remark}
	\label{remark:}
	\autoref{theorem:intro-hypercohomology-to-quotient-stack-bijection}
	is proven via a geometric perspective which works over an arbitrary
	base.
	Some previous results in arithmetic statistics adapting a geometric perspective include
	\cite[Theorem 2.1]{wood:compositionBase},
	\cite[Theorem 1.1 and 2.1]{wood:geometric-bh-III-quartic},
	and
	\cite[Theorem 1.4]{wood:composition-bharg-II-III-paramaterization}.
	We find this perspective helps clarify the assumptions on the base
	scheme needed to obtain the desired orbit parameterization.
	This perspective gives a natural motivation, described in
	\autoref{section:geometric-bijection},	
	for how one might come up with
	the description of the relevant moduli stack as a global quotient.
	Specifically, the affine space can be understood as that associated to a certain linear 
	system on a Hirzebruch surface, and the group we quotient by is the automorphism
	group of that Hirzebruch surface.
\end{remark}

\begin{remark}
	\label{remark:}
	The construction appearing in \autoref{section:geometric-bijection}
	generalizes to
	give a new way to understand $n$-torsion line bundles on covers of arbitrary degree.
	It seems that in several cases, structure theorems similar to
	\autoref{theorem:intro-hypercohomology-to-quotient-stack-bijection}
	should exist. For example, it seems an interesting and tractable problem
	to work out an analog of 
	\autoref{theorem:intro-hypercohomology-to-quotient-stack-bijection}
	for $n$-torsion line bundles on degree $3$ covers (in place of degree
	$2$ covers).
\end{remark}

We next discuss some connections to parameter spaces for the class groups in
	quadratic fields previously appearing in the literature.

\begin{remark}
	\label{remark:}
	Although the orbit parameterization of \autoref{theorem:intro-hypercohomology-to-quotient-stack-bijection}
	is new,
	a different parameterization of nearly the same stack in the case $n =
	3$ was given in
	\cite[Theorem 13]{bhargava:compositionLawsI}.
	See also
	\cite[Corollary 11]{bhargavaV:the-mean-number-of-3-torsion-elements}.
	There is a map from our parameterization to that in 	
	\cite[Theorem 13]{bhargava:compositionLawsI}
	given by ``taking the inflection subscheme,''
	a construction is discussed in
	\cite[Chapter 4]{landesman:thesis}.

	It appears to us that this construction has some semblance of
	the construction of ``bigger spaces which separate invariants'' in the
	literature.
	For example, there is a space $W$ appearing in
	\cite[\S2]{bhargavaSW:squarefree-values-of-polynomial-discriminants-i}
	used in the proof of \cite[Theorem
	1.5]{bhargavaSW:squarefree-values-of-polynomial-discriminants-i}.
	The relation between this ``bigger space'' $W$ and the space of squarefree
	polynomials seems similar to the relation between $\unitResultant n$ and the
	space parameterizing these inflection subschemes.
	This space parameterizing inflection subschemes turns out to be an open
	whose complement
	has codimension $2$ in $\ba^4$ when $n = 3$ but for higher $n$ is a
	variety $X_n$ of
	dimension $4$ embedded in $\ba^{n+1}$. 
	More precisely, $X_n$ is the affine cone over the complement of the
	rational normal curve in its secant variety.
	The fact that $X_n$ is not a dense open in $\ba^{n+1}$ explains why the same counting
	procedure that works when $n = 3$ does not immediately apply for higher
	$n$.

	When working over function fields and taking the Weil restriction along
	$\ba^1_{\mathbb F_q} \to \spec \mathbb F_q$, related Hurwitz stacks have
	appeared in \cite{EllenbergVW:cohenLenstra}.
	However, the Hurwitz stacks there parameterize $\bz/n\bz$ torsors over
	degree $2$ covers and so correspond to order $n$ {\em quotients} of the class
	group.
	On the other hand, the stacks appearing in this paper parameterize
	$\mu_n$ torsors over degree $2$ covers, and so yield order $n$
	{\em subgroups} of the class group (together with data relating to
	the units).
	Of course, the number of order $n$ subgroups and quotients of a finite
	abelian group have the same cardinality, but the relevant moduli stacks
	are different.
\end{remark}
We now comment on the difficulty in using this result to determine the average
size of $n$-torsion in class groups of quadratic fields.
\begin{remark}
	\label{remark:big-open}
	For this remark, we refer to an open subscheme of affine space whose complement has
	codimension $1$ as a {\em small open} and an open subscheme of affine
	space whose complement has codimension at least $2$ as a {\em big open}.
	A seemingly new aspect of the orbit parameterization given in
	\autoref{theorem:intro-hypercohomology-to-quotient-stack-bijection}
	is that it can be understood as the points of the quotient
	stack $[\unitResultant n / \affineAut n]$ where $\unitResultant n$ is a
	small open subscheme of affine space.

%
	To our knowledge, these quotient stacks of small opens
	have not previously appeared in
	arithmetic statistics research. 
	Given this, one might surmise such quotient stacks
	of small opens are a rare
	phenomenon.
	Surprisingly, we believe this phenomenon is quite ubiquitous, and have
	found similar such stacks appearing in preliminary investigations of many
	other problems.
	It seems likely that similar parameterizations naturally appear when investigating
	$n$-torsion in cubic fields, higher moments of $n$-torsion in quadratic
	fields, and $n$-torsion in fields associated to binary forms.
	However, we have not carefully verified the existence of these
	parameterizations, and believe it
	would be quite interesting and nontrivial to construct such
	parameterizations.

	We conclude this remark by pointing out a general feature which explains
	why counting points on quotients stacks of small opens by group actions
	is typically much more difficult than of big opens.
	If our small open is of the form $\ba^m - H$ for $H \subset \ba^m$ a
	hypersurface, nearly all maps $\spec \bz \to \ba^m$ will
	meet $H$ nontrivially.
	However, in a big open, $\spec \bz$ points will rarely meet the
	codimension $2$ complement, and so may typically be counted by a sieving procedure.
	In fact, in the stacks associated to small opens that come up, it is
	typically the case that $\spec \bz$ points lying in this small open
	agree with $\spec \bz$ points of a certain related hypersurface.
	This stems from the fact that $\bz$ has finitely many units.
	For example, in this paper, we investigate the small open 
	$\unitResultant n$ where the
	resultant of two polynomials is a unit. Because the only units in $\bz$
	are $\pm 1$, the $\spec \bz$ points of this open agrees with the $\spec
	\bz$ points of the hypersurface where the resultant is
	$\pm 1$.
	If one could asymptotically count points on this hypersurface, one
	could make great progress toward proving many of the conjectures of
	Cohen-Lenstra
	as in \cite{cohen1984heuristicsShort,cohen1984heuristics}.
\end{remark}

\subsection{Overview}
\label{subsection:overview}

We now give an overview of the remaining sections of the paper.
In \autoref{section:direct-proofs}, we give two direct and simple proofs of
\autoref{theorem:resultant-intro}.
In our opinion, the most important section for understanding the proof
of \autoref{theorem:hypercohomology-to-quotient-stack-bijection}
is \autoref{section:geometric-bijection},
which describes the main geometric idea for parameterizing
$n$-torsion in class groups.
Strictly speaking, the rest of the paper is independent of
\autoref{section:geometric-bijection}, but
this seems to be a more intuitive way to understand
what is going on than the actual proof.
In \autoref{section:hirzebruch-autormorphisms},
we introduce the relevant group $\affineAut n$ of automorphisms we will be quotienting by and
describe its relation to Hirzebruch surfaces.
Following this, we collect definitions of stacks
used throughout the paper in \autoref{section:stacks-definitions}
and prove basic facts about them.
Next, we give numerous equivalent characterizations
of the stack parameterizing $n$-coverings of genus $1$ curves
in \autoref{section:genus-1-curves}.
The heart of the algebro-geometric argument occurs in
\autoref{section:hirzebruch-bijection},
where we relate singular genus $1$ curves to divisors on
Hirzebruch surfaces.
We then explain the connection between singular genus $1$
curves and degree $2$ covers in \autoref{section:degree-two},
which will enable us to connect the preceding discussion 
regarding genus $1$ curves
to quadratic field extensions.
Using the analysis thus far, we deduce
\autoref{theorem:hypercohomology-to-quotient-stack-bijection},
and hence \autoref{theorem:intro-hypercohomology-to-quotient-stack-bijection},
in
\autoref{section:unit-resultant}.
In \autoref{section:cohomology}, when the base $B = \spec \bz$,
we connect the relevant $n$-covering groups to $n$-Selmer groups of number
fields to finish the proof of \autoref{theorem:z-orbit-parameterization}.
We conclude \autoref{section:cohomology} with several examples, starting in \autoref{subsection:examples}.

See \autoref{figure:proof-schematic} for a schematic depiction of how the proof
\autoref{theorem:hypercohomology-to-quotient-stack-bijection} fits together.

At various points the reader may find it useful to consult
\cite{landesman:thesis}, which spells
out some of the arguments of this paper in more detail.

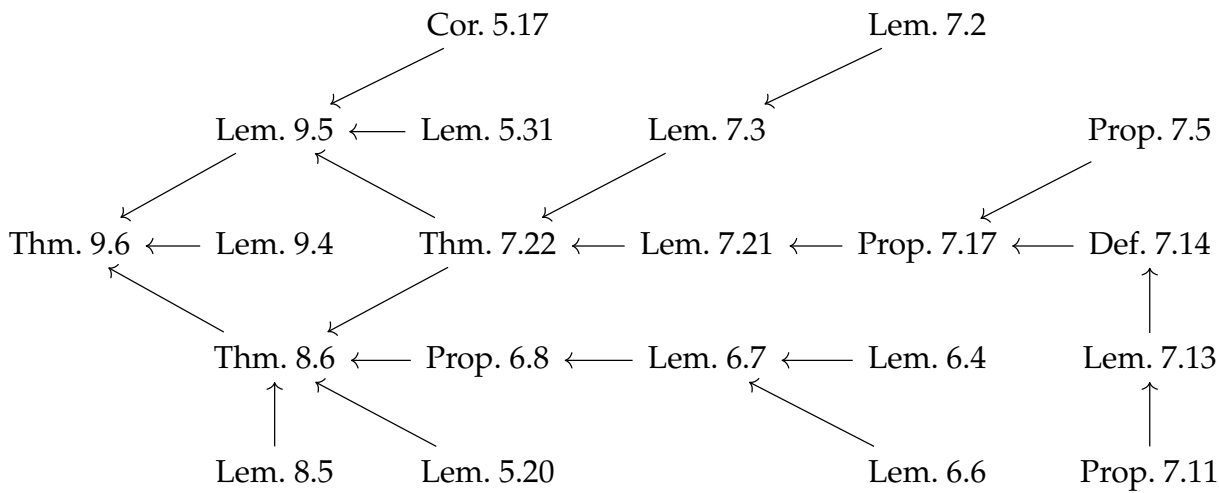
\begin{figure}
	\centering
\begin{equation}
  \nonumber
   \begin{tikzcd}[column sep = 1.7em]
	\qquad &&
  \text{Cor.}~\ref{corollary:singsstack-representable} \ar{dl}&
  &
  \text{Lem.}~\ref{lemma:locally-free-e-n-2-f} \ar{dl} & \\
  \qquad &
  \text{Lem.}~\ref{lemma:unit-resultant-equivalence} \ar{dl} &
  \text{Lem.}~\ref{lemma:smile-stack-representable} \ar{l}&
  \text{Lem.}~\ref{lemma:forward-equivalence} \ar{dl}&&
  \text{Prop.}~\ref{proposition:genus-0-blow-up}\ar{dl}\\
  \text{Thm.}~\ref{theorem:hypercohomology-to-quotient-stack-bijection}&
  \text{Lem.}~\ref{lemma:hirzebruch-quotient-commutes-with-z-points} \ar{l}&
  \text{Thm.}~\ref{theorem:equivalence} \ar{dl} \ar{ul}&
  \text{Lem.}~\ref{lemma:reverse-equivalence} \ar{l}&
  \text{Prop.}~\ref{proposition:f-is-hirzebruch-twist} \ar{l}&
  \text{Def.}~\ref{notation:f-surface} \ar{l}\\
  \qquad & \text{Thm.}~\ref{theorem:fiber-bijection} \ar{ul}&
  \text{Prop.}~\ref{proposition:equivalent-fiber-of-stack} \ar{l}&
   \text{Lem.}~\ref{lemma:fiber-of-stack-pre-quotient} \ar{l}  &
  \text{Lem.}~\ref{lemma:sections-of-quotient-and-torsors} \ar{l}&
  \text{Lem.}~\ref{lemma:map-to-grassmannian} \ar{u}\\
  \qquad & \text{Lem.}~\ref{lemma:smooth-locus-of-curve-from-cover} \ar{u}&
  \text{Lem.}~\ref{lemma:lift} \ar{ul}& &
  \text{Lem.}~\ref{lemma:compactification-scheme} \ar{ul} &
  \text{Prop.}~\ref{proposition:w-flat} \ar{u}
	  \end{tikzcd}
\end{equation}
\caption{
A schematic diagram depicting the structure of the proof of
\autoref{theorem:hypercohomology-to-quotient-stack-bijection}, a slightly
stronger form of
\autoref{theorem:intro-hypercohomology-to-quotient-stack-bijection}.}
\label{figure:proof-schematic}
\end{figure}

\subsection{Notation}
\label{section:notation}
We collect some notation used throughout the paper.

\subsubsection{Notation for Hirzebruch surfaces}
	\label{notation:hirzebruch}
For $n \geq 3$, let $\mathbb F_{n-2}$ denote the Hirzebruch surface $\bp_{\bp^1_\bz} \left( \sco_{\bp^1_\bz} \oplus \sco_{\bp^1_\bz}(n-2) \right)$ over $\spec \bz$.
By construction, we have a factorization $\mathbb F_{n-2} \xra{g} \bp^1_\bz \xra{h} \spec \bz$.
More generally, for $B$ a base scheme, let $(\mathbb F_{n-2})_B$ denote 
$\proj_{\bp^1_B}(\sco_{\bp^1_B} \oplus \sco_{\bp^1_B}(n-2))$.
When the base $B$ is clear from context,
we denote $(\mathbb F_{n-2})_B$ simply by $\mathbb F_{n-2}$ and
denote the projections
$(\mathbb F_{n-2})_B \xra{g} \bp^1_B \xra{h} B$
by $g$ and $h$.
Consider the surjection $\sco_{\bp^1_B} \oplus \sco_{\bp^1_B}(n-2) \to \sco_{\bp^1_B}$ of sheaves on $\bp^1_B$. This surjection is unique up to scaling
on $B$, and so defines a distinguished divisor $E \subset (\mathbb F_{n-2})_B$, whose class we denote by $e$.
We call this divisor the {\em directrix}.
We also refer to a section $\bp^1 \to (\mathbb F_{n-2})_B$ not meeting $E$
as a {\em codirectrix}.
If $\sigma: B \to \bp^1_B$ is a section of $h$, we use $f$ denote the class of
the ``fiber'' divisor $g^{-1}(\sigma)$.
See \cite[Proposition IV.1]{beauville:complex-algebraic-surfaces}
for general background on the Picard group of Hirzebruch surfaces.

\subsubsection{Table of notation}
For the reader's convenience, in \autoref{table:notation}
we collect notation introduced throughout the paper,
roughly in order of appearance.
The descriptions are intended to be terse and not precise.

\begin{figure}
	\centering
\scalebox{0.7}{
	\begin{tabular}{|l|l|l|}
		\hline
	Notation & Description & Location defined \\\hline
$\weierstrass$ & Stack of Weierstrass curves & \autoref{definition:weierstrass}
\\ \hline
$\universalW$ & Universal curve over $\weierstrass$ &
\autoref{definition:smooth-universal-genus-1} 
\\ \hline
$\smoothUW$ & Smooth locus of $\universalW$ over $\weierstrass$ &
\autoref{definition:smooth-universal-genus-1}
\\ \hline
$\cuspw$ & Substack of cuspidal curves in $\weierstrass$ &
\autoref{definition:node-and-cusp-loci}
\\ \hline
$\nodew$ & Substack of nodal curves in $\weierstrass$ & 
\autoref{definition:node-and-cusp-loci}
\\ \hline
$\singw$ & Substack of singular curves in $\weierstrass$ &
\autoref{definition:singsw}
\\ \hline
$\singsw$ & Stack of Weierstrass curves with a marked section in the singular
locus & \autoref{definition:singsw}
\\ \hline
$\cuspsw$ & Substack of $\singsw$ parameterizing cuspidal curves &
\autoref{definition:cuspsw}
\\ \hline
$\nodesw$ & Substack of $\singsw$ parameterizing nodal curves &
\autoref{definition:cuspsw}
\\ \hline
$\nhilb n$ & Hilbert scheme of geometrically
integral degree $n$ genus $1$ curves in $\bp^{n-1}$ & \autoref{definition:nhilb}
\\ \hline
$\singshilb n$ & Curves in $\nhilb n$ with a marked
singular point & \autoref{definition:singshilb-and-singsstack}
\\ \hline
$\nstack n$ & Stack $n$-coverings for Weierstrass curves &
\autoref{definition:nstack}
\\ \hline
$\onestack n$ & Stack of genus 1 degree $n$ curves &
\autoref{definition:nhilb}
\\ \hline
$\nodestack n$ & Substack of $\onestack n$ parameterizing nodal curves
& \autoref{definition:node-and-cusp-loci}
\\ \hline
$\cuspstack n$ & Substack of $\onestack n$ parameterizing cuspidal 
curves & \autoref{definition:node-and-cusp-loci}
\\ \hline
$\singsstack n$ & Stack of genus $1$ degree $n$ curves with a 
section in the singular locus & \autoref{definition:singsstack} 
\\ \hline
$\nodesstack n$ & Substack of $\singsstack n$ parameterizing nodal curves &
\autoref{definition:nodesstack-and-cuspsstack}
\\ \hline
$\cuspsstack n$ & Substack of $\singsstack n$ parameterizing cuspidal curves &
\autoref{definition:nodesstack-and-cuspsstack}
\\ \hline
$\smile n$ & Scheme of smooth curves in the linear system $e + nf$ on
$\mathbb F_{n-2}$ & \autoref{definition:smile-scheme}
\\ \hline
$\smilestack n$ & Stack of smooth curves 
in the linear system $e + nf$ on $(n-2)$-Hirzebruch twists &
\autoref{definition:smile-stack}
\\ \hline
$\affineSections n$ & The affine space associated to $H^0(\bp^1, \sco(2) \oplus \sco(n))$ &
\autoref{definition:affine-sections}
\\ \hline
$E_g$ & The singular genus $1$ curve associated to a degree $2$ finite locally
free map $g$
&\autoref{notation:singular-genus-1-construction}
\\ \hline
$\unitResultant n$ & Open locus of $\affineSections n$ for which the resultant
is a unit
& \autoref{definition:unit-resultant}
\\ \hline
$\affineAut n$ & Group of automorphisms acting on $\affineSections n$ &
\autoref{definition:automorphism-groups}
\\ \hline
$\affineAutSub n$ & Subgroup of $\affineAut n$ ``fixing the
base $\mathbb P^1$''
&\autoref{definition:automorphism-groups}
\\ \hline
$\unipotentRadical n$ & Unipotent radical of $\affineAut n$
&\autoref{definition:automorphism-groups}
\\ \hline
$\projAut n$ & Projective quotient of $\affineAut n$
&\autoref{definition:automorphism-groups}
\\ \hline
$\projAutSub n$ & Projective quotient of $\affineAutSub n$
&\autoref{definition:automorphism-groups}
\\ \hline
	\end{tabular}
}
	\vspace{.5cm}
	\caption{Notation introduced in the paper}
	\label{table:notation}
\end{figure}

\subsection{Acknowledgements}
I would like to thank Anand Patel for meeting weekly for many months
to discuss this project.
He suggested numerous ideas appearing here.
I thank
Tony Feng and Eric Rains for stimulating
discussions leading to much of the material in
\autoref{section:genus-1-curves}.
I thank an anonymous referee for helpful suggestions.
I thank 
Levent Alpoge, 
Brian Conrad, 
Sean Cotner,
Jordan Ellenberg, 
Jean Kieffer,
Nikolas Kuhn, 
Jef Laga,
Arpon Raksit,
Arul Shankar, 
Alex Smith,
Ashvin Swaminathan,
Ravi Vakil,
Martin Widmer,
Melanie Wood,
Bogdan Zavyalov, 
and 
David Zureick-Brown for further helpful conversations and comments.
This material is based upon work supported by the National Science Foundation
Graduate Research Fellowship Program under Grant No. DGE-1656518.

\section{Direct Proofs of \autoref{theorem:resultant-intro}}
\label{section:direct-proofs}

We now state the generalization of \autoref{theorem:resultant-intro} to orders
and give two simple proofs.
One may give a third proof using \autoref{theorem:hypercohomology-to-quotient-stack-bijection}, as we do
in \cite[\S3.3.14]{landesman:thesis}.

\begin{theorem}
	\label{theorem:resultant}
	Let $n \geq 1$ be an integer and fix an integral degree $2$ free
	$\bz$-algebra $R$ of discriminant $d$.
	A primitive quadratic form $q := ax^2 + bxy + cy^2 \in \bz[x,y]$ of discriminant $d$
	corresponds to an element in the subgroup $\cl(R)[n] \subset \cl(R)$ 
	if and only if
	there exists a polynomial $\codirectrix := \sum_{i=0}^n t_i x^i y^{n-i}
	\in \bz[x,y]$
	such that the resultant of $q$ and $\codirectrix$ is either $1$ or $-1$.
\end{theorem}

\subsection{Geometric proof of \autoref{theorem:resultant}}
\label{subsection:idea-resultant}

\begin{figure}
	\centering
	\includegraphics[scale=.6, trim={0 21.5cm 3.5cm 0cm}]{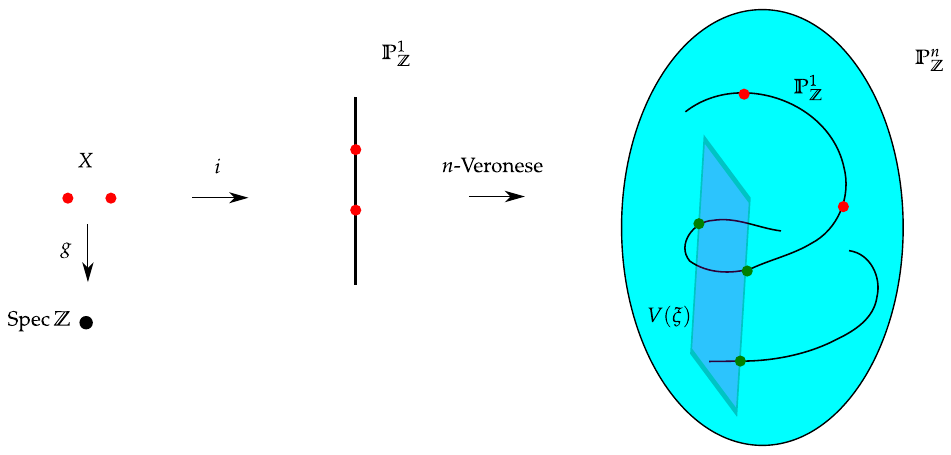}
	\caption{A visualization of why the existence of $\codirectrix$ with $\res(q,\codirectrix)
		= \pm
		1$ forces $q$ to be $n$-torsion.}
	\label{figure:rational-normal-curve}
\end{figure}

Let $R$ and $q$ be as in \autoref{theorem:resultant}.
We wish to show $q$ is $n$-torsion
if and only if there exists $\codirectrix \in \bz[x,y]$ 
with $\res(q,\codirectrix) = \pm 1$.
We accomplish this by the following geometric construction,
which is visualized in
\autoref{figure:rational-normal-curve}.
Let $X := \spec R$, so that $X \simeq \proj \bz[x,y]/(q)$.
Then, $q$ determines an embedding $i: X \to \bp^1_\bz$
corresponding to an invertible
sheaf $\scl_q:= i^* \sco_{\bp^1_\bz}(1)$ on $X$.
See the proof of \cite[Theorem 1.4]{wood:compositionBase} in
\cite[\S3]{wood:compositionBase} for further description of how this bijection works.
An isomorphism $\phi \colon \sco_X \simeq \scl_q^{\otimes n}$
corresponds to a section $s \in H^0(X, \scl_q^{\otimes n})$
vanishing nowhere on $X$.
The restriction map $H^0(\bp^1_\bz, \sco_{\bp^1_\bz}(n)) \to H^0(X,
\scl_q^{\otimes n})$ is surjective because the cokernel injects into
$H^1(\bp^1_\bz, \sco_{\bp^1_\bz}(n) \otimes \sco_{\bp^1_\bz}(X)^\vee) \simeq
H^1(\bp^1_\bz, \sco_{\bp^1_\bz}(n-2)) = 0$.
Therefore, such an $s$ as above exists if and only if there exists
$\codirectrix \in H^0(\bp^1_\bz, \sco_{\bp^1_\bz}(n))$
restricting to $s$.
Therefore, the existence of $\phi$ and $s$ is equivalent to the existence of a section $\codirectrix$
so that $V(\codirectrix)$ does not meet $X = V(q)$. 
The condition that $V(\codirectrix)$ does not meet $V(q)$ can be rephrased as
$\res(q,\codirectrix) = \pm 1$.
\qed

\subsection{Algebraic proof of \autoref{theorem:resultant}}
\label{subsection:idea-resultant-algebraic}

Let $I_q$ denote the ideal class corresponding to $q$.
Using the standard correspondence between equivalence classes of quadratic forms
and ideal classes in quadratic algebras,
we can write $q = \frac{\norm_{R/\bz}(-\beta x + \alpha y)}{\norm_{R/\bz}( \langle
\alpha, \beta \rangle)}$, 
for $\alpha,\beta\in\sco_K$
and $I_q = \langle \alpha,\beta \rangle$.
For any given $\codirectrix \in \bz[x,y]$ homogeneous of degree $n$,
we will show $\res(q,\codirectrix)= \pm 1$ if and only if $I_q^n =
(\codirectrix(\alpha, \beta))$.
This will imply the theorem because $I_q^n = \langle \alpha^n,
\alpha^{n-1}\beta, \ldots, \beta^n \rangle$, and so if
$I_q^n$ is principal, it must be generated by an element of the form
$I_q^n = (\codirectrix(\alpha,\beta))$
for some degree
$n$ homogeneous $\codirectrix \in \bz[x,y]$.

By multiplicativity of the resultant,
\begin{align*}
	\res( \norm_{R/\bz}(-\beta x + \alpha y), \codirectrix) 
	=\res( \norm_{R/\bz}( \langle\alpha, \beta \rangle)q, \codirectrix)
	= \norm_{R/\bz}(
	\langle \alpha, \beta \rangle)^{n}\res(q,\codirectrix).
\end{align*}
Let $\sigma$ denote the unique nontrivial automorphism of $R$ over $\bz$.
Using basic properties of the resultant, such as
\cite[Proposition 8.3]{Lang2002},
\begin{align*}
	\res( \norm_{R/\bz}(-\beta x + \alpha y), \codirectrix) = 
\codirectrix(\alpha, \beta) \cdot \sigma(\codirectrix(\alpha, \beta)) = 
\norm_{R/\bz}\left( \codirectrix(\alpha, \beta) \right).
\end{align*}
Hence,
$\norm_{R/\bz}\left( \codirectrix(\alpha, \beta) \right) =
\norm_{R/\bz}(\langle \alpha, \beta \rangle)^{n}\res(q,\codirectrix)$.
Since we always have $\codirectrix(\alpha,\beta) \in \langle \alpha,\beta \rangle^n$,
the two ideals $\langle \alpha,\beta \rangle^n$ and
$(\codirectrix(\alpha,\beta))$ are equal
if and only if 
$\res(q,\codirectrix)=\pm 1$.
\qed

\section{The geometric bijection}
\label{section:geometric-bijection}

In \autoref{subsection:idea-resultant}, we have already given a direct proof of
\autoref{theorem:resultant}.
However, for the purposes of proving
\autoref{theorem:hypercohomology-to-quotient-stack-bijection},
we need a more precise bijection, established in
\autoref{theorem:hypercohomology-to-quotient-stack-bijection},
which tells us exactly when two pairs
$(q,\codirectrix)$ correspond to the same element of the class group.
In this section, we will describe a geometric construction
to explain when two
pairs $(q,\codirectrix)$ and $(q,\codirectrix')$ 
as in \autoref{theorem:resultant}
correspond to the same element
of the class group.
This construction takes as input 
a degree $2$ finite free cover $g: X \to \spec \bz$, 
a line bundle $\scl$ on $X$,
and an isomorphism
$\iota: \scl^{\otimes n} \simeq \sco_X$.
The pair ($g: X \to \spec \bz, \scl)$ is determined by $q$ while
$\iota$ is determined by the additional data of $\codirectrix$.
The construction outputs a particular divisor on a Hirzebruch surface, up to isomorphism.
Whether $(q,\codirectrix)$ and $(q,\codirectrix')$ correspond to the same element of the class group is closely tied
to whether they are related by the action of a certain group $\affineAut n$.
The group $\affineAut n$ can be understood as a certain linearization of the
automorphism group of the
above mentioned Hirzebruch surface.

\subsection{The construction on fibers}
\label{subsection:idea-on-fibers}

\begin{figure}
	\centering
	\includegraphics[scale=.6, trim={0 13.5cm 0 0}]{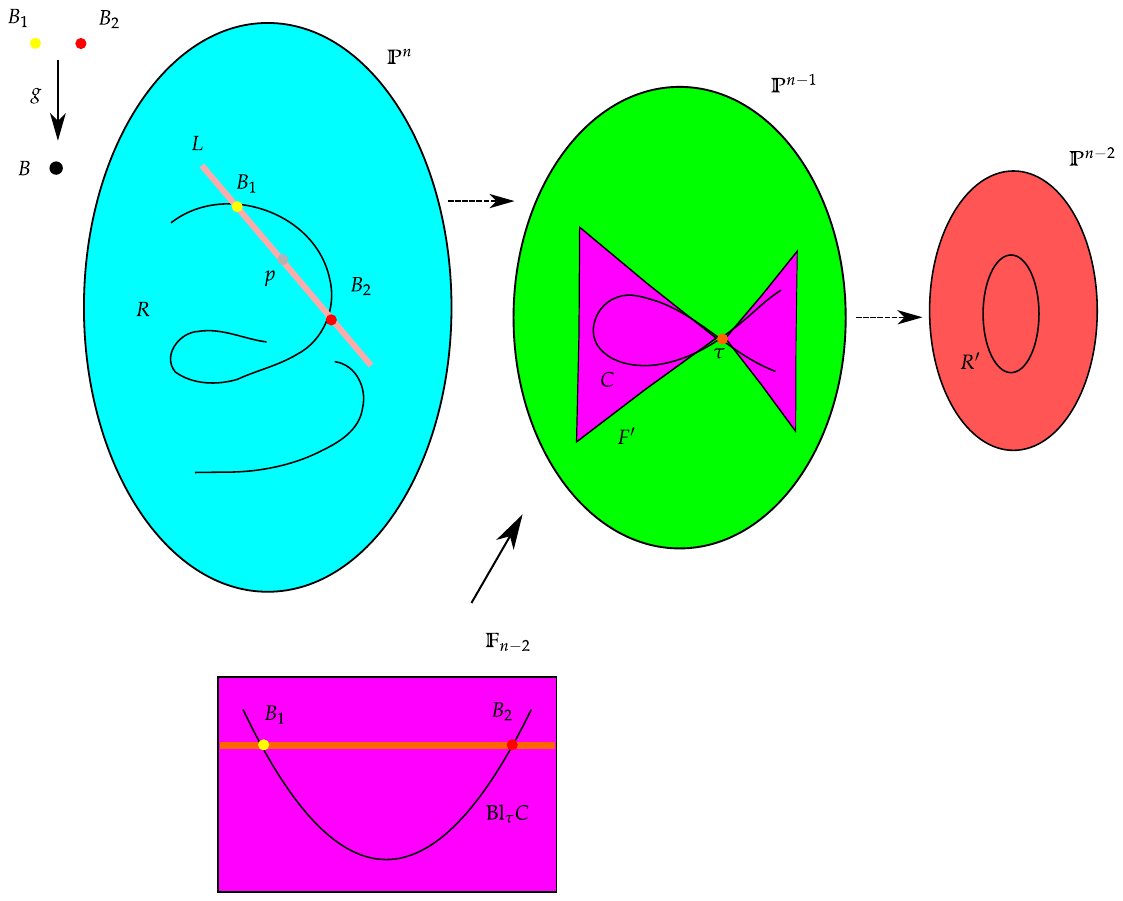}
	\caption{A visualization of the bijection of
	\autoref{subsection:idea-on-fibers}. 
}
	\label{figure:idea-on-fibers}
\end{figure}

For the rest of this section, we assume $B = \spec k$ is a field and $g: X \to
B$ is a double cover with 
$X = B_1 \coprod B_2$ so that $B_1 \simeq B_2 \simeq B$.
Much of what follows easily generalizes to the case $B$ is arbitrary and $X$ is
any degree $2$ finite locally free cover, but we work in the above case to simplify
the exposition.

We now give the geometric construction relating
$(g: X \to \spec B, \scl, \iota: \scl^{\otimes n} \simeq \sco_X)$
to certain smooth sections on a Hirzebruch surface.
\autoref{figure:idea-on-fibers} may be helpful in visualizing the geometric
construction described in this subsection.

To start, we set up notation.
Let $\scl$ be an invertible sheaf on $X$,
(which, in this degenerate case must be isomorphic to $\sco_X$,)
and suppose we are given an isomorphism $\iota: \scl^{\otimes n} \simeq \sco_X$.
Recalling the notation for Hirzebruch surfaces from
\autoref{notation:hirzebruch},
we describe how to obtain a section of class $e + nf$ in the Hirzebruch surface $\mathbb F_{n-2}$ over $B$.

We now construct the desired section of class $e + nf$ in a series of steps.
With reference to \autoref{figure:idea-on-fibers}, 
the line bundle $\scl^{\otimes n}$ gives the map $X \to \bp^n$. 
Saying this more precisely, the sheaf $g_* \scl$ corresponds to a rank $2$ vector space $V$ over $k$.
There is a natural $n$-Veronese embedding $\bp V \to \bp \sym^n V \simeq \bp^{n}$ realizing
$\bp V$ as a rational normal curve $R$ in $\bp^{n}$.

Next, we use the
trivialization $\iota: \scl^{\otimes n} \simeq \sco_X$ to construct the line $L$
containing the image $X \to \bp^n$.
The surjection $\sym^n V  =\sym^n (g_* \scl) \to g_* (\scl^{\otimes n}) \xra{g_* \iota} g_* \sco_X$
gives a line $L$ in $\bp^{n}$. Note that $L \cap R$ consists of two points
corresponding to the two further surjections $g_* \sco_X \to g_* \sco_{B_i} \simeq
\sco_B$ associated to the inclusions $B_i \to X$ for $i \in \{1, 2\}$.

Having constructed the line $L$, we now use the structure map $\sco_B \to g_* \sco_X$ to
obtain a point $p$ on $L$ missing $X$.
Let $Q$ denote the cokernel of $\sco_B \to g_* \sco_X.$
We obtain a composite $\sym^n V \to g_* \sco_X \to Q$ which corresponds to a point $p$ on $L$.
We claim this point $p$ is not one of the two intersection points $L \cap R$.
Indeed, the two intersection points
with $R$ correspond to two idempotent basis vectors $e_1$ and $e_2$ associated
to the inclusions $B_i \to X$, while the
point $p$ corresponds
to the diagonal inclusion $k \to ke_1 \oplus ke_2\simeq W$ sending $1 \mapsto e_1 + e_2$.

We next explain why projecting $R$ from $p$ yields a curve $C$ in $\bp^{n-1}$ lying in a cone
over a rational normal curve $R' \subset \bp^{n-2}$.
Project $R$ from the point $p$ to obtain a singular genus $1$ curve $C$
in $\bp^{n-1}$ from $R$ by gluing
the two points $B_1$ and $B_2$ of $L \cap R$.
Let $\tau$ denote the singular point of $C$.
We claim that in fact $C$ lies in the cone over a rational normal curve in $\bp^{n-2}$.
To see this, note that further projecting $C$ from $\tau$ is equivalent to projecting the original curve $R$ from the
line $L$.
Since $L$ meets $R$ in two points, this projection is a rational normal curve
$R' \subset \bp^{n-2}$.
Therefore, the projection of $C$ from $\tau$ is a rational normal curve, and so $C$ lies in the cone $F'$ over $R'$ passing
through $\tau$.

Finally, we blow up $C$ and $F'$ at $\tau$ to obtain the desired divisor on a Hirzebruch
surface of class $e + nf$.
When we blow $F'$ up at $\tau$, we will obtain a Hirzebruch surface isomorphic to $\mathbb F_{n-2}$.
One can also verify that the blow up of $C$ at $\tau$ is then a smooth curve in the linear system $e + n f$ on $\mathbb F_{n-2}$.
This is the completes our construction.
See \cite[\S1.5.2]{landesman:thesis}
for a detailed generalization of the above construction to the case $B =
\spec \bz$.

\section{The automorphism group scheme of a Hirzebruch surface}
\label{section:hirzebruch-autormorphisms}

In this section, we will describe a certain homogeneous space $\affineSections
n$ for a group action $\affineAut n$ which will enable us to count $n$-torsion
elements of class groups.
In \autoref{lemma:affine-aut-sequences}, we identify the quotient of $\affineAut n$ by a
central copy of $\bg_m$ with the automorphisms of $\mathbb F_{n-2}$.

To describe $\affineAut n$ directly, it is easiest to describe it in terms of its
left action on the rank $n+4$ free $\bz$-module
$H^0(\bp^1_\bz, \sco_{\bp^1_\bz}(2)\oplus
\sco_{\bp^1_\bz}(n))$.
This module will be one of the central objects of this paper, and so we give the corresponding scheme a name.
\begin{definition}
	\label{definition:affine-sections}
	For $n \geq 3$, define $\affineSections n := \spec \left( \sym^\bullet H^0(\bp^1_\bz, \sco_{\bp^1_\bz}(2)\oplus
	\sco_{\bp^1_\bz}(n)) \right).$
\end{definition}

Choosing a basis $\bz x \oplus \bz y$ for 
$H^0(\bp^1_\bz,
\sco_{\bp^1_\bz}(1))$,
we can identify points of $\affineSections n$ with pairs 
$(q,\codirectrix) := (\sum_{i=0}^2 a_i x^i y^{2-j}, \sum_{j=0}^n b_j x^j y^{n-j})$ for $a_i \in \bz, b_j \in \bz$ with $0 \leq i \leq 2, 0 \leq j \leq n$.
We can now realize actions of $\bg_m, \bg_a^{n-1}$ and $\gl_2/\mu_{n-2}$ on
$\affineSections n$ as those induced by
\begin{equation}
\begin{aligned}
	\label{equation:gm-action}
	\bg_m \times \affineSections n & \rightarrow \affineSections n\\
	\left( \chi , (q,\codirectrix) \right) & \mapsto
	(\chi q,
	\chi \codirectrix), \\
\end{aligned}
\vspace{-.3cm}
\end{equation}
\begin{equation}
\begin{aligned}
       \label{equation:ga-action}
       \bg_a^{n-1} \times \affineSections n & \rightarrow \affineSections n\\
       ((\alpha_0, \ldots, \alpha_{n-2}) , (q,\codirectrix)) & \mapsto (q,
       \codirectrix + \sum_{i=0}^{n-2} \alpha_i x^i y^{n-2-i}q), \\
\end{aligned}
\vspace{-.3cm}
\end{equation}
\begin{equation}
	\resizebox{1.0\hsize}{!}{$
\begin{aligned}
	\label{equation:gl2-action}
		\gl_2 \times \affineSections n &\rightarrow \affineSections n\\
		 \left( \begin{pmatrix}
			a & b \\
			c &d
		\end{pmatrix}, \left (\sum_{i=0}^2 a_i x^i y^{2-j}, \sum_{j=0}^n b_j x^j
	y^{n-j} \right) \right) &\mapsto 
\frac{1}{ad-bc}
\left(\sum_{i=0}^2 a_i (ax+by)^i (cx+dy)^{2-j}, \sum_{j=0}^n b_j (ax+by)^j
(cx+dy )^{n-j} \right).
\end{aligned}
$}
\end{equation}

\begin{definition}
	\label{definition:automorphism-groups}
	Define $\affineAut n$ as the subgroup of 
	$\gl\left( \affineSections n \right)$ generated the subgroups
	$\bg_m$, $\bg_a^{n-1}$, and $\gl_2/\mu_{n-2}$ induced by the actions defined in \eqref{equation:gm-action},
	\eqref{equation:ga-action},
	and \eqref{equation:gl2-action},
respectively.
Define $\affineAutSub n \subset \affineAut n$ as the subgroup generated by
$\bg_m$, $\bg_a^{n-1}$, and the central $\bg_m$ sitting inside $\gl_2/\mu_{n-2}$ under the actions
\eqref{equation:gm-action}, \eqref{equation:ga-action}, and
\eqref{equation:gl2-action}.
Define $\unipotentRadical n \subset \affineAut n$ as the subgroup isomorphic to
$\bg_a^{n-1} \subset \affineAut n$ coming from \eqref{equation:ga-action}.

Let $\projAut n$ denote the image of $\affineAut n$ under the map to
$\gl(\affineSections n) \to \pgl(\affineSections n)$
and
$\projAutSub n$ denote the image of 
$\affineAutSub n$ under the map
$\gl(\affineSections n) \to \pgl(\affineSections n)$.
\end{definition}

So far, it is not clear whether the group $\affineAut n$, which is generated by
$\unipotentRadical n, \bg_m$, and $\gl_2/\mu_{n-2}$, contains elements 
which are not products of elements of these three subgroups.
The following lemma establishes that that all element of
$\affineAut n$ are products
of elements from these three subgroups
and also relates $\affineAut n$ to the automorphism group of a Hirzebruch
surface.

\begin{lemma}
	\label{lemma:affine-aut-sequences}
	For any $n \geq 3$, 
	\begin{enumerate}
		\item	$\affineAutSub n \simeq \unipotentRadical n \rtimes
			\bg_m^2$,
\item we have an exact sequence
	\begin{equation}
		\label{equation:affine-aut-exact-sequence}
		\begin{tikzcd}
			0 \ar {r} & \affineAutSub n  \ar {r} & \affineAut n \ar
			{r} & \pgl_2 \ar {r}
			& 0 
	\end{tikzcd}\end{equation}
	where the induced map $\gl_2/\mu_{n-2} \to \affineAut n \to \pgl_2$ is the
	quotient of $\gl_2/\mu_{n-2}$ by its central $\bg_m \subset
	\gl_2/\mu_{n-2}$,
\item	$\affineAut n \simeq \unipotentRadical n \rtimes \left( \bg_m \times
	(\gl_2/\mu_{n-2}) \right)$,
\item $\projAut n \simeq \unipotentRadical n \rtimes \gl_2/\mu_{n-2}$ and is
	identified with the automorphism group scheme $\aut_{\mathbb
	F_{n-2}/\bz}$.
	\end{enumerate}
\end{lemma}
\begin{proof}
	We first prove $(1)$.
	This will follow from the direct calculation that for any ring $R$ and any
	$g, g' \in \unipotentRadical n(R)$ and $h, h' \in \bg_m^2(R)$
	we have $h^{-1}gh \in \unipotentRadical n(R)$.
	To check this identity, we may work on a flat cover of $R$, and hence
	assume the $R$ points $h, h' \in \bg_m^2(R) \subset (\bg_m \times
\gl_2/\mu_{n-2})(R)$ lift to
	$R$ points of $\bg_m \times \gl_2$.
	Let $h$ be given by $(\chi, \zeta) \in \bg_m^2(R) \subset (\bg_m \times
\gl_2)(R)$
	and let $g$ correspond to 
	a tuple $(\alpha_0, \ldots, \alpha_{n-2})$ as in
	\eqref{equation:ga-action}. Define $\alpha := \sum_{i=1}^{n-1} \alpha_i
	y^i x^{n-2-i}$.
	Then, for $(q,\codirectrix) \in \affineSections n$, we have
	\begin{align*}
		h^{-1}gh(q,\codirectrix) = h^{-1} g(\chi q, \chi \zeta^{n-2} \codirectrix)
		= h^{-1}(\chi q, \chi\zeta^{n-2} \codirectrix + \alpha \chi
		q) = (q,
			\codirectrix
		+ \alpha \zeta^{2-n} q).
	\end{align*}
	Therefore, $h^{-1}gh = \zeta^{2-n}\cdot g\in \unipotentRadical n(R)$.

	The above calculation implies that we may find $h'' \in \bg_m^2(R)$ so that
	$(gh)(g'h') = (gg')(h'' h')$.
	Therefore every element in $\affineAutSub n$ is a product of an
	element of $\unipotentRadical n$ and an element of $\bg_m^2$.
	This shows $\affineAutSub n$ is an extension of $\bg_m^2$ by
	$\unipotentRadical n$,
	and it is in fact a semidirect product because $\bg_m^2$
	embeds in $\affineAutSub n$ by construction.

	Next, we check $(2)$. As a first step, we verify $\affineAutSub n \subset \affineAut n$ is a normal
	subgroup and every element of $\affineAut n$ can be written as a product
	of an element of $\affineAutSub n$ and an element of $\gl_2/\mu_{n-2}$.
	Analogously to our computation for $(1)$,
	it is enough to show that for any ring $R$ and any $h \in
	\affineAutSub n (R), g \in \gl_2/\mu_{n-2}(R)$, we have $ghg^{-1} \in \affineAutSub n
	(R)$.
	Again, to check this identity, we may work on a flat cover of $R$ so
	as to assume $g$ lifts to a point of $\gl_2$.
	By construction of $\affineAutSub n$, $h$ acts on the quotient
	$H^0(\bp^1_\bz, \sco_{\bp^1_\bz}(2))$ of
	$\affineSections n =
	H^0(\bp^1_\bz, \sco_{\bp^1_\bz}(2)\oplus
\sco_{\bp^1_\bz}(n))$
	only by scaling via the central copy of $\bg_m \subset \gl\left(
	H^0(\bp^1_\bz, \sco_{\bp^1_\bz}(2))\right)$.
	Via a direct calculation, the subgroup of $\gl_2$
	acting via this central $\bg_m$ on 
	$H^0(\bp^1_\bz, \sco_{\bp^1_\bz}(2))$
	is precisely the central $\bg_m \subset \gl_2$, which already 
	factors through
	$\affineAutSub n \subset \affineAut n$.
	Therefore, $\affineAutSub n$ is characterized as the subgroup of
	$\affineAut n$ whose action on 
	$H^0(\bp^1_\bz, \sco_{\bp^1_\bz}(2))$ factors through the central copy of
	$\bg_m \subset \gl\left(H^0(\bp^1_\bz, \sco_{\bp^1_\bz}(2))\right)$.
		Therefore, $\affineAutSub n$ is a normal subgroup of
		$\affineAut n$.	
	
		We conclude the verification of $(2)$ by showing $\affineAut n/\affineAutSub n \simeq \pgl_2$
		and identifying the composition $\gl_2/\mu_{n-2} \to \affineAut
		n \to \pgl_2$.
	The quotient $\affineAut n/\affineAutSub n$ is generated by $\gl_2/\mu_{n-2}$. 
	As mentioned above, $\affineAutSub n$ is characterized as the subgroup
	of $\affineAut n$ which acts by the central $\bg_m$ on
$\gl(H^0(\bp^1_\bz, \sco_{\bp^1_\bz}(2)))$.
	However, 
	the subgroup of $\gl_2/\mu_{n-2}$ intersecting the central
	$\bg_m \subset \gl\left(H^0(\bp^1_\bz, \sco_{\bp^1_\bz}(2))\right)$
	is the central $\bg_m \subset \gl_2/\mu_{n-2}$.
	Therefore, the quotient $\affineAut n/\affineAutSub n$
is identified with $(\gl_2/\mu_{n-2})/\bg_m \simeq \pgl_2$,
	with the induced map $\gl_2 \to \affineAut n \to \pgl_2$ the
	natural quotient map by the central $\bg_m \subset \gl_2/\mu_{n-2}$.

	Now, we check $(3)$.
	Since we have already shown $\unipotentRadical n$ is the unipotent
	radical of
	$\affineAutSub n$, it is a characteristic subgroup, i.e., it is
	preserved by automorphisms of $\affineAutSub n$.
	(Although there may not
		be a good notion of unipotent radical for general relative group schemes, 
		here we simply mean that 
		$\unipotentRadical n$ is a flat subgroup scheme of $\affineAut
		n$ which base changes to the unipotent radical on every
		geometric fiber
	over $\spec \bz$.)
	Since $\affineAutSub n$ is normal in $\affineAut n$,
	and $\unipotentRadical n \subset \affineAutSub n$ is a characteristic
	subgroup,
	we obtain $\unipotentRadical n$ is normal in $\affineAut n$.
	The quotient of $\affineAut n$ by $\unipotentRadical n$
	is then generated by $\gl_2/\mu_{n-2}$ induced by
	\eqref{equation:gl2-action}
	together with the $\bg_m$ of \eqref{equation:gm-action},
	which is
	central in $\affineAut n$.
	Because $\bg_m \cap \gl_2/\mu_{n-2} = 1$, this quotient $\affineAut
	n/\unipotentRadical n$ is	
	$\bg_m \times \gl_2/\mu_{n-2}$.
	Since the quotient $\affineAut n \to \bg_m \times \gl_2/\mu_{n-2}$ has a
	section, it follows that
	$\affineAut n \simeq \unipotentRadical n \rtimes \left( \bg_m \times
	(\gl_2/\mu_{n-2}) \right)$.

	Finally, the first part of $(4)$ follows from $(3)$ because, by definition,
	$\projAut n$ is the quotient of $\affineAut n$ by its central copy of
	central $\bg_m$.
	The identification with $\aut_{\mathbb F_{n-2}/\bz}$ was shown in
	\cite[Lemma 2.5]{landesman:the-torelli-map-restricted-to-the-hyperellpitic-locus}.
\end{proof}

\section{Defining various stacks}
\label{section:stacks-definitions}

In this section, we construct various moduli stacks related to genus $1$ curves.
The relation between genus $1$ curves and degree $2$ covers is not described
until much later in
the construction of
\autoref{notation:singular-genus-1-construction}.
See
\autoref{table:notation} for pithy descriptions of many of the stacks we will
construct.
We begin by constructing stacks related to Weierstrass curves starting in
\autoref{subsection:weierstrass},
then construct stacks related to $n$-coverings of the smooth locus of
Weierstrass curves starting in
\autoref{subsection:n-covering-stacks},
and finally construct stacks related to divisors on Hirzebruch surfaces starting in
\autoref{subsection:hirzebruch-section-stacks}. 

\subsection{Weierstrass stacks}
\label{subsection:weierstrass}

We begin by defining the stack of Weierstrass curves. 
By this we mean
genus $1$ geometrically integral curves with a section in the smooth locus. 
We also define various substacks such as the nodal and cuspidal substacks.

\begin{definition}
	\label{definition:weierstrass}
	We define {\em the stack of Weierstrass curves} $\weierstrass$ as the fibered category whose points are tuples
	$(B, f: C \to B, e: B \to C)$ where $f: C \to B$ is a proper flat finitely presented genus $1$ curves with geometrically integral fibers and
	$e: B \to C$ is	a section to $f$, (i.e., $f \circ e = \id_B$,) 
lying in the smooth locus of $f$.	
	The morphisms $(B, f: C \to B, e: B \to C) \to (B',f': C' \to B',e' : B'
	\to C')$ in this fibered category are morphisms $\alpha: B \to B',\beta: C \to C'$
	such that
	\begin{equation}
		\label{equation:}
		\begin{tikzcd} 
			C \ar {r}{\beta} \ar {d}{f} & C' \ar {d}{f'} \\
			B \ar {r}{\alpha}\ar[bend left]{u}{e} & B' \ar[bend left]{u}{e'}
	\end{tikzcd}\end{equation}
	satisfies $\beta \circ e = e' \circ \alpha$, $\alpha \circ f = f' \circ	\beta$ 
	and $C \simeq B \times_{B'} C'$.
\end{definition}

Next, we introduce the stack of Weierstrass curves with a marked singular point.
Throughout, we will indicate the marking of this singular point by including a
tilde in the notation, and we will omit the tilde when we do not mark the
singular point.

\begin{definition}
	\label{definition:singsw}
	Let $\singsw$ denote the fibered category whose $B$ points are tuples
	$(B, f: C \to B, e: B \to C, \tau: B \to C)$ where
	$(B, f: C \to B, e: B \to C) \in \weierstrass(B)$ and
	$\tau: B \to C$ is 
	a morphism such that
	$f \circ \tau = \id$ and $\tau$ factors through the
	singular locus of $f$.
	Morphisms $(B, f: C \to B, e: B \to C, \tau: B \to C) \to (B', f' : C' \to B',e': B' \to C',\tau' : B' \to C')$ consist of maps
	$\alpha: B \to B', \beta: C \to C'$ such that
	the square
	\begin{equation}
		\label{equation:}
		\begin{tikzcd} 
			C \ar {r}{\beta} \ar {d}{f} & C' \ar[swap]{d}{f'} \\
			B \ar {r}{\alpha}\ar[bend left]{u}{e}\ar[bend left = 100]{u}{\tau} & B' \ar[swap, bend right]{u}{e'} \ar[swap, bend right = 100]{u}{\tau'}
	\end{tikzcd}\end{equation}
	satisfies $C \simeq B \times_{B'} C'$,
	$\beta \circ e = e' \circ \alpha$, and $\beta \circ \tau = \tau' \circ \alpha$.

	Let $\singsw \to \weierstrass$ denote the natural map forgetting the
	section $\tau$, and let $\singw$ denote the image of $\singsw$ in
	$\weierstrass$.
\end{definition}

Next, we wish to show the above defined stacks are algebraic.
To do so, we will construct them as quotients of certain Hilbert schemes, which
we introduce next.

\begin{definition}
	\label{definition:hilbert-schemes-of-cubics}
	Let $\flagcubic$ denote the flag Hilbert scheme parameterizing $p \subset X \subset \bp^2$
	where $p$ is a section and $X$ is a relative plane cubic.
	Let $\hilbw$ denote the locally closed subscheme of $\flagcubic$ parameterizing those $p \subset X \subset \bp^2$ such that 
	$X$ is geometrically integral,
	$p \in X$ lies in the smooth locus of $X$,
	and $p$ is a flex point of $X$ (i.e., the tangent line to $X$ at $p$ meets $X$ in a subscheme of degree $3$).

	Let $\twoflagcubic$ denote the flag Hilbert scheme parameterizing $(p,q,X)$ with $p,q \in \bp^2$ two points, $X$ a plane cubic, and $p \in X, q \in X$.
	Let $\hilbsw$ denote the locally closed subscheme of $\twoflagcubic$
	such that $p$ lies in the singular locus of $X$,
	$X$ is geometrically integral, and $q$ is a flex point in the smooth
	locus of $X$.
\end{definition}

Both $\hilbw$ and $\hilbsw$ have actions of $\pgl_3$ induced by its action on the ambient $\bp^2$.

\begin{lemma}
	\label{lemma:weierstrass-representable}
	We have equivalences $\weierstrass \simeq [\hilbw/\pgl_3]$ and $\singsw
	\simeq [\hilbsw/\pgl_3]$. In particular, $\weierstrass$ and $\singsw$
	are algebraic stacks.
\end{lemma}
\begin{proof}
	First, let $\widetilde{\hilbw}$ denote the functor assigning to 
	a scheme $B$
	the set of geometrically integral genus $1$ curves $f: C \to B$ with a
	flex point $e$ in the smooth locus  
	together with an isomorphism $f_* \sco_C(3e) \simeq \sco_B^{\oplus 3}$.
	Since $\hilbw \simeq [\widetilde{\hilbw}/\bg_m]$ with $\bg_m$ scaling $\sco_B$,
	and $\weierstrass \simeq [\widetilde{\hilbw}/\gl_3]$,
	it follows that $\weierstrass \simeq [\hilbw/\pgl_3]$.

	The second isomorphism 
	$\singsw \simeq [\hilbsw/\pgl_3]$
	follows similarly because
	$\hilbsw$ represents the functor assigning to a scheme $B$
	the set of geometrically integral genus $1$ curves $f: C \to B$ with a
	flex point $e$ in the smooth locus and $\tau$ in the singular
	locus,
	together with an isomorphism $f_* \sco_C(3e) \simeq \sco_B^{\oplus 3}$,
	modulo the scaling action of $\bg_m$.
\end{proof}

\begin{definition}
	\label{definition:smooth-universal-genus-1}
	Define $\universalW \to \weierstrass$ as the universal relative proper genus 
	$1$ curve, 
	which is the quotient
	of the universal curve over $\hilbw$ by the $\pgl_3$ action as in
	\autoref{lemma:weierstrass-representable}.
	Define
	$\smoothUW \subset \universalW$ as the open substack given as the 
	smooth locus of $\universalW \to \weierstrass$.
\end{definition}

We next define the nodal and cuspidal substacks of $\singsw$.
Loosely speaking, the nodal substack $\nodesw$ is the open substack of $\singsw$
parameterizing nodal curves, while the cuspidal substack $\cuspsw$ is the closed
substack of $\singsw$ parameterizing cuspidal curves.

\begin{definition}
	\label{definition:cuspsw}
	Let $\nodesw$ be the substack of $\singsw$ 
	parameterizing those tuples $(B,f: C \to B,e : B \to C,\tau: B\to C) \in
	\singsw(B)$ such that $\tau$ maps $B$ isomorphically to the singular
	locus of $f$.
	Let $\cuspsw$ denote the substack of $\singsw$ defined as the fibered category whose fiber over $B$ is a tuple
	$(B, f: C \to B,e : B \to C,\tau: B \to C)$ as in \eqref{definition:singsw} with the following property:
	Let $X \subset C$ denote the singular locus of $f: C \to B$. Then $\ker( f_* \sco_X \to f_* \sco_{\tau(B)})$ is not the pushforward
	of a sheaf from any proper closed subscheme of $B$.
\end{definition}

\begin{remark}
	\label{remark:open-and-closed-node-cusp}
	We note that 
$\nodesw \subset \singsw$ is an open substack and the
substack $\cuspsw \subset \singsw$ is a closed substack.
Indeed, $\nodesw$ is open as it can be described as the substack where the
singular locus of $f: C \to B$ has degree $1$.
Also,
$\cuspsw$
is closed as it can be defined as the substack where the singular locus
has degree more than $1$, i.e., where $\ker( f_* \sco_X \to f_* \sco_{\tau(B)})$
of \autoref{definition:cuspsw}
is supported.
\end{remark}

\subsection{Stacks of $n$-coverings}
\label{subsection:n-covering-stacks}
We next define various stacks and schemes associated to $n$-coverings of
Weierstrass curves.
Recall from \autoref{definition:smooth-universal-genus-1} that $\smoothUW$ is
the smooth locus of the universal curve over $\weierstrass$. It turns out this
relative curve is naturally a relative group scheme.
\begin{lemma}
	\label{lemma:smooth-universal-weierstrass-group}
	The natural map $\smoothUW \to \pic^0_{\universalW \to \weierstrass}$
	sending a section $p$ to $\sco_{\universalW}(p-e)$, for $e$ the identity section,
	is an equivalence. 
	This gives $\smoothUW$ the structure of a commutative group, and in particular endows it
with a notion of multiplication by $n$.
\end{lemma}
\begin{proof}
	This may be verified on geometric
	fibers by the fibral isomorphism criterion \cite[17.9.5]{EGAIV.4}.
	It is then straightforward to directly check for the three cases of smooth, nodal, and
	cuspidal genus $1$ curves over algebraically closed fields $k$.
	For example, for smooth curves, this the usual isomorphism of an elliptic
	curve $E$ over $k$ with $\pic^0_{E/k}$.
\end{proof}

\begin{definition}
	\label{definition:nstack}
	Let $\smoothUW$ act on itself via the multiplication by $n$ map
	$\smoothUW \xra{\times n} \smoothUW$ via
	\autoref{lemma:smooth-universal-weierstrass-group}
	relatively over $\weierstrass$.
	Define the {\em stack of $n$-coverings of Weierstrass curves}
	$\nstack n := [\smoothUW/n \smoothUW]$ as the quotient stack of $\smoothUW$ with respect to the above action of $\smoothUW$ on itself.
\end{definition}

\begin{example}
	\label{example:nstack-1}
	When $n = 1$, we have $\nstack n \simeq \weierstrass$.
\end{example}

It turns out that
$\nstack n$ is actually isomorphic to a certain quotient of a Hilbert scheme by
an action of $\pgl_n$, which we call $\onestack n$ and introduce next.

\begin{definition}
	\label{definition:nhilb}
	Let $n \geq 3$ and let $\nhilb n$ denote the open subscheme of the Hilbert scheme of subschemes of $\bp^{n-1}$ 
	whose geometric points are geometrically integral genus $1$ degree $n$ curves.
	Note this is indeed open by \cite[Th\'eorem\`e 12.2.4(viii)]{EGAIV.3}.
\end{definition}

The embedding of the universal family over $\nhilb n$ into $\bp^{n-1}$ induces
an action of $\pgl_n$ on $\nhilb n$.
We next introduce analogs of the constructions we have just made where we
additionally mark
a section in the singular locus of the genus $1$ curve.

\begin{definition}
	\label{definition:singsstack}
	Let $n \in \bz_{\geq 3}$.
	Define $\singsstack n$ as the fibered category whose $B$-points are
	tuples
	$(B, f: P \to B, \iota: C \to P, \tau: B \to C)$
	such that
	\begin{enumerate}
		\item $f: P \to B$ is a Brauer-Severi scheme of relative
			dimension $n - 1$ over $B$,
		\item $\iota: C \to P$ is a closed immersion,
		\item $f \circ \iota: C \to B$ is a proper flat finitely presented arithmetic genus $1$ curve
	with geometrically integral fibers
	which has degree $n$ fppf locally on $B$, 
		\item $\tau: B \to C$ a morphism so that $f \circ \iota \circ
			\tau = \id$ so that $\tau$ factors through the singular
			locus of $f \circ \iota$.
	\end{enumerate}
	A morphism $(B, f: P \to B, \iota: C \to P, \tau: B \to C) \to (B', f': P' \to B', \iota': C' \to P', \tau': B' \to C')$ is given by maps
	$\alpha: B \to B', \beta: P \to P', \gamma: C \to C'$ so that all squares in
	\begin{equation}
		\label{equation:}
		\begin{tikzcd} 
			C \ar {r}{\gamma} \ar {d}{\iota} & C' \ar {d}{\iota'} \\
			P \ar {r}{\beta} \ar {d}{f} & P' \ar {d}{f'} \\
			B \ar {r}{\alpha} \ar[bend left]{uu}{\tau}& B'\ar[swap, bend right]{uu}{\tau'}
	\end{tikzcd}\end{equation}
	are fiber squares with $\gamma \circ \tau = \tau' \circ \alpha$.
\end{definition}

We next define the nodal and cuspidal loci of $\singsstack n$, $\nodesstack n$ and
$\cuspsstack n$.
The argument for why these are open an closed substacks of $\singsstack n$
is completely analogous to the argument given in
in the case of Weierstrass curves in
\autoref{remark:open-and-closed-node-cusp}.

\begin{definition}
	\label{definition:nodesstack-and-cuspsstack}
	For $n \geq 3$,
	let $\nodesstack n$ be the open substack of $\singsstack n$ 
	parameterizing those tuples $(B, f: P \to B, \iota: C \to P, \tau: B \to
	C) \in \singsstack n (B)$ such that $\tau$ maps $B$ isomorphically to
	the singular locus of $f \circ \iota$.
	Let $\cuspsstack n$ denote the closed substack of $\singsstack n$ defined as the fibered category whose fiber over $B$ is a tuple
	$(B,f: P \to B, \iota: C \to P, \tau: B \to C)$ as in \autoref{definition:singsstack} with the following property:
	Let $X \subset C$ denote the singular locus of $f: C \to B$. Then $\ker( f_* \sco_X \to f_* \sco_{\tau(B)})$ is not the pushforward
	of a sheaf from any proper closed subscheme of $B$.
\end{definition}
In order to show $\singsstack n$ is an algebraic stack, we will construct it as the quotient of a certain
Hilbert scheme by $\pgl_{n}$. We now define this Hilbert scheme.
\begin{definition}
	\label{definition:singshilb-and-singsstack}
	Let $n \in \bz_{\geq 3}$.
	Let $\singshilb n$ denote the functor whose $B$-points are 
	$(B, \iota: C \to \bp^{n-1}_B, \tau: B \to C)$
	defined as follows.
	Let $f:\bp^{n-1}_B \to B$ denote the structure morphism.
	Then $\iota: C \to \bp^{n-1}_B$ is a closed immersion,
	$f \circ \iota: C \to B$ a proper flat finitely presented arithmetic genus $1$ degree $n$ curve
	with geometrically integral fibers,
	and $\tau: B \to C$ a morphism so that $f \circ \iota \circ \tau = \id$
	such that $\tau$ factors through the singular locus of $f \circ \iota$.
	\end{definition}

Note that $\singshilb n$ is represented by a scheme
as it is a locally closed subscheme of a flag Hilbert scheme.

There is a natural map $\singshilb n \to \singsstack n$ sending a $B$-point of
$\singshilb n,$ $(B, \iota: C \to \bp^{n-1}_B, \tau: B \to C),$ to the tuple
$(B, f: \bp^{n-1}_B \to B, \iota: C \to \bp^{n-1}_B, \tau: B \to C)$, considered as a point
over $B$ in the fibered category $\singsstack n$.
Observe that $\pgl_n$ acts naturally on $\singshilb n$ via its action on
$\bp^{n-1}$.
By definition of $\singsstack n$, the map $\singshilb n \to \singsstack n$ is invariant for this action of $\pgl_n$.
Therefore, we obtain a map
$\phi: [\singshilb n/\pgl_n] \to \singsstack n$.

\begin{corollary}
	\label{corollary:singsstack-representable}
	For $n \geq 3$, 
	the map 
$\phi: [\singshilb n/\pgl_n] \to \singsstack n$
constructed above is an equivalence of fibered categories. In particular, $\singsstack n$ is an algebraic stack.
\end{corollary}
\begin{proof}
	We construct the inverse map.		
	Given a point 
	$(B, f: P \to B, \iota: C \to P, \tau: B \to C)$ of $\singsstack n$ over $B$, we need to construct a $\pgl_n$ torsor over this point with a $\pgl_n$ equivariant map to $\singshilb n$.
	Indeed, because $P$ is a Brauer-Severi scheme, we have a $\pgl_n$ torsor
	$T := \isom_B(\bp^n_B, P)$ over $B$ \cite[8.1]{grothendieck:brauer-i}.
	By the universal property of $\isom$, we obtain an isomorphism $P_T
	\simeq \bp^{n-1}_T$. Pulling back $C$ to $T$ gives a closed subscheme
	$C_T \to \bp^{n-1}_T$
	which has degree $n$ because it had degree $n$ over $B$ fppf locally.
	Altogether, this yields the desired $\pgl_n$ equivariant map $T \to \singshilb n$
	which is inverse to $\phi$.
\end{proof}

We next introduce various nodal and cuspidal substacks of $\onestack n$.
The following lemma will be used to define nodal and cuspidal loci without a
marked section.
\begin{lemma}
	\label{lemma:finite-marking-singularity}
	For $n \geq 3$, the maps $\singsstack n \to \onestack n$ are finite.
	Similarly, the map $\singsw \to \weierstrass$ is finite.
\end{lemma}
\begin{proof}
	The idea will be to use the valuative criterion for properness.
	We will begin with verifying the statement for $\singsstack n \to
	\onestack n$.
	Using \autoref{corollary:singsstack-representable} 
	the map $\singsstack n \to \onestack n$ is in fact the quotient of a map of 
	Hilbert schemes
	$\singshilb n \to \nhilb n$ by the respective $\pgl_n$ actions.
	Therefore, it is enough to check $\singshilb n \to \nhilb n$ is finite.
	We will do so by checking it is proper and quasi-finite.
	Quasi-finiteness follows because the singular locus of any geometrically
	integral genus $1$ curve over a field is quasi-finite.
	The valuative criterion for properness is also readily verified because
	given a 
	genus $1$ curve $E$ over a discrete valuation ring $R$ 
	with a map $\tau_K : \spec K(R) \to E$ in the singular locus of $E \to
	\spec R$, the closure of $\tau_K$ in
	$E$ defines the unique extension of $\tau_K$ to a section in the
	singular locus of $E \to \spec R$.

	The proof that $\singsw \to \weierstrass$ is finite is analogous, where
	one uses \autoref{lemma:weierstrass-representable}
	and the map of covers $\hilbsw \to \hilbw$
	in place of 
	$\singshilb n \to \nhilb n$.
\end{proof}

Since the maps $\singsstack n \to \onestack n$ and $\singsw \to \weierstrass$ are
finite by \autoref{lemma:finite-marking-singularity}, and in particular proper, we can
make sense of the nodal and cuspidal loci inside $\onestack n$ and $\weierstrass$
as the images of those in $\singsstack n$ and $\singsw$.
One can alternatively define these loci in terms of their schematic covers by respective Hilbert schemes.
\begin{definition}
	\label{definition:node-and-cusp-loci}
	For $n \geq 3,$ let $f_n: \singsstack n \to \onestack n$ be the natural map forgetting the singular section
	and define $\cuspstack n$ as the closed substack of $\onestack n$ 
	which is the image 
	$f_n(\cuspsstack n)$.
	Let $\nodestack n \subset \onestack n$ denote the locally closed 
	substack
	given as $f_n(\singsstack n) - \cuspstack n$.

	Similarly, let $f: \singsw \to \weierstrass$ denote the projection forgetting the singular section, let
	$\cuspw$ denote the closed substack of $\weierstrass$ given by $f(\cuspsw)$ and let $\nodew$
	denote the locally closed substack of $\weierstrass$ given by $f(\singsw) - \cuspw$.

	We informally say a point of $\onestack n$ or $\weierstrass$ lies in
	the {\em nodal locus}
	if it factors through $\nodestack n$ or $\nodew$
	and lies in the {\em cuspidal locus} if it factors through $\cuspstack n$ or $\cuspw$.
\end{definition}

To conclude our discussion of $\onestack n$ for the moment, we note the
following lemma,
which will allow us to lift certain points of $\onestack n$ to points of $\singsstack
n$.

\begin{lemma}
	\label{lemma:lift}
	Let $B$ be a normal integral scheme with generic point $\eta$ and let $n \geq 3$.
	If $\phi: B \to \onestack n$ is a map such that $\eta$ factors
	through $\nodestack n$, then $\phi$ factors through $\singsstack n$.
\end{lemma}

\begin{proof}
	The map $B \to \onestack n$ corresponds to a proper flat family of genus $1$
	curves $C \to B$ with a closed embedding $\iota: C \to P$ for $f: P \to B$ an
	$(n-1)$-dimensional Brauer-Severi scheme over $B$.
	Further, by assumption, the generic fiber is a nodal curve.

	By definition of $\singsstack n$, we only need produce a section $\tau: B
	\to C$ contained in the singular locus of $f \circ \iota: C \to B$.
	Because the generic fiber of $f \circ \iota$ is nodal by assumption, 
	a local calculation shows
	that the singular locus over the generic point of $B$
	maps isomorphically to the generic point of $B$.
	
	Let $Z \subset C$ denote the singular locus of the map $f \circ \iota$.
	Let $\widetilde{Z}$ denote the normalization of $Z$. 
	Then, $\widetilde{Z} \to B$ is a finite birational map of normal integral schemes, hence an isomorphism. 
	By inverting this isomorphism and composing with the map $\widetilde{Z} \to Z \to C$, we obtain the desired section $\tau: B \to C$ such that $f \circ \iota \circ \tau = \id$.
	Note that $\tau(B)$ factors through the singular locus of $C$
	because $\tau(B)$ is closed by properness of $\tau$ and the generic
	point of $B$ maps to the singular locus 
	of $f \circ \iota$ by assumption.
\end{proof}

\subsection{The stack of Hirzebruch surface sections}
\label{subsection:hirzebruch-section-stacks}
We now define various stacks and schemes relating to sections on Hirzebruch
surfaces.
One of the main tasks of this paper is to connect these to
the above stack of $n$-coverings, $\singsstack n$.
We accomplish this in
\autoref{theorem:equivalence}.

Recall the notation for Hirzebruch surfaces from \autoref{notation:hirzebruch}.

\begin{notation}
	\label{notation:rank-2-sheaf}
	The Hirzebruch surface $\mathbb F_{n-2}$ has an invertible sheaf $\scn := \sco_{\mathbb F_{n-2}}(1) \otimes g^* \sco_{\bp^1_\bz}(2)$.
	Let $\scf := (h \circ g)_*\scn$. By construction, $\scf \simeq H^0(\bp^1_\bz,
	\sco_{\bp^1_\bz}(2) \oplus \sco_{\bp^1_\bz}(n))$.
\end{notation}

\begin{remark}
	\label{remark:universal-family-smile}
	A map $B \to \bp \scf$ 
corresponds to a flat finitely presented family $X \to B$ with an embedding $X \to (\mathbb F_{n-2})_B$ with each fiber in the linear system
associated to $\scn$ of \autoref{notation:rank-2-sheaf}.
This yields a description of $\bp \scf$ as a subscheme of a component of the Hilbert scheme of subschemes of $\mathbb F_{n-2}$ over $\spec \bz$.
There is a corresponding universal family $\univfn n \subset \bp \scf \times \mathbb F_{n-2}$
with projection map $\pi: \univfn n \to \bp \scf$.
\end{remark}
We next define $\smile n$ as the subscheme of $\bp \scf$ parameterizing smooth
members of the linear system associated to $\scn$.

\begin{definition}
	\label{definition:smile-scheme}
	With $\univfn n$ as defined in
	\autoref{remark:universal-family-smile},
	let $Z \subset \univfn n$ denote the singular locus of $\pi: \univfn n \to \bp \scf$ and
	let $\pi(Z)$ denote the image of $Z$ in $\bp \scf$ .
	Define $\smile n := \bp \scf - \pi(Z)$.
\end{definition}

We will soon define the stack $\smilestack n$ in \autoref{definition:smile-stack}
which we will later see is $[\smile n / \aut_{\mathbb F_n/\bz}]$
in \autoref{lemma:smile-stack-representable}.
This stack $\smilestack n$ will involve twists of Hirzebruch surfaces, which we
now define.

\begin{definition}
	\label{definition:hirzebruch-twist}
	For $n \in \bz_{\geq 1}$ and $B$ a scheme, we define an {\em $n$-Hirzebruch twist} over $B$ as
	a tuple $(B, h: X \to B, g: F \to X)$
	where
	\begin{enumerate}
		\item $h: X \to B$ is a $1$-dimensional Brauer-Severi scheme
			over $B$,
		\item $g: F \to X$ is a relative dimension $1$ projective
			bundle over $X$\footnote{In other words, $F \to X$ is
				isomorphic to $\bp (\sce)$ for some rank $2$
				locally free sheaf
				$\sce$ on $X$. In particular, $F \to X$ is, Zariski-locally
			on $X$, isomorphic to $\bp^1_X$.} 
			such that
	there is an fppf cover $B' \to B$ having the property that
	$B' \times_B F \simeq (\mathbb F_{n})_{B'}$.
	\end{enumerate}
\end{definition}

A basic property of Hirzebruch surfaces that continues to hold for
$n$-Hirzebruch twists is the following.
\begin{lemma}
	\label{lemma:divisors-on-twists}
	For any $n \geq 1$, any $n$-Hirzebruch twist $F \xra{g} X \xra{g} B$ possesses
	a relative effective Cartier divisor $E \subset F$ and an invertible
	sheaf $\scm$  satisfying the following property:
	for $B' \to B$ a cover with $F_{B'} \simeq \mathbb F_n$,
	the pullback of $E$ to $B'$ is the relative directrix on $\mathbb F_n$
	over $B'$
	and the pullback of $\scm$ has class $2f$ over $B'$.
\end{lemma}
\begin{proof}
	We may explicitly take $\scm$ to be
	$h^*(\Omega^1_{X/B})^\vee$
	on $F$.
	The directrix $E' \subset \mathbb F_{n-2}$, an effective Cartier
	divisor, is
	preserved scheme theoretically by automorphisms of $\mathbb F_{n-2}$
	as it is the unique section in the unique divisor class
	of negative
	self intersection. Therefore, $E'$ descends 
	to the desired subscheme $E \subset F$.
\end{proof}

\begin{notation}
	\label{notation:abuse-hirzebruch}
	In light of \autoref{lemma:divisors-on-twists}, we will continue to use $e$ to denote the class
of a directrix $E$ on an $n$-Hirzebruch twist, and we use $2f$ to denote the
class of $\scm$ as in \autoref{lemma:divisors-on-twists}.
Keep in mind there may be no invertible sheaf $\scn$ with $\scn^{\otimes 2}
\simeq \scm$, hence no sheaf ``of class $f$'' on $F$.
\end{notation}

\begin{definition}
	\label{definition:smile-stack}
	Let $n \in\bz_{\geq 3}$.
	Define $\smilestack n$,
	the {\em stack of volatility smiles},
	as the fibered category over schemes whose objects over a scheme $B$ are tuples
	$(B, h: X \to B, g: F \to X, i: Z \to F)$ where 
	$(B, h: X \to B, g: F \to X)$ is an $(n-2)$-Hirzebruch twist over $B$
	and $i: Z \to F$ is a closed subscheme which fppf locally on $B$ induces a map to $\smile n$;
	in other words there is an fppf cover $B' \to B$ having the property that
	$B' \times_B F \simeq (\mathbb F_{n-2})_{B'}$ and $B' \times_B Z \to (\mathbb F_{n-2})_{B'}$ is a subscheme smooth over $B'$ which lies
	in the linear system associated to $\scn$ on $\mathbb F_{n-2}$, as defined in \autoref{notation:rank-2-sheaf}.
	
	A morphism $(B, h: X \to B, g: F \to X, i: Z \to F) \to (B', h': X' \to B', g': F' \to X', i': Z' \to F')$
	consists of maps $\alpha: B \to B', \beta: X \to X', \gamma: F \to F', \delta: Z \to Z'$
	making all squares in the diagram
	\begin{equation}
		\label{equation:}
		\begin{tikzcd} 
			Z \ar {r}{\delta} \ar {d}{i} & Z' \ar {d}{i'} \\
			F \ar {r}{\gamma} \ar {d}{g} & F' \ar {d}{g'} \\
			X \ar {r}{\beta} \ar {d}{h} & X' \ar {d}{h'} \\
			B \ar {r}{\alpha} & B'
	\end{tikzcd}\end{equation}
	fiber squares.
\end{definition}

We next wish to show $\smilestack n$ is algebraic for $n \geq 3$. 
(In fact $\smilestack n$ can be analogously defined for $n = 1, 2$ and shown to
be algebraic in those cases as well. However, this requires a separate
definition, and we omit it because we will not need it.)
	To show $\smilestack n$ is algebraic,
	we will construct an equivalence 
	$\smilestack n \to [\smile n/\aut_{\mathbb
	F_{n-2}/\bz}]$.
	The first step to doing so is to understand the stack of $(n-2)$-Hirzebruch
	twists. We then bootstrap by equipping these twists with a section of
	class $e + nf$.
	The next lemma verifies that the stack of $(n-2)$-Hirzebruch twists
	is equivalent to the stack quotient $[\spec \bz/ \aut_{\mathbb
	F_n/\bz}]$.

\begin{lemma}
	\label{lemma:descent-for-F}
	Let $B$ be a scheme. There is an equivalence of categories between $\aut_{\mathbb F_{n-2}/B}$ torsors over $B$
and $(n-2)$-Hirzebruch twists $F \xrightarrow{g} X \xrightarrow{h} B$.
\end{lemma}
\begin{proof}
	Given $F \to X \to B$, we obtain an associated $\aut_{\mathbb F_{n-2}/B}$ torsor $\isom_B(\mathbb F_{n-2}, F)$.

	Conversely, given an $\aut_{\mathbb F_{n-2}/B}$ torsor $T$, we describe the inverse construction by producing the associated $F \to X \to B$.
	First, we construct $F$ as the contracted product $T \overset{\aut_{\mathbb F_{n-2}/B}}{\times} \mathbb F_{n-2}$,
	which we recall is the quotient of $T \times \mathbb F_{n-2}$ by the functorial action of $\aut_{\mathbb F_{n-2}/B}$ given by an element $g$ sending $(t,x) \mapsto (tg^{-1}, gx)$.
	A priori, this quotient is only an algebraic space.
	
	Recall that by \autoref{lemma:affine-aut-sequences},
	$\aut_{\mathbb F_{n-2}/B}$ can be written as an extension of
	$\pgl_2$ by a certain normal subgroup $\projAutSub n$, as defined in 	
	\autoref{definition:automorphism-groups}.
	Let $T/\projAutSub n$ denote the quotient algebraic space which we note has an action of $\pgl_2$
	and define $X := T/\projAutSub n \overset{\pgl_2}{\times} \bp^1$.
	There are maps $T \to T/\projAutSub n$ and $\mathbb F_{n-2} \to \bp^1$ which induce a map $F \to X$.
	Since $\aut_{\mathbb F_{n-2}/B}$ is an affine group scheme, $T$ is a
	scheme by effectivity of descent for affine morphisms.
	Further $T/\projAutSub n$ is a $\pgl_2$ torsor over $B$ and hence also a scheme.
	The contracted product $X$ is then a Brauer-Severi scheme, as this
	contracted product is the standard way to obtain a Severi schemes from the associated torsor
	\cite[8.1]{grothendieck:brauer-i}.

	To conclude, we know
	$F \to X$ is fppf locally isomorphic to a $\bp^1$ bundle over $X$, but
	we need to show it is a scheme and even a Zariski $\bp^1$ bundle over $X$.
	Let $E'$ denote the section of $\mathbb F_{n-2} \to \bp^1$
	corresponding to the surjection $\sco \oplus \sco(n-2) \to \sco$. 
	Then
	$E'$ has divisor class $e$.
	Note that $\projAutSub n$ scheme theoretically preserves $E'$ as follows from \autoref{lemma:divisors-on-twists}. 
	Therefore, 
	$\sco_{\mathbb F_{n-2}}(e)$
	descends to an invertible sheaf on $F$, which we also call $\sco_F(e)$.
	This sheaf $\sco_F(e)$ is relatively very ample for the map $F \to X$.
	So, by descent for polarized schemes,
	we obtain that $F$ is a scheme, and contains a closed subscheme 
	$E \subset F$ which fppf locally becomes the directrix in $\mathbb F_{n-2} \to \bp^1$.
	By cohomology and base change, $h_* (\sco_F(E))$ is a locally free
	sheaf of rank $2$ on $B$, and so $\sco_F(E)$ induces the desired
	map $F \to \bp( h_* \sco_F(E))$ over $X$. This map is an isomorphism, as
	may be verified on geometric fibers by the fibral isomorphism criterion \cite[17.9.5]{EGAIV.4}.
\end{proof}

Let $n \geq 3$. We are now ready to construct the equivalence $\smilestack n \to
[\smile n/\aut_{\mathbb F_{n-2}/\bz}]$.
To start, we construct the map.
A map $B \to \smilestack n$ corresponds to the data 
$(B, h: X \to B, g: F \to X, i: Z \to F)$.
Define $I := \isom_B(\mathbb F_{n-2}, F)$.
This yields an isomorphism $(\mathbb F_{n-2})_I \simeq F_I$ 
together with a subscheme $X_I \subset F_I \simeq (\mathbb F_{n-2})_I$
that induces a map $I \to \smile n$.
Since $I$ is a torsor over $B$ using the equivalence from
\autoref{lemma:descent-for-F}, this altogether yields our desired map
$\phi: \smilestack n \to
[\smile n/\aut_{\mathbb F_{n-2}/\bz}]$.

\begin{lemma}
	\label{lemma:smile-stack-representable}
	For $n \geq 3$, the map $\phi$ constructed above is an equivalence of fibered categories. In particular, $\smilestack n$
	is an algebraic stack.
\end{lemma}
\begin{proof}
	We will construct the inverse map.	
	Given an
	$\aut_{\mathbb F_{n-2}/\bz} \times_\bz B$ torsor $I$ over $B$ and a map
	$I \to \smile n$ corresponding to a divisor $\widetilde{Z}  \subset (\mathbb
	F_{n-2})_I$ which is equivariant for the $\aut_{\mathbb F_{n-2}/\bz}
	\times_\bz B$ action, we wish to construct
	a map $B \to \smilestack n$.
	Using 
	\autoref{lemma:descent-for-F},
	we obtain an $(n-2)$-Hirzebruch twist $(B, h: X \to B, g: F \to X)$
	which pulls back to $(\mathbb F_{n-2})_I \to \bp^1_I \to I$ over $I$.
	Because the subscheme $\widetilde{Z} \subset (\mathbb F_{n-2})_I$ is
$\aut_{\mathbb F_{n-2}/\bz} \times_\bz B$ equivariant,
it descends to a closed subscheme $i: Z \to F$ which induces the map $B \to
\smilestack n$.
Note that $Z$ is smooth and has class $e + nf$, as may be verified fppf locally on $B$.

	We claim this map is inverse to the map $\phi$.
	This bijection
	follows from the bijection established in 
	\autoref{lemma:descent-for-F}, 
	together with uniqueness of descent for the closed immersion $i: Z \to F$.
\end{proof}

\section{Equivalent descriptions of $n$-coverings}
\label{section:genus-1-curves}

The main goal of this section is to prove
\autoref{proposition:equivalent-fiber-of-stack},
which gives several equivalent descriptions of $n$-coverings of a genus $1$
curves.
As preparation for proving these equivalent descriptions, we now
review some generalities on 
derived functor cohomology
of complexes.

\subsection{Derived functor cohomology of two-term complexes}
\label{subsection:derived-functor-cohomology}

Given a space $X$ and a complex of sheaves $\scc := [\scf \xra{\phi}
\scg]$ with $\scf$ in degree $0$ and $\scg$ in degree $1$, there are two distinguished triangles
\begin{equation}
	\label{equation:first-distinguished-triangle}
	\begin{tikzcd}
		\scg[-1] \ar {r} & \scc \ar {r} & \scf \to \\
\end{tikzcd}\end{equation}
\vspace{-1cm}
\begin{equation}
	\label{equation:second-distinguished-triangle}
	\begin{tikzcd}
		\ker \phi \ar {r} & \scc \ar {r} & \coker \phi[-1] \to
\end{tikzcd}\end{equation}
We will be most interested in the case that $\phi: \scf \to \scg$ is the specific
two-term complex associated to multiplication by $n$ on the smooth locus of a
Weierstrass curve.
\begin{example}
	\label{example:cohomology-from-triangles}
	Suppose $g: X \to
\spec \bz$ is a
finite locally free degree $2$ map 
and take $\scf = \scg = g_* \bg_m/\bg_m$
and $\phi = \times n$.
Then, taking cohomology associated to 
\eqref{equation:first-distinguished-triangle} and
\eqref{equation:second-distinguished-triangle} we obtain the
exact sequences
\begin{equation}
	\label{equation:first-cohomology-from-triangle}
\begin{tikzpicture}[baseline= (a).base]
\node[scale=.8] (a) at (0,0){
	\begin{tikzcd}
		0 \ar {r} &  \frac{H^0(\spec \bz, g_* \bg_m/\bg_m)}{nH^0(\spec
		\bz, g_* \bg_m/\bg_m)}  \ar {r} & H^1(\spec \bz, g_* \bg_m/\bg_m
		\xra{\times n} g_* \bg_m/\bg_m) \ar {r}{\rho} & H^1(\spec \bz, g_* \bg_m/\bg_m)[n] \ar {r} & 0 
\end{tikzcd}};
\end{tikzpicture}
\vspace{-.3cm}
\end{equation}
\begin{equation}
	\label{equation:second-cohomology-from-triangle}
\begin{tikzpicture}[baseline= (a).base]
\node[scale=.7] (a) at (0,0){
	\begin{tikzcd}
		0 \ar {r} &  H^1(\spec \bz, g_* \bg_m/\bg_m[n]) \ar
		{r}{\upsilon} & H^1(\spec \bz, g_* \bg_m/\bg_m \xra{\times n}
		g_* \bg_m/\bg_m)  \ar {r} & H^0\left( \spec \bz, \coker\left( g_* \bg_m/\bg_m \xra{\times n} g_* \bg_m/\bg_m \right) \right).
\end{tikzcd}};
\end{tikzpicture}
\end{equation}
To obtain \eqref{equation:first-cohomology-from-triangle}, we are using that the boundary map in the cohomology sequence associated to 
\eqref{equation:first-distinguished-triangle} is given by $\times n$
because the 
boundary map in
\eqref{equation:first-distinguished-triangle} is identified with $\phi: \scf \to
\scg$.
\end{example}

\subsection{Relative genus $1$ curves}
\label{subsection:relative-genus-1-curves}

Much of the remainder of this section was inspired by 
\cite[\S2]{artinSD:the-shafarevich-tate-conjecture-for-pencils-of-elliptic-curves}.
Our primary goal in this section is to prove
\autoref{proposition:equivalent-fiber-of-stack},
which gives several equivalent characterizations of points of $\nstack n$,
and relates it to $\onestack n$.
Let $E \to B$ be a genus $1$ curve with geometrically integral fibers and smooth locus
$E^{\sm}$.
One of the main issues in characteristic dividing $n$ is that $\bg_a
\xrightarrow{\times n} \bg_a$ is the $0$ map, and hence is not surjective.
In order to deal with this issue, instead of working with $H^1(B, E^{\sm}[n])$, we work with the hypercohomology group
$H^1(B, E^{\sm} \xrightarrow{\times{n}} E^{\sm})$.
The next two lemmas relate hypercohomology to points of $\nstack n$.

\begin{lemma}
	\label{lemma:sections-of-quotient-and-torsors}
	Let $\phi: G \to H$ be a map of smooth commutative group schemes over $B$. 
	Let $[H/\phi(G)]$ denote the quotient stack of $H$ by the action of $G$ via $\phi$.
	Then $H^1(B, H \xrightarrow{\phi}G) = H^0(B, [H/\phi(G)])$.
	These sets are also in bijection with
	pairs $(T,\psi: \phi_* T \to H)$ up to isomorphism of torsors, where
	$T$ is a $G$ torsor and $\psi$ is an isomorphism of $H$
	torsors.
\end{lemma}
\begin{proof}
	Given a $G$ torsor $T'$, let $\phi_* T'$ denote the $H$ torsor given explicitly as the image of $[T'] \in H^1(B, G)$ under the map $H^1(B, G) \to H^1(B, H)$ induced by $\phi$.

	We will first show elements of $H^0(B, [H/\phi(G)])$, i.e., maps $B \to [H/\phi(G)]$, are in bijection with pairs $(T, \psi)$ where $T \to B$ is a $G$ torsor and $\psi: \phi_* T \to H$ is
	an isomorphism of $H$-torsors.
	From the definition,
	an element of $H^0(B,[H/\phi(G)])$
	corresponds to a $G$ torsor $T \to B$ 
	and a $G$-equivariant map $\alpha: T \to H$.
	Such $G$-equivariant maps are in bijection with $H$-equivariant maps
	$\psi: \phi_* T \to H$ via precomposition with the natural map $\theta:
	T \to \phi_* T$ induced by $\phi: G \to H$.

	It remains to show that pairs $(T, \psi)$ for $T$ a $G$-torsor and $\psi$ a trivialization of $\phi_* T$ correspond to elements of
	$H^1(B, G \xrightarrow{\phi} H)$. We can describe elements of
	$H^1(B, G \xrightarrow{\phi} H)$ in terms of Cech hypercohomology
	by a pair $(\alpha, \beta)$ where $\alpha$ is a $1$-cocycle for $G$ and $\beta$ is a $0$-cochain for $H$ with $\phi(\alpha_{ij}) = \beta_i - \beta_j$, taken
	with respect to an fppf cover $U_i$ of $B$ trivializing $T$.
	The datum of the $\alpha_{ij}$ are equivalent to the specification of the torsor $T$ while the equality 
	$\phi(\alpha_{ij}) = \beta_i - \beta_j$,
	is equivalent to giving a trivialization of the torsor $\phi_* T$.
\end{proof}

At this point the reader may wish to recall the definitions of $\weierstrass, \universalW, \smoothUW$
and $\nstack n$ from \autoref{definition:weierstrass},
\autoref{definition:smooth-universal-genus-1}, and \autoref{definition:nstack}.

\begin{lemma}
	\label{lemma:fiber-of-stack-as-torsors}
	Let $g: B \to \weierstrass$ be given by a tuple $(B, f: E \to B, e: B \to
	E)$.
	Let $E^{\sm}$ denote the smooth locus of $f$ and let $\pi_n: \nstack n
	\to \weierstrass$ denote the canonical projection induced by $\smoothUW
	\to \weierstrass$.
	Then, $(B \times_{g,\weierstrass,\pi_n} \nstack n)(B) \simeq H^1(B,
	E^{\sm} \xrightarrow{\times n} E^{\sm})$.
\end{lemma}
\begin{proof}
	Because of the fiber squares
	\begin{equation}
		\label{equation:}
		\begin{tikzcd} 
			E^{\sm} \ar {r} \ar {d}{\times n} & \smoothUW \ar {d}{\times n} \\
			E^{\sm} \ar {r} \ar {d} & \smoothUW \ar {d} \\
			B \ar {r}{g} & \weierstrass
	\end{tikzcd}\end{equation}
	we obtain a fiber square
	\begin{equation}
		\label{equation:}
		\begin{tikzcd} 
			\left[E^{\sm}/nE^{\sm}\right] \ar {r} \ar {d} & \nstack n \ar {d} \\
			B \ar {r}{g} & \weierstrass,
	\end{tikzcd}\end{equation}
	where $[E^{\sm}/nE^{\sm}]$ denotes the quotient stack of $E^{\sm}$ by the action of $E^{\sm}$ on itself via multiplication by $n$.
	Therefore, $(B \times_{g,\weierstrass,\pi_n} \nstack n)(B)$ is identified with $[E^{\sm}/nE^{\sm}](B) = H^0(B, [E^{\sm}/nE^{\sm}])$.
	By \autoref{lemma:sections-of-quotient-and-torsors}, we can then
	identify $H^0(B, [E^{\sm}/nE^{\sm}]) \simeq H^1(B, E^{\sm}\xrightarrow{\times n} E^{\sm})$.
\end{proof}
Recall that for $G \to B$ a group scheme, $T$ a $G$-torsor over $B$, and $n
\in \bz_{\geq 1}$,
we use $n_* T$ to denote the $G$ torsor corresponding to the image of $[T] \in H^1(B, G)$
under the map $H^1(B, G) \to H^1(B, G)$ induced by $\times n: G \to G$.
We next show that torsors for the smooth locus of
a genus $1$ curve always have natural compactifications.

\begin{lemma}
	\label{lemma:compactification-scheme}
	Given $(B, f: E \to B, e: B \to
	E) \in \weierstrass(B)$, and $E^{\sm}$ the smooth locus of $f$,
	let $h: T \to B$ be an $E^{\sm}$ torsor with $n_*T$ the trivial
	$E^{\sm}$ torsor.
	Then, there exists a flat projective scheme
	$\overline{h}: \overline{T} \to B$ which is a relative curve of genus $1$ with geometrically integral fibers, such that $T$ is identified with
	the smooth locus of $\overline{h}$.
	Additionally, there is a Brauer-Severi scheme $P \to B$ of relative dimension $n-1$ and a map $\overline{T} \to P$, which is an embedding if $n \geq 3$.

	Further, $\overline{T}$ is unique in the following sense: given any other such
	flat proper algebraic space relative genus $1$ curve $\overline{h}': \overline{T}' \to B$ with geometrically
	integral fibers so that $T$ is identified with the smooth locus of $\overline{h}'$ there is a unique isomorphism
	$\sigma: \overline{T} \to \overline{T}'$ over $B$ so that the
	composition $T \to \overline{T} \xrightarrow{\sigma} \overline{T}'$ is
	the given inclusion $T \to \overline{T}'$.
\end{lemma}
\begin{proof}
	The existence of $\overline{T}$ as a proper algebraic space (instead of
	a projective scheme) is
	formal.
	Indeed, using \autoref{lemma:sections-of-quotient-and-torsors},
	our given torsor $T$ with trivialization of $n_* T$ corresponds to
	an element of $H^1(B, E^{\sm}
	\xrightarrow{\times n} E^{\sm})$.
	In terms of Cech cocycles, 
	for $U_i \to B$ a suitable fppf cover,
	this element may be described as a $1$-cocycle $s_{ij} \in H^0(U_{ij}, E^{\sm}|_{U_{ij}})$ 
	and a $0$-chain $t_i$ so that $n \cdot s_{ij} = t_i - t_j$. 
	The $s_{ij}$ specify cocycle data to glue $E_{U_i}|_{U_{ij}}$ to
	$E_{U_j}|_{U_{ij}}$ over $U_{ij}$, and since the $s_{ij}$ are a cocycle, i.e., $s_{ij} + s_{jk} = s_{ik}$ on $U_{ijk}$, 
	they define descent data
	so as to construct an algebraic space $\overline{T}$. Since $T$ is
	constructed using the same cocycle $s_{ij}$, but via gluing
	$E^{\sm}_{U_i}|_{U_{ij}} \to E^{\sm}_{U_j}|_{U_{ij}}$,
	we obtain the desired open embedding $T \to \overline{T}$, as can be
	verified locally.

	The uniqueness claim on $\overline{h}$ may be verified locally on $B$,
	and hence we may assume our torsors are trivial.
	In this case, the uniqueness holds because
	each $E_{ij}$ is separated over $U_{ij}$ so the desired map $\overline{T} \to \overline{T}'$ is unique.

	It remains to show that $\overline{T}$ is a projective scheme.
	For this, we use
	that descent for polarized schemes is effective.
	Namely, choose descent data $s_{ij}$ and $t_i$ for $T$ as above, and let $e : B \to E^{\sm}$ denote the given section in the smooth locus associated with the data of the map $B \to \weierstrass$.
	Consider the line bundle $\scl_{t_i} := \sco_{E_{U_i}}( (n-1) \cdot e +
	t_i )$ on $E$ viewed as a degree $n$ line bundle on $E_{U_i}$.
	The global sections of this invertible sheaf induce an embedding $E_{U_i} \to \proj f_* \scl_{t_i}$.

	We will check next that $(s_{ij})^* \scl_{t_i}|_{U_{ij}} \simeq
	\scl_{t_j}|_{U_{ij}}$ and that the $\scl_{t_i}^{\otimes n}$ descend to an invertible sheaf
	on $\overline{T}$.
	Because translation by $s_{ij}$ corresponds to tensoring with the degree $0$
	line bundle $\sco_{E_{U_{ij}}}(n \cdot s_{ij}- n \cdot e)$,
	we find  $(s_{ij})^* \scl_{t_i} \simeq \sco_{E_{U_{ij}}} (n \cdot s_{ij}
	+ t_i - e)$.
	The condition that $n \cdot s_{ij} = t_i - t_j$
	can be written in terms of degree $0$ line bundles as
$\sco_{E_{U_{ij}}}(n \cdot s_{ij} - n \cdot e) \simeq \sco_{E_{U_{ij}}}(t_i -
t_j)$.
This yields the desired isomorphism 
\begin{align*}
	\scl_{t_i}|_{U_{ij}} \simeq
\sco_{E_{U_{ij}}}((n-1)e+t_i) \simeq \sco_{E_{U_{ij}}}
(n \cdot s_{ij} + t_j - e) \simeq (s_{ij})^* \scl_{t_j}|_{U_{ij}}. 
\end{align*}

	Although the isomorphisms 
	$\scl_{t_i} \simeq (s_{ij})^* \scl_{t_j}$
	may not satisfy the cocycle condition,
	we claim that when we multiply the above isomorphisms by $n$, we obtain isomorphisms
	$\scl_{t_i}^{\otimes n} \simeq (s_{ij})^* \scl_{t_j}^{\otimes n}$, which
	do satisfy the cocycle condition.
	To verify this, we can do so after pushing forward via $f$. We then
	obtain that the induced isomorphisms
	$f_*\scl_{t_i}^{\otimes n} \simeq f_*((s_{ij})^* \scl_{t_j}^{\otimes n})$ do satisfy the cocycle condition 
	because the corresponding $\pgl_n$ torsor
	lifts to a $\gl_n$ torsor
	by \cite[Theorem 6.6.17(ii)]{poonen:rational-points-on-varieties}
	(or more precisely the immediate generalization of its proof to
	arbitrary base schemes in place of fields).
	Therefore, the isomorphisms $\scl_{t_i}^{\otimes n} \simeq (s_{ij})^* \scl_{t_j}^{\otimes n}$ also satisfy the cocycle condition.
	Then, the polarizations $\scl_{t_i}^{\otimes n}$ on
	$E_{U_i}$ induce descent data coming from translation by $s_{ij}$.
	Effectivity of descent for polarized schemes yields a projective scheme $\overline{T} \to B$ whose base change to $U_i$ is $E$.

	Finally, we verify the statement regarding the Brauer-Severi scheme.
	We have descent data for the schemes $P_i := \proj f_* \scl_{t_i}$, with the line bundle $\sco_{P_i}(n)$, again induced by translation by $s_{ij}$.
	Again, by effectivity of descent for polarized schemes, we obtain a scheme $P \to B$ with $P_{U_i} \simeq P_i$, and so $P$ is a Brauer-Severi scheme.
	Effectivity of descent for closed embeddings implies that the natural closed embeddings $E_{U_i} \to P_i$ descend to a closed embedding $\overline{T} \to P$,
	which is the claimed embedding into a Brauer-Severi scheme.
\end{proof}

With notation as in \autoref{lemma:fiber-of-stack-as-torsors}
we use $\aut_{(E,e)/B}$ to denote the automorphism scheme of the genus $1$
curve $E$ which preserve the given section $e$ lying in the smooth locus.

\begin{lemma}
	\label{lemma:fiber-of-stack-pre-quotient}
	With notation as in \autoref{lemma:fiber-of-stack-as-torsors}
	the following sets are isomorphic, functorially in $B$
	and respecting the action of $\aut_{(E,e)/B}(B)$:
	\begin{enumerate}
		\item[\customlabel{nstack-map-pre}{(1)}] $(B
			\times_{g,\weierstrass,\pi_n} \nstack n)(B)$;
		\item[\customlabel{hyper-torsor-pre}{(2)}]$H^1(B, E^{\sm}
			\xrightarrow{\times n} E^{\sm})$;
		\item[\customlabel{nstack-section-pre}{(3)}] $H^0(B,
			[E^{\sm}/nE^{\sm}])$;
		\item[\customlabel{esmooth-torsor-pre}{(4)}] the set of pairs $(T,
			\psi: n_* T \to E^{\sm})$ where $T$ is an $E^{\sm}$
			torsor and $\psi$ is an isomorphism of $E^{\sm}$
			torsors, up to isomorphism;
		\item[\customlabel{picn-point-pre}{(5)}] the set of pairs $(T, M)$
			where $T$ is an $E^{\sm}$ torsor and $M \in
			\pic^n_{\overline{T}/B}(B)$ for $\overline{T}$ the flat
			proper genus $1$ curve associated to $T$ as in
			\autoref{lemma:compactification-scheme},
			up to isomorphism.
			Here $(T,M)$ is isomorphic to $(T,M')$ if they differ
			by translation by a point of
		$E^{\sm}\simeq \pic^0_{{\overline T}/B}$.
	\end{enumerate}
\end{lemma}
\begin{proof}
	The equivalence of \autoref{nstack-map-pre}, \autoref{hyper-torsor-pre} were established in \autoref{lemma:fiber-of-stack-as-torsors}.
	The equivalence of the \autoref{hyper-torsor-pre},
	\autoref{nstack-section-pre}, and \autoref{esmooth-torsor-pre} follows from \autoref{lemma:sections-of-quotient-and-torsors}.
	Finally, the identification
	of \autoref{esmooth-torsor-pre} with
	\autoref{picn-point-pre}
	follows from the same argument given in 
\cite[Proposition
1.7]{artinSD:the-shafarevich-tate-conjecture-for-pencils-of-elliptic-curves}.
	\end{proof}

We are now ready to combine the above lemmas to verify the main result of this
section.

\begin{proposition}
	\label{proposition:equivalent-fiber-of-stack}
	With notation as in \autoref{lemma:fiber-of-stack-as-torsors}
	the following sets are isomorphic, functorially in $B$:
	\begin{enumerate}
		\item[\customlabel{nstack-map}{(1)}] $(B
			\times_{g,\weierstrass,\pi_n} \nstack
		n)(B)/\aut_{(E,e)/B}(B)$;
		\item[\customlabel{hyper-torsor}{(2)}]$H^1(B, E^{\sm}
			\xrightarrow{\times n} E^{\sm})/\aut_{(E,e)/B}(B)$;
		\item[\customlabel{nstack-section}{(3)}] $H^0(B,
			[E^{\sm}/nE^{\sm}])/\aut_{(E,e)/B}(B)$;
		\item[\customlabel{esmooth-torsor}{(4)}] the set of pairs $(T,
			\psi: n_* T \to E^{\sm})$ where $T$ is an $E^{\sm}$
			torsor and $\psi$ is an isomorphism of $E^{\sm}$
			torsors, up to isomorphism, modulo the action of
			$\aut_{(E,e)/B}(B)$;
		\item[\customlabel{picn-point}{(5)}] the set of pairs $(T, M)$
			where $T$ is an $E^{\sm}$ torsor and $M \in
			\pic^n_{\overline{T}/B}(B)$ for $\overline{T}$ the flat
			proper genus $1$ curve associated to $T$ as in
			\autoref{lemma:compactification-scheme},
			up to isomorphism, modulo the action of
			$\aut_{(E,e)/B}(B)$.
			\end{enumerate}
	If further, $n \geq 3$, the above are also equivalent to
	\begin{enumerate}
		\item[\customlabel{compactified-torsor}{(6)}] the set of tuples
			$(T, P, \iota)$, taken up to automorphism, where $T$ is an $E^{\sm}$ torsor, $P$ is an $n-1$ dimensional Brauer-Severi scheme over $B$ and,
			for $\overline{T}$ the flat proper genus $1$ curve associated to $T$, 
			$\iota: \overline{T} \to P$ is a closed embedding;
		\item[\customlabel{hilb-mod-pgl}{(7)}] maps $B \to [\nhilb n/\pgl_n]$ corresponding to
			$(n-1)$-dimensional Brauer-Severi schemes $P \to B$ with closed embeddings $\overline{T}
			\to P$ 
			of genus $1$ flat projective curves with geometrically
			integral fibers, such that the smooth locus of
			$\overline{T} \to B$ is an $E^{\sm}$ torsor.
	\end{enumerate}
\end{proposition}
\begin{proof}
	The equivalence of the \autoref{nstack-map}-\autoref{picn-point} follows from the analogous
	statements in \autoref{lemma:fiber-of-stack-pre-quotient},
	as the bijections there are compatible with the actions of
	$\aut_{(E,e)/B}(B)$.

We now assume $n \geq 3$. We next show how to construct the data of
\autoref{compactified-torsor} from \autoref{esmooth-torsor}, and then how to
construct the data of \autoref{picn-point} from
\autoref{compactified-torsor}. This will be done in a bijective
fashion under the identification of \autoref{esmooth-torsor} and
\autoref{picn-point} above.
Given the data of \autoref{esmooth-torsor}, we obtain the data of
\autoref{compactified-torsor} from
\autoref{lemma:compactification-scheme}. 
Note that if we have $T$ and $T'$ as in
\autoref{esmooth-torsor} which are related by an automorphism of $(E,\sigma)$, 
then they will still yield
isomorphic tuples $(T,P,\iota)$ and $(T', P',\iota')$.

Conversely, given the data of \autoref{compactified-torsor}, we recover the data of
\autoref{picn-point} as follows.
To begin, note that
that $\pic_{P/B} \simeq \underline{\bz}$, for $\underline{\bz}$ the constant group scheme associated to $\bz$ on $B$. Of course, when $P$ is a nontrivial Brauer-Severi scheme, the unique
point of $\pic_{P/B}(B)$ corresponding to the positive generator of $\bz$ will
not correspond to a line bundle on $P$. Rather, it corresponds to the
collection of line bundles
$\sco_{P_{U_i}}(1)$ for $U_i \to B$ an fppf cover trivializing the Brauer-Severi scheme $P \to B$.
There is a natural map $\pic_{P/B} \to \pic_{\overline{T}/B}$ given by pullback, and the image of $1 \in \bz \simeq \pic_{P/B}(B)$ pulls back to an element
of $\pic^n_{\overline{T}/B}(B)$ since on geometric fibers over $B$,
$\overline{T}$ has degree $n$ in $P$.
This is the desired point of $\pic^n_{\overline{T}/B}(B)$ yielding
the data of 
\autoref{picn-point}.
If we have an automorphism $(T,P, \iota) \to (T, P, \iota)$,
it will be induced by an automorphism of $E$ (which possibly does not fix $\sigma$), as can be checked fppf locally
where $T$ becomes isomorphic to $E$.
Therefore, the above constructed element of 
\autoref{picn-point} is well defined.

	It only remains to explain the equivalence of \autoref{compactified-torsor} and \autoref{hilb-mod-pgl}.
	For $B$ a scheme, a point $[\nhilb n/\pgl_n](B)$ corresponds to a
	$\pgl_n$ torsor $R \to B$ and a $\pgl_n$ equivariant map $R \to \nhilb
	n$.
	The latter corresponds to a subscheme $\widetilde{\overline{T}} \to \bp^n_R$
	flat over $B$ of degree $n$ whose geometric fibers lie in $\nhilb n$.
	By the equivalence between $\pgl_n$ torsors and Brauer-Severi schemes
	and effectivity of descent for closed subschemes,
	such data descends to a Brauer-Severi scheme
	$P \to B$ and a
	subscheme $\overline{T} \to P$ flat over $B$ of degree $n$ whose geometric
	fibers lie in $\nhilb n$.
	This shows that such maps $B \to [\nhilb n/\pgl_n]$ are in bijection
	with data as in \autoref{compactified-torsor}.
\end{proof}

\section{Singular genus $1$ curves and Hirzebruch surfaces}
\label{section:hirzebruch-bijection}

This section is perhaps the most technically involved section of the paper,
and its goal is to construct an equivalence of stacks $\smilestack n \simeq
\singsstack n$.
This equivalence seems to us quite intuitive, with the forward map is given by a
certain linear system while the reverse map is given by blowing up the singular
section, see \autoref{subsection:inverse-map}.
However, to actually define the map of stacks, 
we are forced to carry out these constructions carefully in families, which
unfortunately makes the proof rather long.
The forward map is constructed in \autoref{lemma:forward-equivalence}.
The inverse map is quite a bit more involved, and is constructed in
\autoref{lemma:reverse-equivalence}.
The main result is then \autoref{theorem:equivalence}, that these two maps are
inverse.

For $B$ a scheme, let 
$(B, h: X \to B, g: F \to X, i: Z \to F) \in \smilestack n(B)$.
Recall from \autoref{notation:line-bundles-on-F}.
that we have divisor classes $e$
and $2f$ on $F$. 
\begin{notation}
	\label{notation:line-bundles-on-F}
	Recall from \autoref{notation:abuse-hirzebruch}
	that $e$ denotes the class of the relative directrix on $F \to B$.
Observe $Z \to F$ has class $e + nf$.
Also, $g(Z \cap E)$ has class
$\sco_X(2)$ since $E$ restricts to a Cartier
on the smooth genus $0$ curve $Z$ which has degree $2$ since it has degree $2$
on fibers.
We define the invertible sheaf $\scl := \sco_F(Z) \otimes
\sco_F(g^{-1}(g(Z\cap E)))$.
This has class $(e+nf) - 2f = e+(n-2)f$.
\end{notation}

We next verify the complete linear system associated to the class $e + (n-2)f$ on $F$ defines a morphism from $F$ to a rank $n-1$ projective bundle over $B$.
\begin{lemma}
	\label{lemma:locally-free-e-n-2-f}
	With notation as in \autoref{notation:line-bundles-on-F}
	$(h \circ g)_*\scl$ is a locally free sheaf on $B$ of rank $n$.
\end{lemma}
\begin{proof}
	We can verify the statement fppf locally on $B$, and hence we may assume
	$F \simeq \mathbb F_{n-2}$ and $\scl \simeq \sco_F(e + (n-2)f)$.
	Once we verify 
	first that
	$(h \circ g)_* \sco_{\mathbb F_{n-2}}(e + (n-2)f)$ is locally free of rank $n$ and
	second that
	$R^1(h \circ g)_* \sco_{\mathbb F_{n-2}}(e + (n-2)f) =0$
	after base change to any algebraically closed field mapping to $B$, the statement will follow from cohomology and base change over $B$.
	Now that we have reduced to the case $B$ is a point, the claims are
	standard calculations. The first can be deduced by taking cohomology
	of the exact sequence $\sco_{\mathbb F_{n-2}} \to
			\sco_{\mathbb F_{n-2}}(e+(n-2)f) \to
			\sco_C(e+(n-2)f)$
		for $C$ a curve of class $e + (n-2)f$, using Riemann--Roch on $C$ and adjunction on $\mathbb F_{n-2}$.
	The second claim follows from the Leray spectral sequence applied to the
	composition $h \circ g$.
\end{proof}

Define $P := \bp (h \circ g)_* \scl$.
The surjection
$(h\circ g)^* (h \circ g)_* \scl \to \scl$ defines a map 
$\phi: F \to P$ because the linear system is basepoint free, as can be verified on fibers.
We are now ready to construct the desired map
$\smilestack n \to \singsstack n$.

\begin{lemma}
	\label{lemma:forward-equivalence}
	Let $n \in\bz_{\geq 3}$.
The map $\phi: F \to P$ defined above sends $Z$ to a curve $C$ over $B$
	with a section $\tau: B \to C$
	yielding a point $(B, f: P \to B, \iota: C \to P, \tau: B \to C) \in \singsstack n$.
	This induces a map of stacks
	$\smilestack n \to \singsstack n$.
\end{lemma}
\begin{proof}
	Recall that we begin with a $B$-point of $\smilestack n$ corresponding to $Z \xra{i} F \xra{g} X \xra{h} B$
	and we have produced $P := \bp (h \circ g)_* \scl$
	with a map $F \xra{\phi} P$.
	We wish to produce a genus $1$ curve $C \xra{\iota} P$
	with geometrically integral fibers
	and a section $\tau: B \to C$ lying in the singular locus of $C$.
	We will take $C$ to be the image of $Z \xra{i} F \xra{\phi} P$
	and $\tau$ to be the image of the directrix $E \to F \xra{\phi} P$.
	Note that $E$ is contracted to a section under the linear system $(h
	\circ g)_* \scl$, as can be checked fppf locally where $\scl$ is
	isomorphic to $\sco_{\mathbb F_{n-2}}(e + (n-2)f)$.
	The construction we give will be natural, and so maps between $B$ points
	of $\smilestack n$ will induce maps between the corresponding points of
	$\singsstack n$, inducing the desired map of stacks.

	To complete the proof, we check that the tuple 
$(B, f: P \to B, \iota: C \to P, \tau: B \to C)$ described above lies in
$\singsstack n(B)$.
	By \autoref{lemma:smile-stack-representable},
	$\smilestack n$ is smooth and in particular reduced, so any map $B \to \smilestack n$ factors through the reduction of $B$, and we may assume $B$ is reduced.
	Since $B$ is reduced, we can identify the reduced set theoretic and scheme theoretic images of $Z \xra{i} F \xra{\phi} P$ and $E \to F \xra{\phi} P$.
	In particular, the formation of the scheme theoretic image commutes with base change along geometric points of $B$.
	It is therefore enough to complete the proof in the case $B = \spec k$
	for $k$ an algebraically closed field.

	We now assume $B = \spec k$, and can verify 
	$(B, f: P \to B, \iota: C \to P, \tau: B \to C) \in \singsstack n(B)$ in
	this case.
	The remainder of the proof is a
	standard algebro-geometric calculation.
	One first notes that $Z$ is geometrically connected because $e + nf$ is
	ample, and then a standard intersection theory calculation shows it has
	class $e + (n-2)f$.
	Smoothness of $Z$ implies that $Z$ is also geometrically integral.
	By analyzing
	the global sections
	$H^0(\mathbb F_{n-2}, \sco_{\mathbb F_{n-2}}(e + (n-2)f))$
	we see that any divisor meeting the directrix $E$ contains $E$, and so $E$ is
	contracted under $\phi$.
	Further, there is a codimension $1$ subspace vanishing on 
	$E$, and we can use this to deduce that $\phi$ is an embedding away from $E$.
	Because $Z$ meets $E$ in a degree $2$ subscheme, the image $C$ of $Z$
	under $\phi$ is the pushout of $Z$ obtained by gluing this degree $2$
	subscheme to a point, and hence has genus $1$ with singular locus given
	as the image of $Z \cap E$.
\end{proof}

\subsection{Constructing the inverse map}
\label{subsection:inverse-map}

We next construct the map inverse to that of \autoref{lemma:forward-equivalence}.
The basic idea
behind the construction previously described in
\autoref{section:geometric-bijection},
as we now recall.
If we begin with a singular genus $1$ curve in $\bp^{n-1}$ with a marked singular point, there is a unique surface cone formed by the union of lines joining the singular point and points on the genus $1$ curve. Blowing up the curve inside the cone at the singular point yields a rational curve of class $e+(n-2)f$ on the Hirzebruch surface $\mathbb F_{n-2}$.
The simplest path we found to carry out this construction in families was to 
first blow up the curve and use this to construct the Hirzebruch surface as a family of lines over
the blown up curve. The image of this Hirzebruch surface under a suitable map to $\bp^{n-1}$ will then be the desired surface cone.
We complete this construction in 
\autoref{lemma:reverse-equivalence}.

However, carrying out the above construction out in families is unexpectedly subtle, due
to the following issues:
First, we need to know the blow up construction commutes with arbitrary base
change, as is verified in \autoref{proposition:genus-0-blow-up}.
Second, we need to verify a certain scheme theoretic image of a map is flat in
order to apply cohomology and base
change in \autoref{proposition:w-flat}. 
This verification of flatness rests on a cute
generalization of
the Chinese remainder theorem \autoref{lemma:chinese-remainder-generalization}.
As a first step, we now check that the blow up of $C$ is a smooth genus $0$ curve over the base.

\begin{proposition}
	\label{proposition:genus-0-blow-up}
	Let $n \geq 3$
	and let $(B,f:P \to B,\iota: C \to P,\tau: B \to C) \in \singsstack n(B)$.
	Then, the blow up $\blow_\tau C \to B$ is a smooth proper genus $0$ curve with geometrically connected fibers.
	Further, the formation of this blow up commutes with arbitrary base change on $B$.
	That is, for any $B' \to B$, if we let $C' := C \times_B B'$ and $\tau'$ denote the base change of $\tau$ to $B'$,
	the natural map $\blow_\tau C \times_B B' \to \blow_{\tau'} C'$ induced by the universal property of blow ups is an isomorphism.
\end{proposition}
\begin{proof}
	As a first step, we may freely pass to an fppf cover, and hence reduce to the case that $f: C \to B$ has a section
	$\sigma: B \to C$ contained in the smooth locus of $f$.

	We now use the above section to express $C$ as a relative plane curve
	over $B$ defined by a simple equation.
	Since $C$ is a finitely presented genus $1$ curve over $B$ with geometrically integral
	fibers and a section in the smooth locus, from
	the definition of $\weierstrass$
	\autoref{definition:weierstrass},
	we obtain a map $B \to \weierstrass$.
	Because $\weierstrass$ has an fppf cover by $\ba^5_{a_1, a_2, a_3, a_4,
	a_6}$ parameterizing the coefficients of Weierstrass equations,
	(see \cite[Definition 2.2.10 and Proposition 2.2.13]{landesman:thesis} for more details of this standard
	description of Weierstrass equations,)
	after replacing $B$ by an fppf cover we may assume that $B = \spec R$ is affine 
	and $C$
	is defined by an equation
	of the form
	$y^2z + a_1 xyz + a_3 y = x^3 + a_2 x^2z + a_4xz^2 + a_6z^3$ in $\proj R[x,y,z] = \bp^2_B$,
	with $a_1, a_2, a_3, a_4, a_6 \in R$.
	After again possibly replacing $B$ by a cover, we may assume that the section $\tau$ lying in the singular locus 
	is given by $x = y = 0$.
	The condition that $C$ passes through $\tau$ forces $a_6 = 0$ while the condition that $C$ is singular at $\tau$ forces $a_3 = a_4 = 0$.
	Therefore, $C$ is given by an equation of the form $y^2z + a_1 xyz = x^3 + a_2 x^2z$.
	
	A standard, direct calculation shows the blow up of 
	$\spec R[x,y]/(y^2 + a_1 xy = x^3 + a_2 x^2)$
	at $x = y =0$
	is given by $\proj R[x,y,X,Y]/(xY-yX, Y^2 + a_1YX - xX^2 -a_2 X^2)$,
	where $X,Y$ have degree $1$ and $x,y$ have degree $0$.
	This explicit description implies the blow up commutes with arbitrary base
	change.

	It remains to verify that $\blow_\tau C$ is smooth of genus $0$.
	The genus $0$ statement may be checked on fibers, which holds because it is the blow up of a
	geometrically integral genus $1$
	curve at a singular point.
	Finally, 
	using 
	the Jacobian criterion for
	smoothness, it is straightforward to directly verify smoothness of
$\proj R[x,y,X,Y]/(xY-yX, Y^2 + a_1XY - xX^2 - a_2 X^2)$
	on the two charts $Y \neq 0$
	and $X \neq 0$.
\end{proof}

Having constructed our smooth genus $0$ curve,
the next step is to construct a relative $(n-2)$-Hirzebruch twist containing $\blow_\tau C$.
We begin by introducing some notation used to define this Hirzebruch twist.
\autoref{figure:image-blow-up} may be helpful in visualizing some of the objects
at play.

\begin{notation}
	\label{notation:image-of-blow-up}
	Let $n \geq 3$. Let $(B,f:P \to B,\iota: C \to P,\tau: B \to C) \in \singsstack n(B)$.
Blowing $C$ up at $\tau$, we obtain a map $\nu: \blow_\tau C \to C$.
Let $E_\tau C \subset \blow_\tau C$ denote the exceptional divisor associated to the blow up of $C$ at $\tau$.
From this, we obtain a map
\begin{align*}
	\psi := \left( \id, \iota \circ \nu\right) \coprod \left( \id, \iota
	\circ \tau \circ f \circ \iota \circ \nu \right): \blow_\tau C \coprod \blow_\tau C \to \blow_\tau C \times_B P
\end{align*}
Let $W$ denote the scheme theoretic image of $\psi$.
We let $i_1: \blow_\tau C \to \blow_\tau C \coprod \blow_\tau C$ denote the
first inclusion and 
$i_2: \blow_\tau C \to \blow_\tau C \coprod \blow_\tau C$
denote the second inclusion.
Let $L$ denote the image of $\psi \circ i_1: \blow_\tau C \to W$ and let $M$ denote the
image of $\psi \circ i_2: \blow_\tau C \to W$.
In particular, the composition of $\psi \circ i_1$ with the projection $W \to P$ is $\iota
\circ \nu$ while the composition of $\psi \circ i_2$ with the projection $W \to P$ is
$\iota \circ \tau \circ f \circ \iota \circ \nu$,
the constant map to $\iota \circ \tau(B)$.
\end{notation}

\begin{figure}
	\centering
	\includegraphics[scale=.6, trim={0 15.5cm 0 1cm}]{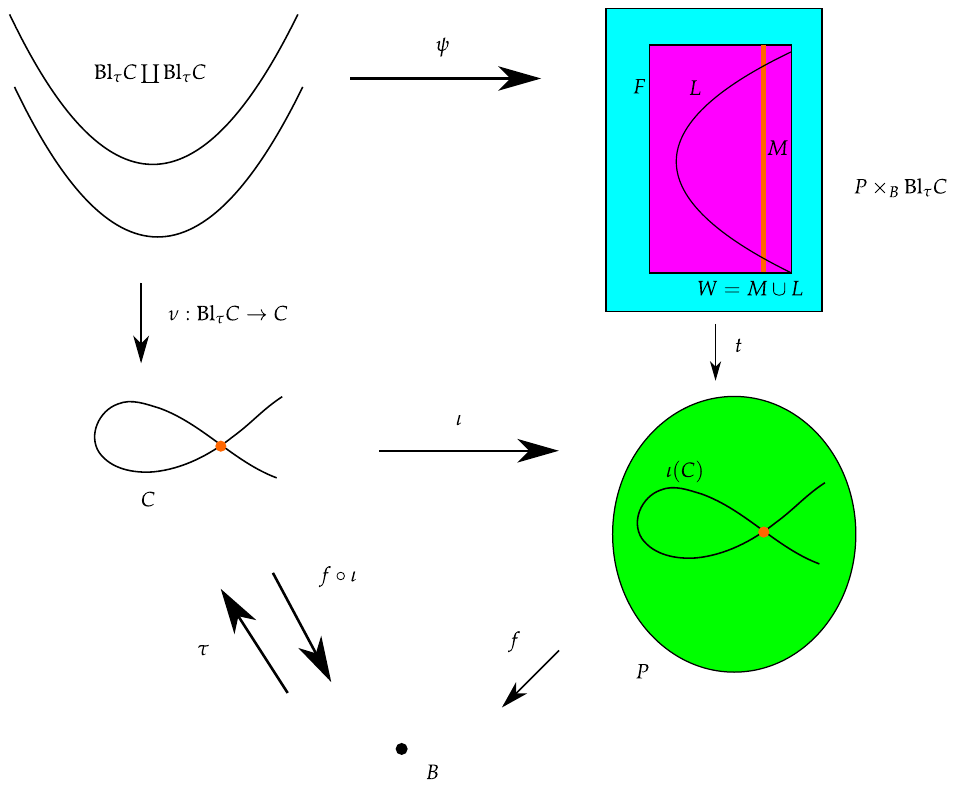}
	\caption{A visualization of some of the objects appearing in
		\autoref{notation:image-of-blow-up},
		\autoref{notation:s-t-maps}, and
		\autoref{notation:f-surface}.}
	\label{figure:image-blow-up}
\end{figure}

\begin{remark}
	\label{remark:}
	On fibers over $B$, we will see that $W$ is given as two copies of
	$\blow_\tau C$ glued along $E_\tau C$. Upon projecting to $P$,
	$\psi \circ i_1(\blow_\tau C)$ maps to $C$ while
	$\psi \circ i_2(\blow_\tau C)$ is contracted to the image of $\tau$.
We also note that $\psi$ is a closed embedding when restricted to either copy of $\blow_\tau C$, since the composition of $\psi$ with the first projection to $\blow_\tau C$ is an
isomorphism.
\end{remark}
\begin{remark}
	\label{remark:map-from-gluing}
	Note that $i_2(\blow_\tau C)$ is mapped to $P$ via the constant map at
	$\tau$, and so its intersection with $\psi \circ i_1(\blow_\tau C)$ is the preimage of the section $\tau$ along $\nu: \blow_\tau C \to C$.
However, by definition of the exceptional divisor of a blow up at $\tau$, this preimage is precisely $E_\tau C$.
Therefore, by the universal property of gluing along closed subschemes, the map $\blow_\tau C \coprod \blow_\tau C \to \blow_\tau C$
induces a map $\rho: \blow_\tau C \coprod_{E_\tau C} \blow_\tau C \to \blow_\tau C$.
Note that the cofiber product
$\blow_\tau C \coprod_{E_\tau C} \blow_\tau C$
is a scheme by 
\cite[\href{https://stacks.math.columbia.edu/tag/0E25}{Tag 0E25}]{stacks-project}.
\end{remark}

In order to construct the relative Hirzebruch surface containing $\blow_\tau C$, we will construct a map from $\blow_\tau C$ to a Grassmannian of lines.
To make this construction, we will need to invoke cohomology and base change.
In turn, to apply cohomology and base change, we will need to know $W \to B$ is
flat, which we check in \autoref{proposition:w-flat}.
We now present the unexpectedly tricky verification of this fact.
In order to verify flatness, we first need an alternate description of $W$ for which the following
generalization of the Chinese remainder theorem will be crucial.
The proof is an elementary exercise in abstract algebra which we omit.

\begin{lemma}
\label{lemma:chinese-remainder-generalization}
Let $R$ be a ring\footnote{By implicit assumption, ring means commutative
ring with unit.} and $I_1, I_2 \subset R$ be two ideals. The natural projection
		\begin{align*}
			\gamma: R/(I_1 \cap I_2) & \rightarrow R/I_1 \times_{R/(I_1 + I_2)} R/I_2\\
r +(I_1 \cap I_2) & \mapsto (r + I_1, r + I_2)
		\end{align*}
		is an isomorphism.
	\end{lemma}

We now present the formerly mentioned alternate description of $W$.

\begin{lemma}
	\label{lemma:image-description}
	With notation as in \autoref{notation:image-of-blow-up}, $W$ is isomorphic to two copies of $\blow_\tau C$ glued along the closed subscheme $E_\tau C$.
\end{lemma}
\begin{proof}
	Quasi-compactness of $\psi$ allows us to compute the scheme theoretic image of $W$ affine locally on $\blow_\tau C \times_B P$.	
	Since $\psi$ is proper and quasi-finite,
	$\psi$ is finite, hence affine.
	The preimage of $\spec R \subset \blow_\tau C \times_B P$ under $\psi$
	is then an affine open in $\blow_\tau
	C \coprod \blow_\tau C$ which we may write as 
$\spec A_1 \times A_2$ for $\spec A_1 \subset i_1(\blow_\tau C)$ and $\spec A_2
\subset i_2(\blow_\tau C)$. 
	The scheme theoretic image is then given locally over $\spec R$ as $\spec (R/\ker \psi^\sharp)$ for $\psi^\sharp : R \to A_1 \times A_2$.
	For $j \in \left\{ 1,2 \right\}$, let $i_j^\sharp : R \to A_j$ denote the maps induced by $i_j$ 
	and let $I_j := \ker i_j^\sharp$.

	We want to show $\spec R/(I_1 \cap I_2)$ is given as the restriction of $\blow_\tau C \coprod_{E_\tau C} \blow_\tau C$ to $\spec A_1 \times \spec A_2$.
	Observe that $\spec R/(I_1 + I_2)$ is the restriction of $E_\tau C$ to $\spec R$. 
	Indeed, since $\spec R/(I_1 + I_2) = \spec R/I_1 \cap \spec R/I_2$, this
	follows from the equality
	$\im \psi \circ i_1 \cap \im \psi \circ i_2 = E_\tau C$,
	which was explained in \autoref{remark:map-from-gluing}.
	Under the identifications, $R/I_1 \simeq A_1$ and $R/I_2 \simeq A_2$,
	the surjections $R/I_1 \to R/(I_1+I_2)$ and
	$R/I_2 \to R/(I_1+I_2)$ correspond to the two closed immersions
	$E_\tau C \to i_1(\blow_\tau C)$ and $E_\tau C \to i_2(\blow_\tau C)$.
	Then, $\blow_\tau C \coprod_{E_\tau C} \blow_\tau C$
	is given locally by the spectrum of $R/I_1 \times_{R/(I_1 + I_2)} R/I_2$.
	It then follows from \autoref{lemma:chinese-remainder-generalization} that the natural projection identifies 	
	$\spec R/(I_1 \cap I_2)$, the scheme theoretic image of $\phi$ restricted to $\spec R$, with the gluing $\spec R/I_1 \times_{R/(I_1 + I_2)} R/I_2$.
	Because the above constructions are compatible with localization, it follows that the natural map $\blow_\tau C \coprod_{E_\tau C} \blow_\tau C \to W$ is an isomorphism.
	\end{proof}

Using the above description of $W$, we are ready to verify it is flat.

\begin{proposition}
	\label{proposition:w-flat}
	With notation as in \autoref{notation:image-of-blow-up}, the natural map $\rho: W \to \blow_\tau C$, as defined in \autoref{remark:map-from-gluing} is locally
	free of degree $2$. Further, $W \to B$ is flat.
\end{proposition}
\begin{proof}
	By \autoref{lemma:image-description}, we have an isomorphism $W \simeq
	\blow_\tau C \coprod_{E_\tau C} \blow_\tau C$.
Since $\blow_\tau C \to B$ is flat by \autoref{proposition:genus-0-blow-up}, it
suffices to show $\rho : W \to \blow_\tau C$ is flat.
	We may work affine locally on $\blow_\tau C$, and hence assume it is of
	the form $\spec A$ and the preimage under $\rho$ is of the form 
$\spec (A \times_{A/f} A)$
	for $f$ a non-zerodivisor, as $E_\tau C$ is a Cartier divisor.

Hence, we have reduced to the elementary algebra exercise of showing that $A
\times_{A/(f)} A$ is a free $A$ module of rank $2$.
Indeed, the map
of $A$-modules $A \times A \to A \times_{A/(f)} A$ given by $(a,b) \mapsto (a,
a+fb)$
is an isomorphism, as may be verified using that $f$ is a non-zerodivisor.
\end{proof}

Having shown $W \to B$ is flat, we next construct the $(n-2)$-Hirzebruch twist.
Before embarking on the construction, we describe the idea.
Recall that $W \to B$ factors through $\blow_\tau C$, and the fibers of $W \to \blow_\tau C$ are degree $2$ subschemes of $\bp^{n-1}$.
There is then a unique line in $\bp^{n-1}$ spanned by this degree $2$ subscheme, and the union of these lines varying over points of $\blow_\tau C$
spans the desired Hirzebruch surface.
Due to the issue that $B$ may be non-reduced, we need to carry out this construction in families, as we now do.

\begin{notation}
	\label{notation:s-t-maps}
	Recall that $P$ is a projective bundle, as it is a Brauer-Severi scheme over $B$ with a section $\iota \circ \tau$,
	and so comes with an invertible sheaf $\sco_P(1)$.
Define the projections
\begin{equation}
	\label{equation:}
	\begin{tikzcd}
		\qquad & \blow_\tau C \times_B P \ar {ld}{s} \ar {rd}{t} & \\
		\blow_\tau C && P.
\end{tikzcd}\end{equation}
Let $\sci_{W/P \times_B \blow_\tau C}$ denote the ideal sheaf of $W$ in $P \times_B \blow_\tau C$.
\end{notation}
Pushing forward the ideal sheaf exact sequence twisted by $t^* \sco_P(1)$ along $s$ we obtain the exact sequence
\begin{equation}
	\label{equation:pushforward-containing-w-long}
	\begin{tikzcd}
		0 \ar {r} & s_* \left( \sci_{W/P \times_B \blow_\tau C} \otimes
		t^* \sco_P(1)\right) \ar {r} & s_* \left( t^* \sco_P(1)\right)
		\ar {r} & s_* \left( \sco_{W} \otimes t^*
		\sco_P(1)\right)\ar{r}& \qquad\\
		\qquad \ar{r} & R^1 s_* \left( \sci_{W/P \times_B \blow_\tau C} \otimes t^* \sco_P(1)\right) \ar {r} & R^1 s_* \left( t^* \sco_P(1)\right).
\end{tikzcd}\end{equation}
In order to define our desired family of lines, we need the following consequence of cohomology and base change.
We note that this will crucially use the flatness of $W$ over $\blow_\tau C$,
established in \autoref{proposition:w-flat}.
\begin{lemma}
	\label{lemma:map-to-grassmannian}
	We have $R^1 s_* \left( \sci_{W/P \times_B \blow_\tau C} \otimes t^* \sco_P(1)\right) = 0$ and the 
	three nonzero terms in the first line of \eqref{equation:pushforward-containing-w-long}
form a short exact sequence of locally free sheaves.
The first nonzero term has rank $n-2$, the second has rank $n$, and the third has rank $2$.
Further, for $f: P \to B$ the structure map,
we have a natural identification $s_* \left( t^* \sco_P(1)\right) \simeq \left(
f \circ \iota \circ \nu \right)^* \left( f_* \sco_P(1)\right).$
\end{lemma}
\begin{proof}
	Using flatness of $W \to \blow_\tau C$, as established in \autoref{proposition:w-flat}, the sheaf $\sco_W \otimes t^* \sco_P(1)$ is flat over $\blow_\tau C$. 
	Since $t^* \sco_P(1)$ is flat over $\blow_\tau C$,
	$\sci_{W/P \times_B \blow_\tau C} \otimes t^* \sco_P(1)$ is also flat over $\blow_\tau C$,
	so cohomology and base change applies to the above three sheaves.
	From cohomology and base change,
	$R^1 s_* \left( t^* \sco_P(1)\right)= 0$ so $s_* \left( t^*
	\sco_P(1)\right)$ commutes with base change and is locally free of rank
	$n$.

	Next, by \autoref{proposition:w-flat},
	$s_*\left( \sco_{W} \otimes t^* \sco_P(1) \right)$ 
	is locally free of rank $2$ and its formation commutes with arbitrary
	base change.

	The next step is to show 
	$s_* \left( t^* \sco_P(1)\right) \to s_* \left( \sco_{W} \otimes t^* \sco_P(1)\right)$ is surjective
	and that $R^1 s_* \left( \sci_{W/P \times_B \blow_\tau C} \otimes t^* \sco_P(1)\right) = 0$.
	Surjectivity may be checked on geometric points by cohomology and base
	change, and so follows from surjectivity of the restriction maps
	$H^0(\bp^{n-1}_x, \sco_{\bp^{n-1}_x}(1)) \to H^0(W_x, \sco_{\bp^{n-1}_x}(1)|_{W_x})$. 
	Then, $R^1 s_* \left( \sci_{W/P \times_B \blow_\tau C} \otimes t^* \sco_P(1)\right) = 0$
	because $R^1 s_* \left( t^* \sco_P(1)\right)= 0$, as verified above.

	Finally, the natural identification
	$s_* \left( t^* \sco_P(1)\right) \simeq \left( f \circ \iota \circ \nu \right)^* \left( f_* \sco_P(1)\right)$
	follows from flat base change applied to the sheaf $\sco_P(1)$ on $P$.
\end{proof}

We are now prepared to construct the sought $(n-2)$-Hirzebruch twist $F$. Here is the
definition.
\begin{notation}
	\label{notation:f-surface}
	Let $n \geq 3$, let $(B,f:P \to B,\iota: C \to P,\tau: B \to C) \in \singsstack n(B)$
	and retain notation from \autoref{notation:s-t-maps}.
	By \autoref{lemma:map-to-grassmannian}, the exact sequence 
	furnished by the first three nonzero terms of
\eqref{equation:pushforward-containing-w-long} yields a map $\omega: \blow_\tau C \to \mathbb G(1,P)$, the Grassmannian of lines in $P$.
The universal bundle over $\mathbb G(1,P)$ pulls back along $\omega$ to a relative family of lines $F$ with an embedding
$F \hookrightarrow \blow_\tau C \times_B P$,
where $F \simeq \bp \left(s_* \left( \sco_{W} \otimes t^* \sco_P(1)\right) \right)$
over $\blow_\tau C$.
\end{notation}

\begin{remark}
	\label{remark:map-from-blowup-to-F}
	Additionally, we can realize a map $i_W : W \hookrightarrow F$ via the surjection of sheaves
$\rho^* s_* \left( \sco_{W} \otimes t^* \sco_P(1)\right) = \rho^* \rho_* \left( (t^*\sco_P(1))|_W\right) \to (t^*\sco_P(1))|_W$ coming from the
natural adjunction.
\end{remark}

We will next verify that $F$ as defined in \autoref{notation:f-surface} is 
an $(n-2)$-Hirzebruch twist.
It will also be crucial to know that the exceptional divisor $E_\tau C$ is a Cartier divisor in $\blow_\tau C$,
as we now check.
\begin{lemma}
	\label{lemma:exceptional-divisor-has-degree-2}
	Retaining notation from \autoref{notation:image-of-blow-up},
	we have an isomorphism $L \cap M \simeq E_\tau C$, for $E_\tau C$ the
	exceptional divisor of the blow up $\nu$ of $C$ at $\tau$.
	Further, $E_\tau C$ is a degree $2$ relative effective Cartier divisor on $L$ over $B$.
\end{lemma}
\begin{proof}
	First, we show $E_\tau C$ is a relative effective Cartier divisor on $\blow_\tau C$.
	Since $E_\tau C$ is identified with the restriction of $L$ to $M$ by
	\autoref{lemma:image-description}, this will prove $L \cap M \simeq E_\tau C$.
	By \cite[\S8.2, Lemma 6]{BoschLR:Neron}, 
	it is enough to know $E_\tau$ is a effective Cartier divisor and remains such when restricted to each fiber over $B$.
	These claims follow from the universal property of blow ups and because
	the formation of the blow up at $\tau$ commutes with base change
	on $B$ by \autoref{proposition:genus-0-blow-up}.

	To conclude the proof, we need to check $E_\tau C$ has degree $2$. This may
	be verified on geometric fibers over points of $B$. 
	For example, one may deduce this from explicit computations in the nodal and cuspidal cases.
\end{proof}

We now show 
$F$ is an $(n-2)$-Hirzebruch twist.

\begin{proposition}
	\label{proposition:f-is-hirzebruch-twist}
	Suppose we are given a map $F \to \blow_\tau C \to B$, as in
	\autoref{notation:f-surface}.
	gives $F$ the structure of an $(n-2)$-Hirzebruch twist over $B$.
	\end{proposition}
\begin{proof}
	By \autoref{proposition:genus-0-blow-up}, $\blow_\tau C$ is a $1$-dimensional Brauer-Severi scheme over $B$
	and by definition $F$ is a relative dimension $1$ Zariski-locally trivial projective bundle over $\blow_\tau C$.
	Therefore, it only remains to show that $F$ is isomorphic to $\mathbb F_{n-2}$ fppf locally on $B$.

	Since $\blow_\tau C$ is a smooth 
	genus $0$ curve over $B$ with geometrically connected fibers, we may replace $B$ by a suitable fppf cover so as to
reduce to the case that $\blow_\tau C \simeq \bp^1_{B}$.
	Then, in \autoref{notation:f-surface}, we constructed
	$F$ as the projectivization of the rank $2$ vector bundle $s_* \left(
	\sco_{W} \otimes t^* \sco_P(1)\right)$ on $\bp^1_B$ which we wish to
	show is isomorphic to $\sco_{\bp^1_B} \oplus \sco_{\bp^1_B}(n-2)$ affine
	locally on $B$.
	
	By spreading out, we can reduce to the case $B$ is a local scheme, i.e., $B$ is the spectrum of a local ring.
	Next, recall the group of line bundles on $\bp^1_B$, for $B$ a local
	scheme, is $\bz$. A representative for the element $m \in \bz \simeq
	\pic(\bp^1_B)$ is given by $\sco_{\bp^1_B}(m)$ for
	$m \in \bz$,
	as may be deduced from
	the fibral isomorphism criterion \cite[17.9.5]{EGAIV.4}.

	We next claim there is an exact sequence 
	\begin{equation}
		\label{equation:vector-bundle-filtration}
		\begin{tikzcd}
			0 \ar {r} &  \sco_{\bp^1_B}(n-2) \ar {r} & s_* \left( \sco_{W} \otimes t^* \sco_P(1)\right)  \ar {r} & \sco_{\bp^1_B} \ar {r} & 0.
	\end{tikzcd}\end{equation}
	We will verify exactness of this sequence in
	\autoref{lemma:exactness-of-vector-bundle-sequence}.
	The result then follows once we verify extensions as in
	\eqref{equation:vector-bundle-filtration} split.
	Indeed, such extensions split because
	$\Ext^1_{\bp^1_B}(\sco_{\bp^1_B}, \sco_{\bp^1_B}(m)) = H^1(\bp^1_B,
	\sco_{\bp^1_B}(m)) = 0$; this vanishing may be deduced from an application of the Leray spectral
	sequence, using that $B$ is a local scheme.
\end{proof}

	\begin{lemma}
		\label{lemma:exactness-of-vector-bundle-sequence}
		The sequence \eqref{equation:vector-bundle-filtration} is exact.
		The surjection $s_* \left( \sco_{W} \otimes t^* \sco_P(1)\right) \to \sco_{\bp^1_B}$
		in \eqref{equation:vector-bundle-filtration} is identified with the restriction map 
	$r: s_* \left( \sco_{W} \otimes t^* \sco_P(1)\right) \to s_* \left( \sco_{M} \otimes t^* \sco_P(1)\right)$
	for $M$ as in \autoref{notation:image-of-blow-up}
	\end{lemma}
	\begin{proof}
		We start by constructing a surjection $s_* \left( \sco_{W} \otimes t^* \sco_P(1)\right)  \to \sco_{\bp^1_B}$.
	Indeed, recall from \autoref{notation:image-of-blow-up} that $W$ is constructed as the image of two copies of $\blow_\tau C \simeq \bp^1_B$ in $\bp^1_B \times_B P$.
	Note that $M$ maps to $P$ via the constant map through $\tau$, and therefore $\sco_M \otimes t^* \sco_P(1) \simeq \sco_M$.
	The restriction map $r: s_* \left( \sco_{W} \otimes t^* \sco_P(1)\right) \to s_* \left( \sco_{M} \otimes t^* \sco_P(1)\right) \simeq s_* (\sco_M) \simeq \sco_{\bp^1_B}$
	will be our desired surjection.

	It remains to show the restriction map $r$ is surjective with kernel isomorphic to $\sco_{\bp^1_B}(n-2)$.
	Let $\sci_{M/W}$ denote the ideal sheaf of $M$ in $W$. This is supported on $L$ and its restriction to $L$ is isomorphic to the invertible sheaf $\sco_L(-E_\tau C)$ on $L$,
	for $E_\tau C$ the exceptional divisor in $\blow_\tau C \simeq \bp^1_B$.
	Note this sheaf is invertible as $E_\tau C$ is a Cartier divisor by \autoref{lemma:exceptional-divisor-has-degree-2}.
	From this, we obtain
	an exact sequence
\begin{equation}
	\label{equation:}
	\begin{tikzcd}
	0 \ar {r} & s_* (\sco_L(-E_\tau C) \otimes t^* \sco_P(1)) \ar {r} & s_*
	(\sco_W \otimes t^* \sco_P(1)) \ar {r} & s_* (\sco_M \otimes t^*
	\sco_P(1)).
\end{tikzcd}\end{equation}
Note first that 
$R^1 s_* (\sco_L(-E_\tau C) \otimes t^* \sco_P(1)) = 0$ 
follows from cohomology and base change because $s$ maps $L$ isomorphically to
$\bp^1_B$.
Hence, the above sequence is right exact.

To conclude, it only remains to identify $s_* (\sco_L(-E_\tau C) \otimes t^* \sco_P(1)) \simeq \sco_L(n-2) \simeq \sco_{\bp^1_B}(n-2)$.
In fact, since $s$ is an isomorphism on $L$, it is enough to show
$s_* \sco_L(-E_\tau C)  \simeq \sco_{\bp^1_B}(-2)$ and 
$s_* t^* \sco_P(1)  \simeq \sco_{\bp^1_B}(n)$.
Since $B$ is a local scheme, $\pic(\bp^1_B) \simeq \bz$, we can verify the
above claims after restriction to any point of $B$.

We now complete the proof by computing the degrees of
$s_* \sco_L(-E_\tau C)$ and $s_* t^* \sco_P(1)$ in the case $B$
is the spectrum of a field $k$.
The former has degree $2$
by \autoref{lemma:exceptional-divisor-has-degree-2}.
The latter has degree $n$ because the composition
$\blow_\tau C \simeq L \to W \to \blow_\tau C \times_B P \to P$ agrees with
$\blow_\tau C \xra{\nu} C \xra{\iota} P$
and $\iota$ realizes $C$ as a degree $n$ curve in $P$
by definition of $\singshilb n$.
\end{proof}

Fix $n \geq 3$.
Using our above construction of the $(n-2)$-Hirzebruch twist,
there is a map $\singshilb n \to \smilestack n$, which we now describe.

\begin{construction}
	\label{construction:reverse-map}
Given $(B,f:P \to B,\iota: C \to P,\tau:B \to C) \in \singsstack n(B)$,
this data is mapped
to the point of $\smilestack n$ described by 
$(B, h: \blow_\tau C \to B, g: F \to \blow_\tau C, i : L \to F)$
for $g: F \to \blow_\tau C$ as in \autoref{notation:f-surface} and $L$ as in \autoref{notation:image-of-blow-up}.
The map $i$ is given as the composition of the maps
$\psi \circ i_1: \blow_\tau C \to W$ of \autoref{notation:image-of-blow-up} and $i_W : W \to F$ of \autoref{remark:map-from-blowup-to-F}.
\end{construction}
From \autoref{proposition:f-is-hirzebruch-twist}, we find $F \xra{g} \blow_\tau C \xra{h} B$ is an $(n-2)$-Hirzebruch twist,
and hence to construct the map 
$\singshilb n \to \smilestack n$
on objects, it suffices to verify that $L$ has class $e+nf$ on $F$.
This and more is established in the following lemma.

\begin{lemma}
	\label{lemma:class-L-and-M}
	Retain notation as in \autoref{notation:image-of-blow-up} and
	\autoref{notation:abuse-hirzebruch}.
	\begin{enumerate}
		\item The divisor $M \to F$ has class $e$.
		\item The divisor $L \to F$ has class $e+nf$.
		\item The invertible sheaf $t^* \sco_P(1)|_F$ on $F$ has class $e+(n-2)f$.
	\end{enumerate}
\end{lemma}
\begin{proof}
	We may work fppf locally on $B$ to assume
	$F \simeq \mathbb F_{n-2}$ over $B$ and $B$ is a local scheme.
	Because
	$\pic_{(\mathbb F_{n-2})_B/B} \simeq \ul{\bz}^2,$
	generated by the class of the fiber and exceptional divisor,
	it is enough to verify the above claims on geometric fibers, in which
	case they reduce to standard intersection theory calculations.

	We deduce the first part holds using that
	the directrix class on the Hirzebruch surface $\mathbb
	F_{n-2}$ over $\bp^1_B$ corresponds to the surjection $\sco_{\bp^1_B}(n)
	\oplus \sco_{\bp^1_B} \to \sco_{\bp^1_B}$.
	By
	\autoref{lemma:exactness-of-vector-bundle-sequence}, this is identified with the restriction of
	$s_* \left( \sco_{W} \otimes t^* \sco_P(1)\right)$ (the sheaf whose projectivization is $F$) to $M$.
	It follows that $M$ has class $e$ on $\mathbb F_{n-2}$.

	We may identify the class of $L$ by noting that
	$L \cap M$ has degree $2$ on $L$ by
	\autoref{lemma:exceptional-divisor-has-degree-2},
	and
	the intersection of $L$ with a
	fiber of projection $\rho \circ \psi \circ i_1: L \to W \to \blow_\tau C$ has degree $1$.

	Finally, to identify $t^* \sco_P(1)|_F$, 
	we claim
	this line bundle restricts to a degree $0$ divisor on $M$ and
	a degree $n$ divisor on $L$. 
	These claims can both be verified by choosing a section of
	$\sco_P(1)$ missing $\tau$. 
	The above claims imply
	$t^* \sco_P(1)|_F$ 
	has class $e+(n-2)f$ because this is the unique effective class whose intersection with $e$ is $0$ and whose intersection with $e + nf$ is $n$.
	\end{proof}

We are finally prepared to complete our construction of the map
$\singsstack n \to \smilestack n$.
\begin{lemma}
	\label{lemma:reverse-equivalence}
	There is a map of stacks $\singsstack n \to \smilestack n$
	sending $(B,f:P \to B,\iota: C \to P,\tau:B \to C) \in \singsstack n(B)$
	to the tuple $(B, h: \blow_\tau C \to B, g: F \to \blow_\tau C, i : L
	\to F) \in \smilestack n(B)$
	as defined in \autoref{construction:reverse-map}.
	\end{lemma}
\begin{proof}
	Observe that by \autoref{proposition:f-is-hirzebruch-twist}, $F \xra{g} \blow_\tau C \xra{h} B$ is an $(n-2)$-Hirzebruch twist.
	Further, $L$ has class $e+nf$ on $F$ by
	\autoref{lemma:class-L-and-M}.
	This constructs the desired map on objects.

	To show that morphisms in $\singsstack n$ are taken to morphisms in
	$\smilestack n$, it is enough to verify automorphisms are sent to
	automorphisms.
	The key input here is 
	\autoref{corollary:singsstack-representable},
	showing that
	any automorphism of 
$(B,f:P \to B,\iota: C \to P,\tau:B \to C) \in \singsstack n(B)$
is induced by an element $\phi \in \pgl_n(B) \simeq \aut_{P/B}(B)$ over $B$ 
such that $\phi(C) = C$ and $\phi(\tau(B)) = \tau(B)$.
This automorphism $\phi$ induces automorphisms of $\blow_\tau C, L,M$, and
therefore a compatible automorphism of $F$, which is the data 
of an automorphism of 
	$(B, h: \blow_\tau C \to B, g: F \to \blow_\tau C, i : L \to F) \in
	\smilestack n$.
\end{proof}

We have now defined maps $\Gamma : \smilestack n \to \singsstack n$ in \autoref{lemma:forward-equivalence}
$\Delta : \singsstack n \to \smilestack n$ in \autoref{lemma:reverse-equivalence}
which yield an equivalence of stacks.

\begin{theorem}
	\label{theorem:equivalence}
	For $n \geq 3$, 
	the maps 
$\Gamma : \smilestack n \to \singsstack n$
of \autoref{lemma:forward-equivalence}
and
$\Delta: \singsstack n \to \smilestack n$
of
\autoref{lemma:reverse-equivalence} define an equivalence of algebraic stacks.
\end{theorem}
\begin{proof}
	It is fairly straightforward to verify that the maps
	$\Delta$ and $\Gamma$ constructed in 
	\autoref{lemma:forward-equivalence}
	\autoref{lemma:reverse-equivalence}
	define an equivalence of stacks.
	One may do so by showing $\Delta \circ \Gamma$ and
	$\Gamma \circ \Delta$ are equivalent to the identity natural
	transformation on objects and both $\Delta$ and $\Gamma$ induce
	injective maps on isotropy group schemes at any point.
	Briefly, this follows from the fact that blowing up $C$ at
	$\tau$ is inverse to taking the image of the map from $F \to P$.
	For further details, see \cite[Theorem 3.1.31]{landesman:thesis}, and
	also \cite[Lemma 3.1.28 and
	Lemma 3.1.29]{landesman:thesis}.
\end{proof}

\section{The genus $1$ curve associated to a degree $2$ cover}
\label{section:degree-two}

The main goal of this section is to prove \autoref{theorem:fiber-bijection},
which associates to a degree $2$ cover a certain relative dimension $1$ group
scheme,
and describes 
$n$ coverings of that group scheme in terms of maps to $\smilestack n$.

Given a finite degree $2$ locally free cover $g: X \to B$, 
we now construct an associated genus $1$ curve $E_g \to B$ with $E_g^{\sm} \simeq g_* \bg_m/\bg_m$.
See \autoref{figure:associated-curve} for a visualization of $E_g$.
\begin{notation}
	\label{notation:singular-genus-1-construction}
Let $g: X \to B$ be a finite locally free degree $2$ cover.
The surjection $g^* g_* \sco_X \to \sco_X$ induces a map
$\iota : X \to \bp \left( g_* \sco_X \right)$ over $B$.
Additionally, the injective map of sheaves $\sco_B \to g_* \sco_X$ has
cokernel which is an invertible sheaf, as can be verified affine locally by
a direct computation on coordinate rings.
Therefore, the injection $\sco_B \to g_* \sco_X$  induces a map $\sigma: B \to
\bp (g_* \sco_X)$.

We define the genus $1$ curve $E_g$ associated to $g:X \to B$ as the cofiber
product 
\begin{equation}
	\label{equation:cofiber-product}
	\begin{tikzcd} 
		X \ar {r}{\iota} \ar {d}{g} & \bp \left( g_* \sco_X \right) \ar {d} \\
		B \ar {r}{\tau} & E_g
\end{tikzcd}\end{equation}
Note that $E_g$ exists as a scheme by
\cite[\href{https://stacks.math.columbia.edu/tag/0E25}{Tag
0E25}]{stacks-project}.
\end{notation}

\begin{figure}
	\centering
	\includegraphics[scale=.8, trim={0 23.5cm 0 1cm}]{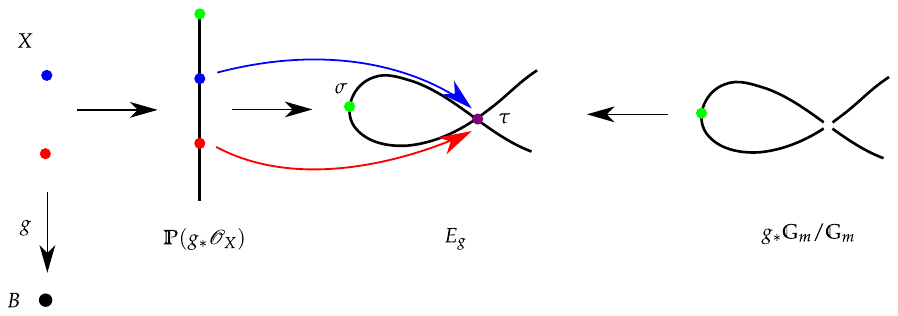}
	\caption{A visualization of the singular genus $1$ curve associated to a
		degree $2$ cover, as defined in
\autoref{notation:singular-genus-1-construction}.}
	\label{figure:associated-curve}
\end{figure}

Our upcoming goal is to show that $E_g$, together with $\sigma$ and $\tau$,
defines a point of $\singsw$, which we accomplish in
\autoref{lemma:generically-etale-to-nodal}.
As a first step, we will need to understand the interaction between the maps
$\iota$ and $\sigma$, so that we will be able to work affine locally away from
each of them.

\begin{lemma}
	\label{lemma:disjoint-directrix}
	With notation as in \autoref{notation:singular-genus-1-construction},
	The images $\iota(X)$ and $\sigma(B)$ are disjoint closed subschemes of
	$\bp \left( g_* \sco_X \right)$.
\end{lemma}
\begin{proof}
	To check $\iota(X)$ is disjoint from $\sigma(B)$, it is equivalent to check the preimage of $\sigma(B)$ in $X \times_B \bp(g_* \sco_X) = \bp (g^* g_* \sco_X)$
	is disjoint from the section $\iota_X : X \to \bp(g^* g_* \sco_X)$ induced by the surjection $g^* g_* \sco_X \to \sco_X$.
	The preimage of $\sigma(B)$ then corresponds to the exact sequence of vector bundles $g^* \sco_B \to g^* g_* \sco_X \to g^* g_* \sco_X/ g_* \sco_B$.
	The desired disjointness then amounts to showing the composition
	$g^* \sco_B \to g^* g_* \sco_X \to \sco_X$ is nonzero on every fiber.
	This non-vanishing holds because the composition above is adjoint to the composition
	$\sco_B \to g_* \sco_X \xra{\id} g_* \sco_X$ induced by the map $X \to B$. 
\end{proof}

We next verify some basic properties of the curve $E_g$.

\begin{lemma}
	\label{lemma:nodal-curve-is-proper}
	For $g: X \to B$ a finite locally free degree $2$ cover,
	$E_g$ is a proper flat finitely presented genus $1$ curve over $B$.
\end{lemma}
\begin{proof}
	By the explicit construction of $E_g$ as a cofiber product,
	we find $\bp \left( g_* \sco_X \right) \to E_g$ is surjective. Since $\bp\left( g_* \sco_X \right) \to B$ is proper,
	it follows $E_g \to B$ is proper as well.
	It follows that $E_g \to B$ is of finite type, being proper.

	To conclude, we need only check $E_g \to B$ is flat of finite presentation.
	To verify both of these properties, we may assume $B = \spec S$ and $X =
	\spec R$ are affine.
	By spreading out, 
	\cite[\href{https://stacks.math.columbia.edu/tag/01ZM}{Tag 01ZM}]{stacks-project},
	we may also assume $R$ and $S$ are finite type over
	$\spec \bz$.

	First, note that $E_g \to B$ is quasi-compact and quasi-separated.
	This follows from \autoref{lemma:disjoint-directrix}
	because the complements of $\tau(B)$ and $\sigma(B)$ are both affine and
	have affine intersection.

	We next show $E_g \to B$ has finite presentation.
	It is enough to show $E_g - \sigma(B)$ has finite presentation since, after possibly shrinking
	$B$, we have $E_g - \tau(B) \simeq \ba^1_B$.
	In terms of rings, $E_g- \sigma(B)$ is the spectrum of the fiber product of rings $S[x] \times_R S$.
	By Zariski localizing further, we may assume $R$ is a free $S$ module of
	rank $2$ with an inclusion $S \to R$ so that $R = S[x]/(f)$, for $f \in S[x]$ a degree $2$ polynomial with invertible leading coefficient on $S$.
	We can therefore identify
	$S[x] \times_R S$ as the subring of $S[x]$ consisting of those elements whose reduction mod $f$ has vanishing $x$ coefficient.
	Since $S[x] \times_R S \subset S[x]$ with $S[x]$ a noetherian ring (by
	our above reductions) that is finitely generated over $S[x] \times_R S$,
	we obtain $S[x] \times_R S$ is also noetherian. Hence,
	$S[x] \times_R S$ is of finite
	presentation over the noetherian $S$.

	It remains to show $S[x] \times_R S$ is flat over $S$.
	It is enough to give a free $S$ basis for $S[x] \times_R S$.
	We now construct this desired basis by viewing $S[x] \times_R S$ as a
	subring of $S[x]$.
	For each $n \geq 2$, note that $x^n$ can be written uniquely as $g + c_n x + d_n$ where $g \in (f)$ and $c_n, d_n \in S$,
	essentially via the Euclidean algorithm.
	It is enough to show that a free basis for $S[x] \times_R S$ is then
	given by $1 \in S \subset S[x]$ together with $x^n-c_n x$ for each $n \geq 2$.
	Indeed, this can be verified by starting with an arbitrary element of
	$S[x] \times_R S$ and algorithmically adding multiples of the above
	generators until we obtain the element $0 \in S$.
\end{proof}

Combining the above, it is fairly straightforward to check $E_g$ together with
$\sigma$ and $\tau$ defines a point of $\singsw$.
\begin{lemma}
	\label{lemma:generically-etale-to-nodal}
	The curve $E_g$ together with the sections $\sigma: B \to E_g$ and
	$\tau: B \to E_g$ define a point $(B, E_g \to B, \sigma: B \to E_g,
	\tau: B \to E_g) \in \singsw(B)$.
	Further, if $X \to B$ is \'etale, 
	$(B, E_g \to B, \sigma: B \to E_g, \tau: B \to E_g) \in \nodesw(B)$ and so $(B, E_g \to B, \sigma: B \to E_g) \in \nodew(B)$.
\end{lemma}
\begin{proof}
	We first check $(B, E_g\to B,\sigma: B \to E_g, \tau: B \to E_g) \in
	\singsw(B)$.
	By \autoref{lemma:disjoint-directrix}, we find $\sigma(B)$ lies in the
	smooth locus of $E_g \to B$.
	Geometric integrality of fibers holds because it may be checked on
	geometric fibers as the formation of the cofiber product is compatible with base change on $B$.
	Similarly, we may verify $\tau(B)$ lies in the singular locus by doing
	so
	geometric fibers. On geometric fibers, this may be deduced from the fact
	that the degree $2$ scheme $\iota: X \to \bp(g_* \sco_X)$ is the preimage of
	$\tau(B)$ in the normalization $\bp(g_* \sco_X)$ of $E_g$.
	The remaining properties were verified in \autoref{lemma:nodal-curve-is-proper}. 

	To conclude, we check that when $g: X \to B$ is \'etale,
	the point $(B, E_g\to B,\sigma: B \to E_g, \tau: B \to E_g) \in \singsw(B)$
	factors through the open substack $\nodesw \subset \singsw$.
	This may be verified on geometric fibers over $\spec k$ where \'etaleness of $X \to B$
	implies $X$ consists of two copies of $\spec k$. Therefore, the singularity
	at $\tau(\spec k)$ is obtained by gluing two copies of $\spec k$ in $\bp^1_k$,
	hence is a node.
\end{proof}

We next show $E_g^{\sm} \simeq g_* \bg_m/\bg_m$.
Recall that whenever $g: X \to B$ is a finite locally free morphism of schemes,
there is a norm map $g_* \sco_X \to \sco_B$ whose formation commutes with
arbitrary base change.
\cite[\href{https://stacks.math.columbia.edu/tag/0BD2}{Tag
0BD2}]{stacks-project}.
This can equivalently be described as a map
$\norm_{X/B}: g_* \ba^1 \to \ba^1$.
\begin{lemma}
	\label{lemma:smooth-locus-of-curve-from-cover}
	We have $E_g^{\sm} \simeq \bp(g_* \sco_X) - \iota(X)$. Further,
	$g_* \bg_m/\bg_m \simeq E_g^{\sm}$.
\end{lemma}
\begin{proof}
	First, we show
	$E_g^{\sm} \simeq \bp(g_* \sco_X) - \iota(X)$.
	Because $\bp(g_* \sco_X) - \iota(X) \to E_g$ is an open immersion, it is
	certainly contained in the smooth locus.
	To verify equality, it is enough to show the image of $\iota(X)$
	is contained in the singular locus, which was verified in
	\autoref{lemma:generically-etale-to-nodal}
	because $\iota(X)$ factors through $\tau(B) \subset E_g$.
	
	To verify 
	$g_* \bg_m/\bg_m \simeq E_g^{\sm}$,
	we will construct this isomorphism affine locally in a fashion compatible with localization so that it glues to give a global isomorphism.
	By Zariski localizing, we may assume $B = \spec S$ and $X = \spec R$,
	with $R$ of the form
	$R = S[x]/(x^2 + ax + b)$.
	Then, $g_* \bg_m$ can be explicitly identified with the open subscheme
	of $\ba^2_B \simeq g_* \ba^1_X$ given by
	those $sx + t \in R$ so that $\norm_{X/B}(sx+t) \neq 0$, as follows from
	the definition of the norm map.
	Computing this directly yields
	\begin{align*}
		\norm_{X/B}\left( sx +t \right) = \det \begin{pmatrix}
			t & -sb \\
			s & -sa+t
		\end{pmatrix}
		= t^2 -ast +bs^2.
	\end{align*}
	Therefore, $\norm_{X/B}(sx-t) = t^2 + ast + bs^2$ and so
	$g_* \bg_m$ is the complement of $t^2 + ast + bs^2$ in $\ba^2_{s,t}$.
	Hence, when we projectivize $\ba^2_{s,t}$, we find $g_*\bg_m/\bg_m$ is identified with the complement of $V(t^2 + ast + bs^2)$ in $\bp^1_{s,t} \simeq \bp (g_* \sco_X)$.

	We claim that the closed subscheme $V(t^2 + ast + bs^2) \subset \bp(g_*
	\sco_X)$ is precisely identified with the image $\iota(X)$.
	To see why this holds, we use that we have chosen the basis $\{1,x\}$ to trivialize $g_*
	\sco_X$.
	Under this basis,
	the image of the closed embedding $\spec R \to \bp(g_* \sco_X)$ is
	identified with the vanishing of the closed subscheme $V(x^2 + a \cdot 1
	\cdot x + b \cdot 1^2)$.
	Upon renaming $x$ as $t$ and $1$ as $s$, we obtain the claimed
	isomorphism $g_* \bg_m/\bg_m \simeq \bp(g_* \sco_X) - \iota(X)$.
\end{proof}

Combining the above discussion in this section with 
\autoref{proposition:equivalent-fiber-of-stack},
\autoref{lemma:lift},
and
\autoref{theorem:equivalence}
gives the following characterization of $n$-coverings of $g_* \bg_m/\bg_m$.

\begin{theorem}
	\label{theorem:fiber-bijection}
	Let $B$ be an integral normal scheme and let $n \geq 3$.
	Fix a degree $2$ locally free cover $g: X \to B$ which is generically \'etale.
	The composite of the bijection of
	\autoref{proposition:equivalent-fiber-of-stack} and
	\autoref{theorem:equivalence} yields a bijection 
	between elements of $H^1(B, g_* \bg_m/\bg_m \xra{\times n} g_*
	\bg_m/\bg_m)/\aut_{(E_g,\sigma)/B}(B)$
	and maps $B \to \smilestack n$
	which map to points $(B, E_g\to B, \sigma: B \to E_g) \in \singw(B)$
	that generically factor through $\nodew$.
	Further, this bijection respects automorphisms in the sense that it identifies automorphisms of the objects
	appearing in \autoref{proposition:equivalent-fiber-of-stack}
	with the $B$ points of the isotropy group of the corresponding map $B
	\to \smilestack n$.
\end{theorem}
\begin{proof}
	Using \autoref{theorem:equivalence}, 
	we have an equivalence
	$\smilestack n \simeq \singsstack n$.
	Composing this with the projection $\singsstack n \to \singsw \to \weierstrass$
	yields a bijection between maps $B \to \smilestack n$ over a given map
	$B \to \weierstrass$
	corresponding to a tuple $(B, E_g, \sigma: B \to E_g) \in \singw(B)$
	and maps $B \to \singsstack n$ over that same map $B \to \weierstrass$.

	When $B$ is integral,
	by \autoref{lemma:generically-etale-to-nodal}, we obtain that that
	the generic point $\eta$ of $B$ under the map $B \to \weierstrass$
	factors through $\nodesw \to \singsw \to \singw \to \weierstrass$.
	This means the generic fiber of the genus $1$ curve corresponding to
	the map $B \to \singsstack n$ is nodal.
	Normality of $B$ and \autoref{lemma:lift} then implies that all such lifts
	$B \to \singsstack n$
	are lifts of maps
	$B \to \onestack n$,
	compatibly with the projection to
	$\weierstrass$.

	By \autoref{lemma:smooth-locus-of-curve-from-cover}, 
	the smooth locus of $E_g$
	is isomorphic to $g_* \bg_m/\bg_m$, and so it follows from the
	equivalence of
	\autoref{proposition:equivalent-fiber-of-stack}\autoref{nstack-map} and
	\autoref{proposition:equivalent-fiber-of-stack}\autoref{hyper-torsor}
	that maps
	$B \to \onestack n$ mapping to $(B, E_g \to B,
	\sigma:B \to E_g) \in \weierstrass (B)$
	are in bijection with 
	elements of $H^1(B, g_* \bg_m/\bg_m \xra{\times n} g_*
\bg_m/\bg_m)/\aut_{(E_g, \sigma)/B}(B)$.

	The final statement regarding automorphisms follows from observing that each of the steps of the above
	bijection also preserve automorphism data, especially using that
	\autoref{theorem:equivalence} is an equivalence of stacks.
\end{proof}

\section{The unit resultant condition} 
\label{section:unit-resultant}

In this section, we prove
\autoref{theorem:hypercohomology-to-quotient-stack-bijection},
which gives a bijection between pairs $(q,\codirectrix) \in \affineSections n$
of unit resultant and points of $\smilestack n$.

\begin{notation}
	\label{notation:q-codirectrix}
	Keep notation as in \autoref{notation:hirzebruch}, letting $\mathbb F_{n-2}$
denote the Hirzebruch surface over a base scheme $B$.
For $\mathbb F_{n-2}\xra{g} \bp^1_B \xra{h} B$ the structure maps,
since $e$ corresponds to the directrix of the Hirzebruch surface,
$g_*(\sco_{\mathbb F_{n-2}}(e)) \simeq 
\sco_{\bp^1_B}(-n+2) \oplus \sco_{\bp^1_B}.$
Therefore,
\begin{align}
	\label{equation:projection-hirzebruch}
	g_* (\sco_{\mathbb F_{n-2}}(e + nf)) &\simeq g_* (\sco_{\mathbb
	F_{n-2}}(e)) \otimes \sco_{\bp^1_B}(n) \simeq  \sco_{\bp^1_B}(2) \oplus \sco_{\bp^1_B}(n).
\end{align}
Under the above isomorphism, the section $s \in H^0(\mathbb F_{n-2}, \sco_{\mathbb F_{n-2}}(e + nf))$
can be equivalently described as a pair $(q,\codirectrix)$ for $q \in
H^0(\bp^1_B,
\sco_{\bp^1_{B}}(2))$ and $\codirectrix \in H^0(\bp^1_B, \sco_{\bp^1_B}(n))$.
Given $(q,\codirectrix)$
corresponding to $s \in H^0(\mathbb F_{n-2}, \sco_{\mathbb F_{n-2}}(e + nf))$,
its vanishing locus defines a subscheme
$Z \subset \mathbb F_{n-2}$.
The complete linear system $\sco_{\mathbb F_{n-2}}(e + (n-2)f)$ determines
a map $\mathbb F_{n-2} \to \bp^{n-1}_B$ by
\autoref{lemma:locally-free-e-n-2-f}, and hence a map $\iota: Z \to
\bp^{n-1}_B$.
\end{notation}
We now connect smoothness of $Z$ to the condition that the resultant is a unit.
\begin{lemma}
	\label{lemma:resultant-1-smooth}
	Keeping notation as in \autoref{notation:q-codirectrix},
	a subscheme $W \subset \mathbb F_{n-2}$ determined by a section
	$(q,\codirectrix) 
	\in H^0\left(B, \sco_{\bp^1_B}(2) \oplus
	\sco_{\bp^1_B}(n)\right)$
	not vanishing on any fibers of $h \circ g$
	is smooth if and only the resultant $\res(q,\codirectrix)$ lies in
	$\bg_m(B)$.
\end{lemma}
\begin{proof}
	Because $(q,\codirectrix)$ comes from a fixed linear system and is
	not zero on any fiber, its vanishing locus is
	flat and locally of finite presentation, so it suffices to show $W$ is
	smooth over every point $p$ of $B$ if and only if $\res(q,\codirectrix) \in
	\bg_m(B)$.
	Let $\ol{q}$ and $\ol{\codirectrix}$ denote the restrictions of $q$
	and $\codirectrix$ in the residue field $\kappa(p)$.

	We will show $W_p$ is smooth over $\spec \kappa(p)$
	if and only if $\res(\ol{q}, \ol{\codirectrix}) \neq 0 \in \kappa(p)$.
	Because $W_p$ has class $e + (n-2)f$ and is nonzero,
	the map $g|_{W_p}: W_p \to \bp^1_p$ is generically an isomorphism.
	Further, $g|_{W_p}$ is an isomorphism if and only if $W_p$ contains no fibers of $g$.
	Therefore, when $W_p$ contains no fibers of $g$, $W_p$ will necessarily
	be smooth, and conversely if $W_p$ does contain a fiber of $g$ it will
	be singular at the point of intersection of that fiber with another
	component of $W_p$.
	Finally, $W_p$ contains a fiber over some point $t \in \bp^1_p$ if
	and only if 
	$\res(\ol{q},\ol{\codirectrix}) = 0 \in \kappa(p)$
	because both conditions are equivalent to the simultaneous vanishing of $\ol{q}$ and
	$\ol{\codirectrix}$ at $t$.
\end{proof}

Motivated by \autoref{lemma:resultant-1-smooth},
we now define the subscheme of $\affineSections n$ corresponding to the locus where
the resultant is a unit.
\begin{definition}
	\label{definition:unit-resultant}
	There is a resultant map $\res: \affineSections n \to \ba^1_{\spec \bz}$ sending
	$(q,\codirectrix) \mapsto \res(q,\codirectrix)$.
	Viewing $\bg_m \subset \ba^1_{\spec \bz}$ as the complement of the origin,
	define the open subscheme $\unitResultant n := \res^{-1}(\bg_m) \subset
	\affineSections n$.
\end{definition}

We are nearly ready to prove our main result, but first we state two
preparatory lemmas, which relate various quotient stacks.

\begin{lemma}
	\label{lemma:hirzebruch-quotient-commutes-with-z-points}
	Keep notation as in \autoref{notation:q-codirectrix}. 	There is an injective map 
	$\unitResultant n(B)/\affineAut n(B) \hookrightarrow
[\unitResultant n/\affineAut n](B)$
which is a bijection if $H^1(B, \pgl_2) = H^1(B, \bg_m) = H^1(B, \bg_a) =0$.
\end{lemma}
\begin{proof}
	We have a sequence of pointed sets
	\begin{equation}
		\label{equation:resultant-orbit-sequence}
		\begin{tikzcd}
			0 \ar {r} &  H^0(B, \affineAut n) \ar {r} & H^0(B,
			\unitResultant n) \ar {r} & H^0(B, [\unitResultant
					n/\affineAut n]) \ar {r}
	& H^1(B, \affineAut n).
	\end{tikzcd}\end{equation}
	This implies the injectivity claim. 

	For the final statement, it is enough to check 
	$H^1(B, \affineAut n) = 0$ when 
	$H^1(B, \pgl_2) = H^1(B, \bg_m) = H^1(B, \bg_a)=0$.
	This holds by
	\autoref{lemma:affine-aut-sequences},
	which shows 
	$\affineAut n$
	is an iterated extension of $\pgl_2$ by copies of $\bg_m$ and $\bg_a$.
\end{proof}

\begin{lemma}
	\label{lemma:unit-resultant-equivalence}
	For $n \geq 3$, we have equivalences of stacks
	\begin{equation}
	[\unitResultant n/\affineAut n] \simeq [\smile n/\aut_{\mathbb F_{n-2}/\bz}]
	\simeq \smilestack n \simeq \singsstack n \simeq [\singshilb n/\pgl_n]
	\end{equation}
\begin{align*}
\end{align*}

\end{lemma}
\begin{proof}
	We have $\smile n \simeq [\unitResultant n/\bg_m]$ 
	using \autoref{lemma:resultant-1-smooth}.
	We also have an isomorphism
	and $\aut_{\mathbb F_{n-2}/\bz} \simeq \projAut n \simeq [\affineAut
	n/\bg_m],$
	using \autoref{lemma:affine-aut-sequences}(4).
	Together, these yield the equivalence 
	$[\unitResultant n/\affineAut n] \simeq [\smile n/\aut_{\mathbb
	F_{n-2}/\bz}]$.

	The remaining identifications follow from
	\autoref{lemma:smile-stack-representable}, 
	\autoref{theorem:equivalence}, and
	\autoref{corollary:singsstack-representable}.
	which respectively yield the isomorphisms $[\smile n/\aut_{\mathbb F_{n-2}/\bz}]
	\simeq \smilestack n \simeq \singsstack n \simeq [\singshilb n/\pgl_n]$.
\end{proof}

Our above results now combine to give descriptions of elements $(q,\codirectrix)$ with unit resultant
as points of various stacks.
Given a quadratic form $q \in H^0(\bp^1_B, \sco_{\bp^1_B}(2))$, we use $V(q)$ to denote the subscheme
of $\bp^1_B$ defined by $q$.
Recall that given a locally free degree $2$ cover $X \to B$, we use $E_g$ with
section $\sigma$ to denote the corresponding genus $1$ curve as in
\autoref{notation:singular-genus-1-construction}.

Additionally, note there is a map from $\unitResultant n$ to the stack of degree
$2$ finite locally free covers which sends $(q,\codirectrix)$ to $V(q)$.
Since this map is invariant for the action of $\affineAut n$, we obtain a map
$\Pi_n$ from $\left[ \unitResultant n/\affineAut n \right]$ to the stack of
degree $2$ finite locally free covers.

\begin{theorem}
	\label{theorem:hypercohomology-to-quotient-stack-bijection}
	Let $B$ be a normal integral scheme and $n \in \bz_{\geq 3}$ be an integer.
	Fix a degree $2$ locally free generically \'etale cover $g: X \to B$.
	There is an injection from orbits $(q,\codirectrix) \in 
	\unitResultant n(B)/\affineAut n(B)$ such that
	$V(q) \simeq X$
	to
	$\Pi_n^{-1}([X]) \subset [\unitResultant n/\affineAut n](B)$.
	In turn, $\Pi_n^{-1}([X])$
	is identified bijectively with
	elements of $H^1(B, g_* \bg_m/\bg_m \xra{\times n} g_*
	\bg_m/\bg_m)/\aut_{(E_g,\sigma)/B}(B)$.
	The above injection is a bijection if $H^1(B, \pgl_2) = H^1(B, \bg_m) =
	H^1(B, \bg_a) =0$.

	Further, the above injection identifies the following three groups:
	\begin{enumerate}
		\item the $B$ points of the isotropy group scheme of the
			corresponding map $B \to [\nhilb n/\pgl_n]$
	via the bijection of \autoref{proposition:equivalent-fiber-of-stack};
\item the $B$ points of the isotropy group scheme of the
	corresponding map $B \to \smilestack n$;
\item the stabilizer in $\affineAut n(B)$ of $(q,\codirectrix)$.
	\end{enumerate}
	\end{theorem}

\begin{proof}
	To prove the first part, we wish to produce an injection from 
	orbits $(q, \codirectrix) \in \unitResultant n(B)/\affineAut n(B)$
	to points
	$B \to \smilestack n$ over 
	$(B, E_g \to B, \sigma: B \to E_g, \tau: B \to E_g) \in \singw(B)$
	which is a bijection if $H^1(B, \pgl_2) = H^1(B, \bg_m) = H^1(B, \bg_a) = 0$.

	We claim there is a sequence of maps
	\begin{align*}
		\unitResultant n(B)/\affineAut n(B) \to [\unitResultant
		n/\affineAut n](B) \to [\smile n / \aut_{\mathbb F_{n-2}/B}](B)
		\to \smilestack n(B)
	\end{align*}
	where all but the first maps are bijections, and the first map is an
	injection which is a bijection if 
	$H^1(B, \pgl_2) = H^1(B, \bg_m) = H^1(B, \bg_a) = 0$.
	Indeed, the statement for the first map is
	\autoref{lemma:hirzebruch-quotient-commutes-with-z-points}.
	The remaining maps are bijections by
	\autoref{lemma:unit-resultant-equivalence}.

	We next check that the above identifications send the set of $(q,\codirectrix) \in
	\unitResultant n(B)/\affineAut n(B)$ with $V(q) \simeq X$ to the set
	$H^1(B, g_* \bg_m/\bg_m \xra{\times n} g_* \bg_m/\bg_m)/\aut_{(E_g,\sigma)/B}(B)$.
	Keeping notation as in
	\autoref{notation:singular-genus-1-construction},
	by \autoref{theorem:fiber-bijection}, there is a bijection between
	maps
	$B \to \smilestack n$ lying over
	$(B, E_g \to B, \sigma: B \to E_g, \tau: B \to E_g) \in \singw(B)$
	and elements of $H^1(B, g_* \bg_m/\bg_m \xra{\times n} g_* \bg_m/\bg_m)/\aut_{(E_g,\sigma)/B}(B)$.
	Because the above maps of stacks are compatible with the projection
	to $\singw$, to give the desired identification,
	it is enough to show that we can recover $V(q)$ from $E_g$.
	But indeed, 
	using that $B$ is normal, the normalization of $E_g$
	is $\bp(g_* \sco_X)$, as this is normal and maps birationally to $E_g$.
	Then, $V(q)$ can be recovered as the preimage of the singular locus
	of $E_g \to B$ under the map $\bp(g_* \sco_X) \to E_g$.

	Combining the above identifications, we then obtain that pairs
	$(q, \codirectrix) \in \unitResultant n(B)/\affineAut n(B)$ with $V(q)
	\simeq X$ inject into 
	$H^1(B, g_* \bg_m/\bg_m \xra{\times n} g_*
	\bg_m/\bg_m)/\aut_{(E_g,\sigma)/B}(B)$.
	Further, we obtain this injection is an isomorphism if 
	$H^1(B, \pgl_2) = H^1(B, \bg_m) = H^1(B, \bg_a) = 0$.

	To conclude, we wish to identify the $B$-points of the isotropy group scheme of
	$B \to [\nhilb n/\pgl_n]$,
	the $B$ points of the isotropy group scheme of $B \to \smilestack n$
	and the stabilizer of $(q,\codirectrix)$ in $\affineAut n(B)$.
	The identification of the first two follows from
	\autoref{lemma:unit-resultant-equivalence}.
	The final identification also follows from
	\autoref{lemma:unit-resultant-equivalence}
	because the the stabilizer of a point $(q,\codirectrix) \in \unitResultant n$
	in $\affineAut n(B)$
	is given by the $B$ points of the isotropy group associated to the map $B \to
	[\unitResultant n/\affineAut n]$.
\end{proof}

\section{Cohomological Comparison}
\label{section:cohomology}

We next give a concrete description of the groups
$H^1(\spec \bz, g_*\bg_m/\bg_m \xrightarrow{\times n} g_*\bg_m/\bg_m)$
appearing in \autoref{theorem:hypercohomology-to-quotient-stack-bijection} in
the case $B = \spec \bz$,
which relates them to $n$-torsion in class groups.
To state this comparison in \autoref{lemma:comparison-selmer}, we review the
notion of the $n$-Selmer group of a number field.

\begin{remark}
	\label{remark:n-selmer-group-of-number-field}
	Recall that the $n$-Selmer group of a number field $K$ is defined as 
	\begin{align*}
		\sel_n(K) := \left\{ \alpha \in
		K^\times : \text{ there exists a fractional ideal }\mathfrak a
	\text{ of $K$ with }\left( \alpha \right) = \mathfrak a^n
\right\}/\left( K^\times \right)^n.
	\end{align*}
	Let $X := \spec \sco_K$.
	In fact, $H^1(X, \mu_n) \simeq \sel_n(K)$ as we now explain.
	The restriction map $H^1(X, \mu_n) \to H^1(\spec K, \mu_n)$ has image
	landing inside $\sel_n(K)$. This gives a natural map
	$H^1(X, \mu_n) \to \sel_n(K)$ which is injective because $X$ is normal
	and is surjective because both $H^1(X, \mu_n)$ and $\sel_n(K)$ have the
	same order as 
	they are extensions of $\cl(K)[n]$ by $\sco_K^\times / \left( \sco_K^\times
	\right)^n$.
\end{remark}

\begin{lemma}
	\label{lemma:comparison-selmer}
	Let $n \in \bz_{\geq 1}$ and let $g: X \to B$ be a connected finite locally free
	cover of arbitrary degree.
There is a natural map
$\rho: H^1(X, \mu_n) \to H^1(B, g_* \bg_m/\bg_m \xrightarrow{\times n} g_*
\bg_m/\bg_m)$.
When $B = \spec \bz$, $\rho$
is identified with the quotient $H^1(X, \mu_n) \to \coker(H^1(\spec \bz, \mu_n)
\to H^1(X, \mu_n))$.
Maintaining $B = \spec \bz$,
it follows that 
$H^1(B, g_* \bg_m/\bg_m	\xrightarrow{\times n} g_* \bg_m/\bg_m)$
is an extension of $H^1(X, \bg_m)[n] = \cl(X)[n]$ by $(\sco_X^\times(X)/ n
\sco_X^\times(X)) / \pm 1$.
In particular, if $X = \spec \sco_K$ for $K$ a number field,
$H^1(B, g_* \bg_m/\bg_m	\xrightarrow{\times n} g_* \bg_m/\bg_m)
\simeq \coker(\sel_n(\bq) \to \sel_n(K))$.
\end{lemma}
\begin{proof}
	The distinguished triangle associated to the three complexes
	\begin{equation}
		\label{equation:}
		\begin{tikzcd}
			\bg_m \ar {r} \ar {d}{\times n} & g_* \bg_m \ar {r} \ar
			{d}{\times n} & g_*\bg_m/\bg_m \ar
			{d}{\times n} \\
			\bg_m \ar {r} & g_* \bg_m \ar {r} & g_* \bg_m/\bg_m
	\end{tikzcd}\end{equation}
	on $B$
	gives rise to an exact sequence on hypercohomology
\begin{equation}
	\label{equation:}
	\begin{tikzcd}
		& H^1(B, \bg_m \xrightarrow{\times n} \bg_m)  \ar {r}{\alpha} & H^1(B,
		g_*\bg_m \xrightarrow{\times n} g_*\bg_m) \ar {r} & \qquad \\
		\ar{r} & H^1(B,	g_*\bg_m/\bg_m \xrightarrow{\times n} g_*\bg_m/\bg_m) \ar {r} &
		H^2(B, \bg_m \xrightarrow{\times n} \bg_m).
\end{tikzcd}\end{equation}
Using the distinguished triangle
\eqref{equation:first-distinguished-triangle}
and the vanishing $H^1(B, \bg_m) = H^2(B, \bg_m) = 0$,
we find $H^2(B, \bg_m \xrightarrow{\times n} \bg_m)$.
We also obtain isomorphisms
$H^1(B, \mu_n) \simeq H^1(\bg_m \xrightarrow{\times n} \bg_m)$
and
$H^1(X, \mu_n) \simeq H^1(X, \bg_m \xrightarrow{\times n} \bg_m) \simeq
H^1(B,g_*\bg_m \xrightarrow{\times n} g_*\bg_m)$.
This identifies
$H^1(X, \mu_n) \to H^1(B, g_* \bg_m/\bg_m \xrightarrow{\times n} g_* \bg_m/\bg_m)$
with
the quotient $H^1(X, \mu_n) \to \coker(H^1(\spec \bz, \mu_n)
\to H^1(X, \mu_n))$.
The description of 
$H^1(B, g_* \bg_m/\bg_m	\xrightarrow{\times n} g_* \bg_m/\bg_m)$
as an extension then follows from the analogous description of
$H^1(X, \mu_n)$ as an extension of
$H^1(X, \bg_m)[n]$ by $\sco_X^\times(X) / n \sco_X^\times(X)$.

If $X = \spec \sco_K$, the isomorphism
$H^1(B, g_* \bg_m/\bg_m	\xrightarrow{\times n} g_* \bg_m/\bg_m)
\simeq \coker(\sel_n(\bq) \to \sel_n(K))$
follows from 
\autoref{remark:n-selmer-group-of-number-field}
together with the identification
$H^1(B, g_* \bg_m/\bg_m	\xrightarrow{\times n} g_* \bg_m/\bg_m)
\simeq \coker(H^1(\spec \bz, \mu_n) \to H^1(X, \mu_n))$.
\end{proof}

\subsection{Examples}
\label{subsection:examples}
The utility of
\autoref{lemma:comparison-selmer}
is that it lets us compute
$H^1(B, g_* \bg_m/\bg_m
\xrightarrow{\times n} g_* \bg_m/\bg_m)$.
We next illustrate the usefulness of \autoref{lemma:comparison-selmer}
and
\autoref{theorem:hypercohomology-to-quotient-stack-bijection}
with some concrete examples. 

\begin{remark}
	\label{remark:example-setup}
	Fix $g: X \to \spec \bz$ a normal integral finite degree $2$ locally free cover of discriminant $d$.
For each primitive quadratic form $q$ of discriminant $d$
we can
ask how many $\affineAut n$ orbits of degree $n$ polynomials $\codirectrix$
there are with $\res(q,\codirectrix) = \pm 1$, up to the action of
$\affineAut n$.
By combining \autoref{theorem:hypercohomology-to-quotient-stack-bijection} and
\autoref{theorem:resultant}
with
\autoref{lemma:comparison-selmer},
we can see there are no such values of $\codirectrix$ unless
$q$ lies in $H^1(X, \bg_m)[n]$. In the latter case,
the number of such orbits can be computed by
\autoref{lemma:comparison-selmer}.
\end{remark}
We have verified \autoref{remark:example-setup} using MAGMA in thousands of cases.
Let us now see this carried out in some examples.

\begin{example}
	\label{example:}
	Consider $K = \bq(\sqrt{-23})$ and $X = \spec \sco_K = \spec
\bz[\frac{\sqrt{-23} + 1}{2}]$.
We apply \autoref{remark:example-setup} to this setting.
	We have $\cl(K) \simeq \bz/3\bz$, with representatives for the
	three quadratic forms given by
	$q_1 := x^2 + xy + 6y^2, q_2 := 2x^2 + xy + 3y^2, q_3 := 2x^2 - xy +3y^2$.
	Since these are all $3$-torsion, we expect that for each quadratic form,
	there should exist
	some degree $3$ polynomial $\codirectrix$ with
	$\res(q_i, \codirectrix) = \pm 1$.
	Indeed,
	we see
	$\res(q_1, -y^3) = 1, \res(q_2, -x^3 -xy^2+y^3) = 1$, and $\res(q_3,-x^3-xy^2-y^3) = 1$.
\end{example}
\begin{example}
	\label{example:}
	Consider $K = \bq(\sqrt{-47})$ and $X = \spec \sco_K = \spec
	\bz[\frac{\sqrt{-47}+1}{2}]$.
	Then, $\cl(K) \simeq \bz/5\bz$ with representatives
	given by the quadratic forms
	$q_1 = x^2 + xy + 12y^2, q_2 = 2x^2 + xy + 6y^2, q_3 = 2x^2 - xy + 6y^2,
	q_4 = 3x^2 +xy + 4y^2,$ and $q_5 = 3x^2 -xy + 4y^2$.
	For each $1 \leq i\leq 5$, there is a degree $5$ polynomial $\codirectrix_i$ so
	that
	$\res(q_i,\codirectrix_i) = 1$.
	Namely,
	$\codirectrix_1 = -y^5,
	\codirectrix_2 = -x^5 -3x^3y^2+x^2y^3-xy^4-y^5,
	\codirectrix_3 = -x^5-3x^3y^2-x^2y^3-xy^4+y^5,
	\codirectrix_4 = -x^5-x^4y-x^3y^2+xy^4+y^5,
	\codirectrix_5 = -x^5+x^4y-x^3y^2+xy^4-y^5.$
\end{example}

\begin{example}
	\label{example:}
	When $K = \bq(\sqrt{-1})$ and $X = \spec \sco_K = \spec
	\bz[\sqrt{-1}]$,
	the unique equivalence class of quadratic forms has representative
	$x^2 + y^2$.
	Because $H^1(X, \mu_n)/ H^1(\spec \bz, \mu_n)$ has size $2$ for any $n$,
\autoref{lemma:comparison-selmer} predicts that for any
	positive integer $n$, there
	should be two orbits of pairs $(q,\codirectrix)$ for $\codirectrix$
	of degree $2n$ with
	resultant $1$.
Indeed, the two orbits correspond to $\codirectrix = x^{2n}$ and
$\codirectrix = (xy)^n$.
\end{example}

\begin{example}
	\label{example:cube-roots-of-unity}
	In the case $K = \bq(\sqrt{-3})$ and $g: X = \spec \sco_K = \spec
	\bz[\frac{1 + \sqrt{-3}}{2}] \to \spec \bz$, $\cl(K)$ is the trivial group, but
\autoref{lemma:comparison-selmer} predicts there
		should be two orbits of pairs $(q,\codirectrix)$ with
		resultant $1$.
		Note that $\# H^1(X, g_*\bg_m/\bg_m \xrightarrow{\times 3} g_*
		\bg_m/\bg_m) = 3$, but the quotient of this group
		by $\pm 1$ has size $2$.
		We can take $q := x^2 + xy + y^2$ as a representative for the
		quadratic form.
		In this case, the two polynomials $y^3$ and $y^2x$ have
		resultant $1$ with $q$, and lie in distinct $\affineAut 3$
		orbits.
		Note that $y^3 + y^2x$ is another such polynomial which
		corresponds to the third element of 
		$H^1(X, g_*\bg_m/\bg_m \xrightarrow{\times 3} g_*\bg_m/\bg_m)$
		that is identified with $y^2x$ under the automorphism of $g$.
\end{example}

\begin{example}
	\label{example:}
	Let $K = \bq(\sqrt{5})$ and $g: X = \spec \sco_K = \spec
	\bz[\frac{\sqrt{5} + 1}{2}] \to \spec \bz,$ which has discriminant $20$.
	Then, $\cl(K)= 1$, and a representative is given by the quadratic form
	is given by $q = x^2 - 5y^2$.
	By \autoref{lemma:comparison-selmer}, there
	are $2 = \# (\sco_K^\times/3\sco_K^\times)/\aut_{X/\spec
	\bz}(\spec\bz)$ 
	unit resultant
	$\affineAut n$ orbits.
	Representatives for the two orbits are given by
	$\codirectrix_1 = y^3$ and $\codirectrix_2 = -4xy^2-9y^3$.
	\end{example}

\bibliographystyle{alpha}
\bibliography{/home/aaron/Dropbox/master}
\end{document}